\documentclass[reqno]{amsart}

\usepackage[letterpaper, hmarginratio=1:1]{geometry}
\usepackage[hidelinks]{hyperref}                           

\usepackage{color}
\usepackage{bbm,bm}
\usepackage{stmaryrd}
\usepackage{xcolor}
\usepackage{enumerate}
\usepackage{amssymb}
\usepackage{amsmath, amsthm}
\usepackage{cleveref}
\usepackage[mathscr]{euscript}
\usepackage{tikz}
\usepackage{pgfplots}
\usetikzlibrary{calc}

\usepackage[
backend=bibtex,
style=alphabetic,
giveninits=true,
url = false,
doi = false,
isbn =false,
maxbibnames=99,
maxalphanames=99,
]{biblatex}
\renewbibmacro{in:}{}
\DeclareFieldFormat{title}{#1} 
\DeclareFieldFormat[article]{title}{#1}
\DeclareFieldFormat[inbook]{title}{#1}
\DeclareFieldFormat[incollection]{title}{#1}
\AtBeginBibliography{\small}

\newtheorem{ccounter}{ccounter}[section]

\newtheorem{theorem}[ccounter]{Theorem}

\newtheorem{lemma}[ccounter]{Lemma}
\newtheorem{corollary}[ccounter]{Corollary}

\newtheorem{proposition}[ccounter]{Proposition}
\theoremstyle{definition}
\newtheorem{definition}[ccounter]{Definition}
\newtheorem{remark}[ccounter]{Remark}

\newcommand\bbD{\mathbb{D}}
\newcommand\E{\mathbb{E}}

\newcommand\N{\mathbb{N}}

\renewcommand\P{\mathbb{P}}
\newcommand\R{\mathbb{R}}
\newcommand\bfR{\boldsymbol{R}}
\newcommand\bbR{\mathbb{R}}

\newcommand\bbH{\mathbb{H}}
\renewcommand\L{\mathbf{L}}

\newcommand\bbT{\mathbb{T}}
\newcommand\V{\boldsymbol{V}}

\newcommand{\bbV}{\mathbb{V}}

\newcommand{\Zplus}{\mathbb{Z}_{\ge 1}}

\newcommand{\fourthirds}{1 - 1/\L}

\newcommand\be{\begin{equation}}
\newcommand\ee{\end{equation}}

\newcommand{\eps}{\varepsilon}
\renewcommand{\epsilon}{\varepsilon}
\newcommand{\iu}{\mathrm{i}}

\renewcommand{\log}{\ln}

\newcommand{\unn}[2]{[\![#1,#2]\!]}

\renewcommand{\phi}{\varphi}

\def\beq{\begin{equation}}
\def\eeq{\end{equation}}

\numberwithin{equation}{section}
\renewcommand{\L}{\mathfrak L}
\newcommand{\s}{\mathfrak s}

\newcommand{\bex}{\begin{equation*}}
\newcommand{\eex}{\end{equation*}}

\newcommand{\interior}{\operatorname{int}}

\renewcommand{\Re}{\operatorname{Re}}
\renewcommand{\Im}{\operatorname{Im}}
\newcommand{\Imaginary}{\operatorname{Im}}

\newcommand{\one}{\mathbbm{1}}
\newcommand{\hamiltonian}{\boldsymbol{H}_{K,t}}
\newcommand{\hamintro}{\boldsymbol{H}_t}
\newcommand{\adjacency}{\boldsymbol{A}_K}

\title{Mobility Edge for the Anderson Model on the Bethe Lattice}
\author{Amol Aggarwal}
\author{Patrick Lopatto}
\date{\today}
\begin{document}
\begin{abstract}
We pinpoint the spectral decomposition for the Anderson tight-binding model with an unbounded random potential on the Bethe lattice of sufficiently large degree. We prove that there exist a finite number of mobility edges separating intervals of pure-point spectrum from intervals of absolutely continuous spectrum, confirming a prediction of Abou-Chacra, Thouless, and Anderson  \cite{abou1973selfconsistent}. A central component of our proof is a monotonicity result for the leading eigenvalue of a certain transfer operator, which governs the decay rate of fractional moments for the tight-binding model's off-diagonal resolvent entries.
\end{abstract}
\maketitle

\setcounter{tocdepth}{1}
\tableofcontents

\section{Introduction}

\subsection{Background} 
In 1958, Anderson proposed that the diffusion of waves in a random environment should halt completely at sufficiently high disorder \cite{anderson1958absence}. He was  motivated by experimental results on electron spins in doped semiconductors indicating relaxation times far longer than those predicted by prevailing theories \cite{feher1959electron}.
However, the phenomenon uncovered by Anderson is not specific to electron wave functions, but instead a general feature of the dynamics of all waves in disordered media, both quantum and classical \cite{lagendijk2009fifty}. Consequently, his proposal, now known \emph{Anderson localization}, has become a foundational concept in condensed matter physics \cite{abrahams201050,sheng2006introduction,lee1985disordered, thouless1974electrons,LesHouches1978,anderson1978local}. 

Anderson's original analysis was based on the \emph{tight-binding model}, defined as follows. For any graph $\mathbb{G}$, let $\boldsymbol{A} $ denote the adjacency matrix of $\mathbb{G}$, and let $\boldsymbol{V}$ denote a diagonal operator whose non-zero entries are independent random variables, called the \emph{potential}. The tight-binding Hamiltonian is defined as $\hamintro =  -t \boldsymbol{A} + \boldsymbol V$, where $t \ge 0 $ represents the inverse of the disorder strength. The time evolution induced by this Hamiltonian models a single electron moving by nearest-neighbor hopping (the $\boldsymbol{A} $ term) through a lattice with impurities (the $\boldsymbol{V}$ term). 

A common choice for $\mathbb{G}$ is the $d$-dimensional lattice $\mathbb{Z}^d$, with the case $d=3$ being of obvious physical significance. 
Different phase diagrams have been proposed, depending on the dimension. For $d \le 2$, the system should exhibit Anderson localization at all energies and positive values of $t$. For $d \ge 3$, one should find localization at all energies when $t$ is sufficiently small. When $d \ge3$ and $t$ is large, a sharp transition between phases of localization and delocalization, called a \emph{mobility edge}, is predicted; see, for instance, the book \cite{aizenman2015random}.

Certain aspects of this picture have received rigorous proof. Localization at all energies in one dimension was demonstrated in
\cite{ishii1973localization,gol1977pure}. For $d\ge 2$, localization for large disorder, or large energies, was originally shown in \cite{frohlich1983absence} using multi-scale analysis (see also \cite{von1989new,spencer1988localization}). Shorter proofs, using the fractional moment method, were given in \cite{aizenman1993localization,aizenman1994localization}. Additionally,  \cite{aizenman1994localization} established exponential decay of the eigenfunction correlator, which was  inaccessible by multi-scale analysis (until the much more recent argument in \cite{imbrie2016multi}). Poisson spectral statistics in finite volumes, a signature of localization, were demonstrated in \cite{minami1996local}, and the technically challenging case of localization for Bernoulli potentials was addressed in \cite{ding2018localization,li2019anderson,carmona1987anderson} (see also \cite{imbrie2021localization}). 
However, the existence of a delocalized regime for $d \ge 3$ has not been rigorously established. As such, the more refined question of proving the presence of a mobility edge for the Anderson model on $\mathbb{Z}^d$ remains unresolved.

Another common choice for $\mathbb{G}$ is the infinite $K$-regular tree, where a similar phase diagram is expected. Also called the \emph{Bethe lattice},  it has served as an indispensable tool for studying Anderson localization (and the mobility edge) for decades. The Bethe lattice is believed to retain many of the essential features of the tight-binding model on $\mathbb{Z}^d$ (for $d \ge 3$), but its simpler topology permits more explicit calculations; see, for example, the physics works 
 \cite{abou1973selfconsistent,abou1974self,mirlin1991localization,miller1994weak}. Indeed, in \cite{abou1973selfconsistent,abou1974self}, the study of Anderson localization on $\mathbb{Z}^d$ is approached by perturbation theory applied to a certain self-consistent equation, which is exact on the Bethe lattice. Additionally, there has recently been significant interest in many-body localization, where the Anderson model on the Bethe lattice yields useful qualitative insights through  approximate mappings from certain many-body Hamiltonians 
 \cite{basko2006metal, gornyi2005interacting, altshuler1997quasiparticle}. This connection has prompted a renewed focus on Bethe lattice model in the physics community; for details, we refer to \cite{pascazio2023anderson} and the references therein.
 
 We now delineate what is known rigorously for the Bethe lattice. For unbounded potentials and every $t >0$, it was shown in \cite{aizenman1994localization,aizenman1993localization} that localization is present at sufficiently large energies. 
 Delocalization at small energies for both bounded and unbounded potentials, identified by the presence of  purely absolutely continuous spectrum, was shown in \cite{klein1998extended}; see also  \cite{klein1996spreading,aizenman2006stability,froese2007absolutely}. 
 More generally,  \cite{aizenman2013resonant} proved a sufficient condition for the existence of absolutely continuous spectrum;  this criterion is phrased in terms of a fractional moment bound for the off-diagonal entries of the resolvent of $\hamintro$ and almost complements the earlier Simon--Wolff criterion for pure-point spectrum (indicating localization) \cite{SCSRP,aizenman1993localization}.
They further applied it to prove several results about the absolutely continuous spectrum of $\hamintro$ under sufficiently small disorder; in particular, they showed that it extends beyond the spectrum of $\boldsymbol{A}$ if the potential $\boldsymbol{V}$ is unbounded, and that it exists near the spectral edges of $\hamintro$  if $\boldsymbol{V}$ is bounded. Later work of \cite{bapst2014large} used the criterion of [AW13] to provide upper and lower bounds on the endpoints of intervals of pure-point and absolutely continuous spectrum. These bounds converge towards each other as the degree $K$ of the graph tends to $\infty$ but do not coincide for any finite $K$; hence, the existence of mobility edges remained an open question.

In this article, we establish that mobility edges are present in the Anderson model on the Bethe lattice when the potential is unbounded and the degree $K$ is sufficiently large.

\subsection{Results} \label{s:tightbinding}

We let $K >1$ be an integer parameter and let $\bbT = \bbT_K$ be a rooted tree graph with fixed branching number $K$. 
In particular, every  vertex of $\bbT$ has degree $K+1$, except for the root, which has degree $K$. 
We let $\bbV = \bbV_{K}$ denote the vertex set of  $\bbT$.\footnote{
Our main result concerns $\bbT$, which differs from the Bethe lattice by having a root with degree $K$ instead of $K + 1$. This modification is a fairly standard one in the literature  \cite{aizenman2013resonant,bapst2014large} and is made only for notational convenience; it is straightforward to see that this directly implies the same result for the Bethe lattice itself.}

For any vertex $v \in \bbV$, let $\delta_v \in L^2 ( \bbV)$ denote the unit vector supported on $v$. We let $\adjacency$ denote the operator corresponding to the (infinite) adjacency matrix of $\bbT$, defined by stipulating that $\langle \delta_v , \adjacency \delta_w \rangle = 1$ if $w$ is a child of $v$, or $v$ is a child of $w$, and $\langle \delta_v , \adjacency \delta_w \rangle = 0$ otherwise. 
Let $(V_v)_{v \in \bbV}$ be a collection of independent, identically distributed (i.i.d.) 
random variables, and define the operator $\V$ densely on $L^2(\bbV)$ by $\langle \delta_v , \V \delta_w \rangle = V_v$ if $v=w$, and $0$ otherwise (then extending it by linearity). 
For any $t \in \R$ we define a random operator $\hamiltonian$ by
\be\label{e:hamiltonian}
\hamiltonian  = - t \adjacency + \V ,
\ee 
which is self-adjoint on the domain of $\V$. (See, e.g., \cite{kirsch2008invitation} for details.) 

In this work, we always suppose that the common distribution of the entries of $\V$ has a density, which we denote $\rho$. 
We introduce the following condition on probability densities, which our main result will assume holds for $\rho$. 
\begin{definition}\label{d:Lregular}
Let $p: \R \rightarrow [0, \infty)$ be a probability density function. Given $\L >0$, we say that $p$ is \emph{$\L$-regular} if the following conditions hold.
\begin{enumerate}
\item We have $p(x) \le \L ( 1 + x^2)^{-1}$ for all $x \in \R$.
\item We have $ p(x) \ge \L^{-1}$ for all $x\in [- \L^{-1}, \L^{-1}]$. 
\item 
The derivative $p'$ exists everywhere and 
$|p'(x) | \le \L ( 1 + |x| )^{-\fourthirds}$ for all $x \in \R$. In particular, the density $p$ is Lipschitz continuous with Lipschitz constant $\L$.
\item For every $v>0$, we have $\inf_{|x| < v} p(x) > 0$.
\item There are finitely many values of $x \in \R$ such that $p'(x) = 0$. 
\end{enumerate}
\end{definition}
\begin{remark}
The conclusions in all of our results will depend on $\rho$ only through $\L$. Further, the final two conditions in \Cref{d:Lregular} enter into the proof of our main result (\Cref{t:main} below) only through their use as hypotheses for the first part of \Cref{l:criteria}, a criterion for absolutely continuous spectrum proved in \cite{aizenman2013resonant}. They are used nowhere else. As alluded to in \cite{aizenman2013resonant}, it seems likely that a version of this criterion could be derived for compactly supported potential distributions, at the cost of additional technicalities. If such a result were available, then our main result could also be stated for such potentials. 
\end{remark}

We will distinguish between localization and delocalization for $\hamiltonian$ through its spectrum, understanding intervals with pure-point spectrum as localized regions, and intervals with absolutely continuous as delocalized regions; see, e.g., \cite[Chapter 3]{teschl2014mathematical} for definitions of the two spectral types.  These spectral properties are intimately related to the transport properties discussed above; see \cite[Section 7.3]{kirsch2008invitation} for details.

The following theorem establishes the existence of a mobility edge  for $\hamiltonian$ when $K$ is sufficiently large. We parameterize  $t$ as $t = g(K \log K)^{-1}$ with $g>0$ constant; this scaling places the mobility edge in a compact interval independently of the value of $K$ and has appeared numerous times before in the literature  \cite{bapst2014large, warzel2014surprises, abou1973selfconsistent}. The theorem states that for sufficiently large $K$, a mobility edge is present near every energy $E\in \R$ such that $\rho(E) = (4g)^{-1}$ and $\rho'(E) \neq 0$. This confirms a prediction of Abou-Chacra, Thouless, and Anderson 
for the location of the mobility edge \cite{abou1973selfconsistent}.

When $\rho$ is unimodal and symmetric, which is the case typically considered in the physics literature, our result implies the usual description of a region of delocalized energies at zero sandwiched between unbounded regions of localization. More precisely, there exists  $E_{c}(g)>0$ such that the spectrum is absolutely continuous in the interval $|E| < E_c$, and pure-point for $|E| >  E_c$.  When $\rho$ is not unimodal, interleaved regions of localization and delocalization, separated by multiple mobility edges, may occur (depending on the strength of the disorder). This scenario is illustrated in \Cref{f:figure1}. 

In the statement below, the set $\mathcal M$ represents the set of mobility edges for a given parameter $g$. The proof of this theorem is given in \Cref{s:mainproofintro}. 

\begin{theorem}\label{t:main}
For any real number $\L> 1$, there exists a constant $K_0 (\L) > 1$ such that the following holds for all $K \ge K_0$. 
Fix $g \in \R$ and an $\L$-regular probability density $\rho$ such that 
\be
 \frac{1}{ 4 \| \rho \|_{\infty}} + \frac{1}{\L}  < g <\L.
\ee 
Suppose further that for any point $E$ such that $\rho(E) = 1/ (4g)$, either $\rho'(E') \ge 1/ \L$ for all $E' \in [ E - 1/\L, E + 1/\L]$, or $\rho'(E') \le  - 1/ \L$ for all $E' \in [ E - 1/\L, E + 1/\L]$. Set $t = g (K \log K)^{-1}$, and consider the operator $\hamiltonian  = - t \adjacency + \V$, where the common distribution of the non-zero entries of $V$ has density $\rho$. 

There exists a finite set $\mathcal{M} = \{ \mathfrak{M}_1, \mathfrak{M}_2, \ldots , \mathfrak{M}_n\} \subset \mathbb{R}$ such that the following holds, where we label the $\mathfrak M_i$ in increasing order.  
We denote $\mathcal{M}^{\circ} = \mathbb{R} \setminus \mathcal{M}$ and let $\mathcal{E} = \{ E \in \mathbb{R} : \rho (E) = 1/(4g) \}$.
\begin{enumerate}
\item The cardinality of $\mathcal M$ equals the cardinality of $\mathcal E$. Moreover, labeling the elements of $\mathcal E$ in increasing order as $\mathfrak E_1, \dots, \mathfrak E_n$, we have 
$|\mathfrak{M}_k - \mathfrak{E}_k| < \L^{-1}$ for every $k \in \unn{1}{n}$.
\item  The set $\mathcal{M}^{\circ}$ is a finite disjoint union of open intervals. Let I be any such open interval in this union, and set $\ell = \inf I$.
\begin{enumerate}
\item 
If $\ell = -\infty$ or $\rho' (\ell) < 0$, then $\hamiltonian$ almost surely has pure-point spectrum on $I$.
\item
 If $\rho' (\ell) > 0$, then $\hamiltonian$ almost surely has absolutely continuous spectrum on $I$.
\end{enumerate}
\end{enumerate}
\end{theorem}

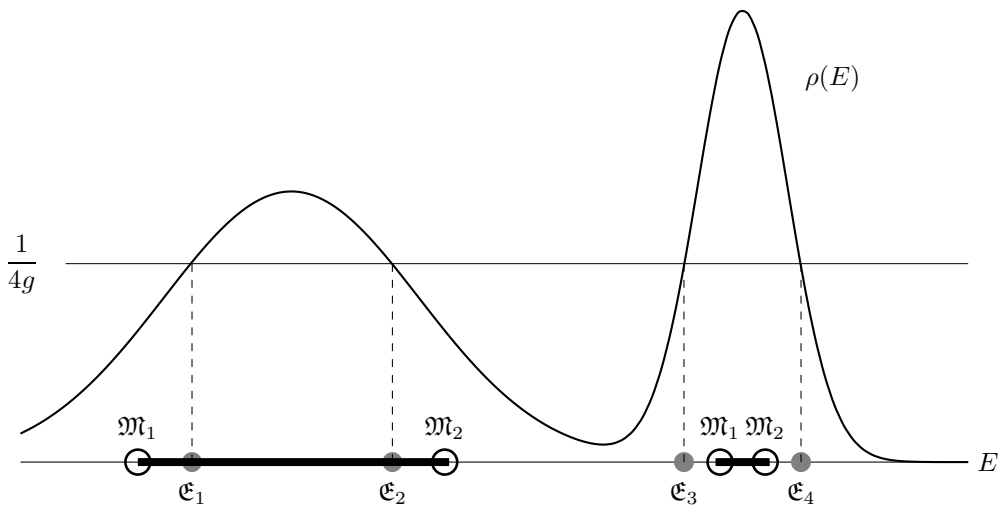
\begin{figure}
\begin{tikzpicture}[scale=1.2, samples=100,emptycircle/.style={circle, draw=#1, fill=none},
    emptycircle/.default=black]
\draw[-] (-5.5,0) -- (5,0) node[right] {$E$};
\draw[smooth, domain=-5.5:5, color=black, thick] 
    plot (\x,{5*e^(- 2 * (\x-2.5)^2)+ 3*e^(- (1/4)*(\x+2.5)^2)  }) ; 
\node at (-5.2,2.2)[left] {$\displaystyle\frac{1}{4g}$};
\node at (3.5,4)[above] {$\rho(E)$};
\draw[smooth, very thin, domain=-5:5, color=black] plot (\x, 2.2 );
\node[fill,circle,label=below :$\mathfrak E_1$, color=gray, scale = .8] at (-3.6,0) {};
\draw [dashed] (-3.6,0) -- (-3.6,2.2);
\node[fill,circle,label=below :$\mathfrak E_2$, color=gray, scale = .8] at (-1.38,0) {};
\draw [dashed] (-1.38,0) -- (-1.38,2.2);
\node[fill,circle,label=below :$\mathfrak E_3$, color=gray, scale = .8] at (1.85,0) {};
\draw [dashed] (1.85,0) -- (1.85,2.2);
\node[fill,circle,label=below :$\mathfrak E_4$, color=gray, scale = .8] at (3.147,0) {};
\draw [dashed] (3.147,0) -- (3.147,2.2);
\draw[smooth,  domain=-4.2:-.75, color=black, line width=3pt] plot (\x, 0 );
\node[emptycircle, line width = 1pt, label=above :$\mathfrak M_1$, color=black] at (-4.2,0) {};
\node[emptycircle, line width = 1pt,label=above :$\mathfrak M_2$, color=black] at (-.8,0) {};
\draw[smooth,  domain=2.2:2.8, color=black, line width=3pt] plot (\x, 0 );
\node[emptycircle, line width = 1pt, label=above :$\mathfrak M_1$, color=black] at (2.25,0) {};
\node[emptycircle, line width = 1pt, label=above :$\mathfrak M_2$, color=black] at (2.75,0) {};
\end{tikzpicture}
\caption{Illustration of \Cref{t:main}. Shaded regions on the horizontal axis indicate regions of absolutely continuous spectrum, and unshaded regions indicate pure-point spectrum.}
\label{f:figure1}
\end{figure}

\Cref{t:main} states that for sufficiently large $K$, near each point $\mathfrak{E}_k  \in \mathcal E$, there exists a point $\mathfrak M_k$ separating an interval of spectral localization (on one side of $\mathfrak M_k$) from an interval of spectral delocalization (on the other side of $\mathfrak M_k$), under the assumption that $\rho'$ is bounded away from zero near \emph{every} point in $\mathcal E$. We have chosen this formulation for brevity. However, our proof actually shows the stronger, local statement that the existence of a sharp transition within a distance $\L^{-1}$ of any given $\mathfrak{E}_k  \in \mathcal E$ requires only that $\rho'$ be bounded away from zero on $[\mathfrak{E}_k  - \L^{-1}, \mathfrak{E}_k + \L^{-1}]$, and nowhere else.

\begin{remark}
It is quickly verified that our argument, in combination with \cite[Proposition 2.6]{aizenman2013resonant}, gives the stronger conclusion that the localized components of $\mathcal M^\circ$ also exhibit \emph{exponential dynamical localization} (see \cite[Definition 1.1]{aizenman2013resonant}). We omit this discussion for brevity. 
\end{remark}
\subsection{Related Works}
The existence of a localized regime on tree graphs with random branching numbers was demonstrated in \cite{damanik2020localization}. Recently, the existence of a delocalized regime was shown for a variant of the tight-binding model, in which the diagonal potential is replaced by a diagonal block matrix containing Gaussian entries \cite{yang2025delocalization}. The presence of a localized regime in a closely related model was shown earlier in \cite{peled2019wegner}.

A mobility edge is also predicted to exist in several random matrix models \cite{biroli2021levy,cizeau1994theory,tarquini2016level}. Band matrices are thought to exhibit a mobility edge in dimensions $d \ge 3$ when the band width is sufficiently large, but kept constant as the matrix dimension grows \cite{RBSM,DSMRM}. There has been significant progress in understanding these matrices when $d=1$ \cite{bourgade2018random,bourgade2019random,yang2018random,URBMSM,URBM,STM,yau2025delocalization,schenker2009eigenvector,cipolloni2024dynamical,chen2022random,dubova2024quantum,sodin2010spectral}, $d = 2$ \cite{dubova2025delocalization}, and $d\ge 7$ \cite{DQDRBMHD,DQDRBM,BUQURBM}. However, in each of the cases studied in the above works, the  spectrum is expected to be either all localized or all delocalized, and there is no mobility edge. 

Further, a mobility edge is also expected in the adjacency matrices of  Erd\H{o}s--R{\'e}nyi random graphs with constant average degree \cite{fyodorov1991localization,mirlin1991universality}. However, this prediction must be understood in the appropriate way. Such a graph will (with high probability) possess a single connected component that contains a positive fraction of the vertices, called the giant component. Each other component contains an asymptotically vanishing fraction of the vertices; they are typically treelike and give rise to eigenvalues throughout the spectrum accompanied by localized eigenvectors. However, if one restricts to the giant component, the traditional scenario of a mobility edge separating regimes of localization and delocalization should be restored \cite{cugliandolo2024multifractal}. The existence of both localized and delocalized eigenvectors in the spectrum has been rigorously established \cite{bordenave2011rank,coste2021emergence,bordenave2017mean,arras2021existence,chayes1986density,salez2015every}.
Stronger results quantifying the regions of (de)localization were established in \cite{alt2021delocalization,alt2021completely,alt2024localized,alt2023poisson} when  the average degree of the graph diverges slowly (logarithmically) in its number of vertices. When the edges of the graph are given complex Gaussian weights, a sharp transition in the second correlation function of the characteristic polynomial was shown in \cite{afanasiev2016correlation} through a rigorous supersymmetric argument. 

A mobility edge was also predicted to exist for symmetric random matrices with heavy-tailed entries, called L\'evy matrices, when the entries have infinite first moment \cite{cizeau1994theory,tarquini2016level}. This was recently proved in \cite{aggarwal2022mobility} when the entries are $\alpha$-stable variables with $\alpha$ close to $0$ to $1$; for further rigorous results on the phase diagram, see \cite{aggarwal2021eigenvector,aggarwal2021goe,bordenave2017delocalization,bordenave2013localization}. 

\subsection{Proof Ideas}
We now briefly (and informally) describe the main ideas of our proof. A more detailed outline with precise statements is given in \Cref{s:proof}.

\subsubsection{Reduction to a Transfer Operator}
We begin by reducing the proof of \Cref{t:main} to the analysis of a certain transfer operator. First, results of \cite{aizenman2013resonant} imply that the spectral type of $\hamiltonian$ at any $E \in \R$ is  determined by sign of a certain functional $\phi(1;E)$, which is expressed in terms of the resolvent $\bfR(z) = [R_{uv}(z)] = (\hamiltonian - z)^{-1}$ of $\hamiltonian$. Specifically, for any integer $L \ge 1$, define
	\begin{flalign}
		\label{e:phiLKLsketch}
		\Phi_{L} (s; z) = K^L \cdot \mathbb{E}  \Big[ \big| R_{0{v_L}} (z) \big|^s \Big], \qquad \varphi_L (s; z) = L^{-1} \log \Phi_L (s; z),
	\end{flalign}
	\noindent where $\mathbb{V}  (L)$ is the set of vertices at the $L$-th level of $\mathbb{T}$, and define the limits (when they exist)
	\begin{flalign}
		\label{phisketch} 
		\varphi (s; z) = \displaystyle\lim_{L \rightarrow \infty} \varphi_L (s; z), \qquad \varphi (s; E) = \displaystyle\lim_{\eta \rightarrow 0} \varphi (s; E + \mathrm{i} \eta).
	\end{flalign}
If $\phi(E) < 0$, then the Simon--Wolff criterion \cite{SCSRP} implies that $\hamiltonian$ has pure-point spectrum at $E$. If $\phi(1;E)>0$, the resonant delocalization criterion of Aizenman--Warzel \cite{aizenman2013resonant} implies that $\hamiltonian$ has absolutely continuous spectrum at $E$. Thus, we must pinpoint the sign of each $\phi(1;E)$. 

Next, since the off-diagonal resolvent entry $R_{0v}$ can be expressed as a product of diagonal resolvent entries (see \Cref{rproduct}), it can be shown that $K^{-L} \cdot \Phi_L (s; z)$ can be computed through  $L+1$ iterations of the \emph{transfer operator} $F_{s,z}$, acting on functions $u : \mathbb{R} \rightarrow \mathbb{R}$, defined by 
\be\label{e:Fdefz}
(F_{s,z} u)(x) = 
\frac{t^{2-s}}{|x|^{2-s}}
\int_{-\infty}^\infty
\rho_z \left( - y - \frac{t^2}{x} \right) u(y) \, dy, 
\ee

\noindent where $\rho_z$ denotes the density of the random variable $R_{00}(z)^{-1}$. Letting $L$ tend to $\infty$, it follows that $\varphi (s; z)$ can be expressed through the spectral radius $\lambda_{s,z}$ of $F_{s,z}$, namely, 
\be\label{e:philambdaintro}
\phi(s ; z) = \ln K + \ln \lambda_{s,z}.
\ee 

\noindent Thus, the spectral type of $\bm{H}_{K,t}$ at $E$ is determined by whether $\lambda_{1^-,z} > 1/K$ (absolutely continuous) or $\lambda_{1^-,z} < 1/K$ (pure-point). This formulation was essentially known; see, e.g., \cite{bapst2014large}. We explain the brief derivation of $F_{s,z}$ in \Cref{s:origintransfer}, and we formalize the functional analytic setup (defining the Banach space on which $F_{s,z}$ acts and applying the Krein--Rutman theorem) in \Cref{s:KRapp}.

A direct analysis of $F_{s,z}$ is complicated by the fact that $R_{00} (z)$ might be complex, with correlated real and imaginary parts. To circumvent this, we proceed under the assumption\footnote{A similar assumption appeared in the original self-consistent heuristics \cite{abou1973selfconsistent} of Abou-Chacra--Thouless--Anderson from over fifty years ago that correctly predicted the approximate location of the mobility edge.}  that $\lim_{\eta \rightarrow 0} R_{00}(E + \iu \eta)$ is real almost surely; this is equivalent to assuming that $\bm{H}_{K,t}$ has pure-point spectrum at $E$. Then, as $\eta \rightarrow 0$, the density $\rho_{E+\mathrm{i} \eta}$ in \eqref{e:Fdefz} is replaced by one  for a real-valued random variable, $\rho_E$. Denote the associated transfer operator by $F_{s,E}$ and spectral radius by $\lambda_{s,E}$. 

Although this $\lambda_{s,E}$ is only directly related to $\varphi (s; E)$ if $\lim_{\eta \rightarrow 0} R_{00} (E + \mathrm{i} \eta)$ is almost surely real, one can still prove that the spectral type of $\bm{H}_{K,t}$ at $E$ is determined by whether $\lambda_{1^-,E} > 1/K$ or $\lambda_{1^-,E} < 1/K$. Indeed, assume to the contrary that $\lambda_{1^-, E} > 1/K$ but $\bm{H}_{K,t}$ has pure-point spectrum at $E$. The latter implies that $\lim_{\eta \rightarrow 0} R_{00} (E + \mathrm{i}\eta)$ is almost surely real, justifying the assumption above. Hence, \eqref{e:philambdaintro} yields $\varphi (1; E) = \log K + \log \lambda_{1^-, E} > 1$, which implies $\bm{H}_{K,t}$ has absolutely continuous spectrum  at $E$, a contradiction. While this reasoning would not apply directly when $\lambda_{s,E} < 1/K$, it can be combined with a continuity argument from \cite[Section 12]{aggarwal2022mobility} (implemented in \Cref{s:lyapunovcontinuity} and \Cref{s:lyapunovcontinuity2}) to prove pure-point spectrum in this case. 

Hence, we must determine the sign of $\lambda_{1^-,E} - 1/K$. When $\rho(E)$ is bounded away from $(4g)^{-1}$, this was done by Bapst \cite{bapst2014large}, by analyzing the action of $F_{s,E}$ on certain well-chosen test functions. We reproduce these concise arguments in \Cref{s:evaluebounds} below (modifying them slightly to our needs), to provide rough upper and lower bounds on $\lambda_{s,E}$ for such $E$. It remains to understand the sign of $\lambda_{1^-, E} - 1/K$ when $\rho (E) \approx (4g)^{-1}$. This is the regime when one should expect the question to become more subtle, as here we have $\lambda_{1^-, E} - 1/K \approx 0$. 

The main contribution of this paper is \Cref{t:mainlambda}, stating that $\lambda_{s,E}$ is strictly monotonic around such $E$. The only example we know of for which such monotonicity has been proven is L\'{e}vy matrices \cite{aggarwal2022mobility}, where the proof was facilitated by an explicit (though fairly involved) formula for the associated $\lambda_{s,E}$. For the tight-binding model on the Bethe lattice with general potential, such a formula for $\lambda_{s,E}$ is not known (nor really expected to exist), so we proceed differently.

\subsubsection{Local Monotonicity for $\lambda_{s, E}$}

\label{LambdaMonotone}

The proof of \Cref{t:mainlambda} breaks into three main steps. In what follows, we let $E_0 \in \mathbb{R}$ satisfy $\rho(E_0) = (4g)^{-1}$ and $\rho'(E_0) < 0$ (the case when $\rho' (E_0) >0$ is analogous); we then analyze $\lambda_{s,E}$ for $E$ close to $E_0$. We also let $\Gamma$ denote a real-valued solution (which can be shown to exist by \Cref{l:pEexists}) to the recursive distributional equation
\be\label{e:rdeplusintro}
\Gamma  \overset{d}{=} \frac{1}{ V_0  - E - t^2 \sum_{i=1}^K \Gamma_i},
\ee 
where the $\Gamma_i$ are i.i.d. copies of $\Gamma$, and $V_0$ is an independent random variable with density $\rho$. Observe that any limit point of $(R_{00} (E+\mathrm{i}\eta))_{\eta \rightarrow 0}$ satisfies \eqref{e:rdeplusintro}, by the Schur complement identity (see \Cref{q12}). Thus, we let $\rho_E$ denote the density of $\Gamma^{-1}$. In view of \eqref{e:rdeplusintro}, we expect for $K$ large (and thus for $t \sim (K \log K)^{-1}$ small) that $\rho_E (x)$ is close to the density $\rho (x+E)$ of $V_0 - E$. This suggests the approximation $\rho_E (x) = \rho_E (x) + O(t)$, which will serve as a useful guide. However, it will not always be of direct use, since the small but constant order error of $t$ will often be larger than the scale on which we are operating.

{\textit{1. Estimating $\rho_{E_2} - \rho_{E_1}$}}: To understand how $\lambda_{s,E}$ varies with $E \approx E_0$, we begin by studying how the operator $F_{s,E}$ varies with $E$, to which end we must analyze how the density $\rho_E$ in its definition varies with $E$. We mention that a preliminary version of this question would be to show that it is unique for a given $E$, namely, that $(R_{00} (E+\mathrm{i}\eta))_{\eta \rightarrow 0}$ has a single limit point (or that \eqref{e:rdeplusintro} admits a unique real-valued solution); to our knowledge, this was not known previously. 

For $E_1 < E_2$ close to $E_0$, we in \Cref{s:selfenergydensityestimates} more generally prove two estimates on $\rho_{E_2} - \rho_{E_1}$. First, we show as \Cref{l:amoluniform} that
\begin{flalign}
	\label{rhoe1e2} 
\big| \rho_{E_2}(x) - \rho_{E_1}(x)  \big| \le C |E_1 - E_2| \big(|x|+1\big)^{-\fourthirds } ,
\end{flalign} 
which shows that $\rho_E$ is Lipschitz in $E$, with a Lipschitz constant decaying in $x$. Second, in the same proposition, we estimate the sign of $\rho_{E_2} - \rho_{E_1}$ near $0$, by showing that   
\be\label{e:le11finiteintro}
\rho_{E_2}(x) - \rho_{E_1} (x)  \le  - c(E_2-E_1)
\ee 
for $x$ close to zero. Observe since $\rho'(E) < 0$ for $E \approx E_0$ that this is suggested by the fact that $\rho_E (x) = \rho (x+E) + O(t)$, but not implied by it for $E_2-E_1 \ll t$. We prove \Cref{l:amoluniform}  through a Gronwall-type argument, where an integral representation for the difference $\rho_{E_2} - \rho_{E_1}$ is bounded in terms of itself.

\textit{2. Shape of the Eigenfunction $v_{s,E}$:} To precisely understand how $\lambda_{s,E}$ varies in $E$, we must not only understand the operator $F_{s,E}$ but also the associated eigenfunction $v_{s,E}$. This differs from prior work, such as \cite{bapst2014large}, which circumvented the analysis of $v_{s,E}$ by instead estimating the action of $F_{s,E}$ on certain test vectors (which loses careful control on $\lambda_{s,E}$ for $E \approx E_0$). We approximate the eigenfunction $v_{s,E}$ in \Cref{e:evectorbounds}, where we begin by defining an explicit function $\tilde{u}$ by
\be\label{e:tildeuintro}
\tilde u(x) 
 = (t^2 \ln K)^{-1} \cdot \one\{ |x| \le  Ct^2 \}
+  |x|^{-(2-s)} \cdot \one\{  Ct^2  \le  |x|  \}.
\ee 

\noindent Observe that, if $s$ is very close to $1$, the function $\tilde{u} (x)$ behaves as $(t^2 \log K)^{-1}$ for $x \le Ct^2$ and as $x^{-1}$ for $x \ge Ct^2$. This is the general behavior we expect for the eigenfunction $v_{s,E}$. However, the precise location and nature of the crossover around $x = O(t^2)$, at which $v_{s,E} (x)$ changes from behaving as $x^{-1} \sim t^{-2}$ to as $(t^2 \log K)^{-1}$, is not transparent to us. So, instead of comparing $v_{s,E}$ directly to $\tilde{u}$, we will compare $F_{s,E} v_{s,E}$ to  $F_{s,E} \tilde u$. Specifically, as \Cref{l:eigenvectorsandwich}, we show that 
\begin{flalign}
	\label{vuf} 
c (F_{E}\tilde u)(x) \le (F_{s,E} v_{s,E} )(x) \le C (F_{E}\tilde u)(x),
\end{flalign} 

\noindent where the constants $c$ and $C$ are uniform in $K$. While useful, this does not explain how $v_{s,E}$ varies with $E$ and, indeed, we do not know how to do this. 

We bypass this by making use of the identity
\be\label{e:computelnlambda}
\ln \lambda_{s, E} = \lim_{n \rightarrow \infty} \frac{ \ln \| F^n_{E} \tilde u \|_1 }{n}.
\ee 

\noindent The benefit of this representation is that it replaces $v_{s,E}$ with the function $\tilde{u}$ that is independent of $E$. However, it comes at the expense of having to compute high powers of $F_{s,E}$ when applied to $\tilde{u}$.

\textit{3. Estimating Iterates of $F_{s,E}$}: To deduce the monotonicity of $\lambda_{s,E}$, we analyze the difference $\ln \lambda_{s,E_{2}} - \ln \lambda_{s, E_1}$ using  \eqref{e:computelnlambda}. We have 
\be\label{e:Fnormdiff}
\| F^n_{s,E_2} \tilde u \|_1 - \| F^n_{s,E_1} \tilde u \|_1 = 
\int_{-\infty}^\infty   \big( (F^n_{s,E_2}   - F^n_{s,E_1}) \tilde u\big)(x) \, dx.
\ee
We note the identity
\be\label{e:Fdiffsumintro}
F^n_{s,E_2} -  F^n_{s,E_1} = \sum_{j=0}^{n-1} F_{s,E_2}^{n-j -1} (F^\circ) F_{s,E_1}^{j},
\ee 
where we define
\begin{flalign*}
(F^\circ u)(x) = (F_{s,E_2} - F_{s,E_1}) u (x) = \frac{t^{2-s}}{|x|^{2-s}}
\int_{-\infty}^\infty \big( \rho_{s,E_2} ( - y - t^2 x^{-1} )  - \rho_{s,E_1} ( - y - t^2 x^{-1}) \big)u(y) \, dy.
\end{flalign*}

When $|x| \ge C t^{2}$ and $y$ is close to $0$ (which is where the above integral will mainly be supported), \eqref{e:le11finiteintro} implies that $\rho_{s,E_2}(-y-t^2/x) - \rho_{s,E_1} (-y-t^2/x) < c(E_2-E_1)$. This suggests that $F^{\circ}$ should acts as a negative operator for functions supported outside of $[-Ct^2, Ct^2]$. 

In \Cref{s:testestimates}, we study $F^\circ$ by justifying this reasoning. In particular, as \Cref{l:pizza}, we confirm that $F^\circ F^j \tilde u(x)$ is negative for $|x| \ge C t^2$ and bound it away from zero. 
The contribution from $|x| \le C t^2$ might be positive, so we also provide a bound on this quantity in the same \Cref{l:pizza}, using \eqref{rhoe1e2}. Throughout these analyses, we make repeated use of \eqref{vuf}, to bound high powers of $F_{s,E}$ applied to $\tilde{u}$ by a multiple of $v_{s,E}$ (on which $F_{s,E}$ acts diagonally).

In \Cref{s:conclusion}, we use the estimates from \Cref{l:pizza} as inputs to show that the negative contributions from the regime $|x| \ge C t^2$ outweigh the positive ones from when $|x| \le Ct^2$. We therefore deduce that the $F_{s,E_2}^{n-j -1} (F^\circ) F_{s,E_1}^{j}$ from \eqref{e:Fdiffsumintro} essentially act as nonpositive operators on $\tilde u$, when $E_1 < E_2$ are close to $E_0$. Combining \eqref{e:Fnormdiff} and \eqref{e:Fdiffsumintro}, we find $\| F^n_{s,E_2}\|_1 -  \|F^n_{s,E_1}\|_1 \le 0$. Then \eqref{e:computelnlambda} shows that $\lambda_{s,E}$ is weakly decreasing near $E_0$, which is stated as \Cref{l:penultimate}. This argument can be refined, as in \Cref{lambdae120}, to show that $\lambda_{1^-, E}$ is strictly decreasing for $E$ close to $E_0$, which establishes the monotonicity stated in \Cref{t:mainlambda}.

\section{Proof Outline}\label{s:proof}
In \Cref{s:notation}, we set our notations and conventions, which will be in force throughout this article. 
\Cref{s:rez} recalls the definition of the resolvent of $\hamiltonian$. 
\Cref{e:eigenvalueconventions} recalls localization and delocalization criteria for $\hamiltonian$, which are given in terms of the associated free energy functional. 
\Cref{s:transferoperatorintro} introduces a transfer operator, whose leading eigenvalue we will use to characterize the behavior of the free energy, and states our monotonicity result for this eigenvalue, \Cref{t:mainlambda}; this result is the main input to the proof of \Cref{t:main}.
\Cref{s:bootstrapintro} presents preliminary lemmas necessary for our characterization of the localized phase. Finally, we present the proof of \Cref{t:main} in \Cref{s:mainproofintro}, assuming the lemmas stated in this section. 

\subsection{Notation and Conventions}\label{s:notation}

We let $\mathbb{H} = \{ z \in \mathbb{C} : \Imaginary z > 0 \}$. For $z \in \bbH$, we often write $z = E + \iu \eta$  with $E = \Re(z)$ and $\eta = \Im(z)$.  We denote the root of $\bbT$ by $0$. For any $a, b \in \R$ with $a<b$, we let $\unn{a}{b}$ denote the set $\{ d \in \mathbb{Z} : a \le d \le b\}$. For any set $ S \subset \R$, we let $S^c$ denote its complement in $\R$. Similarly, let $\Omega$ denote the probability space on which the disorder variables $(V_x)_{x \in \bbV}$ are defined, and for any event $\mathcal E \subset \Omega$, let $\mathcal E^c$ denote its complement. 

Throughout this work, we consider parameters $g,K, \alpha>0$, and we fix the values
\be\label{e:tKdef}
 t = \frac{g}{K \log K},\qquad \Delta = \frac{t^2}{\alpha},
\ee
unless otherwise noted. 
Intermediate results below will involve various constants, which will sometimes depend on $\L$ from \Cref{t:main}. We typically omit this dependence in this notation whenever it occurs. Throughout the rest of this work, we fix some $\L >0$ and suppose that $\rho$ is an $\L$-regular probability density function on $\R$. We will always assume that $\L> 10$ (as our assumptions become weaker as $\mathfrak{L}$ increases).

\subsection{Resolvent}\label{s:rez}
We now recall the definition of the resolvent of $\hamiltonian$, which plays a central role in the localization and delocalization criteria that follow.

\begin{definition} 
		
\label{d:resolvent} 
	
For every $z \in \mathbb{H}$, the \emph{resolvent} of $\hamiltonian$, denoted by $\bfR: L^2 (\mathbb{V}) \rightarrow L^2 (\mathbb{V})$, is defined by $\bfR= \bfR(z) = (\hamiltonian - z)^{-1}$. For  any vertices $v, w \in \mathbb{V}$, we denote the $(v, w)$-entry of $\boldsymbol{R}$ by $R_{v w} = R_{v w} (z) = \langle \delta_{v}, \bfR\delta_{w} \rangle$. 
	
\end{definition} 
	
\noindent For $z \in \bbH$, let $p_{z}(w) \, dw = p_{K, t, z}(w) \, dw$ denote the density of the complex random variable $R_{00}(z)$. It is a well known consequence of the Schur complement identity (stated as \eqref{qvv} below) that the distribution $p_{E+\iu\eta}(w) \, dw$ is a solution to 
 the recursive distributional equation
\be\label{e:rdeplus}
\Gamma  \overset{d}{=} \frac{1}{ V_0  - E - \iu \eta - t^2 \sum_{i=1}^K \Gamma_i},
\ee 
where $\Gamma$ and the $\Gamma_i$ are i.i.d., and $V_0$ is distributed as an entry of $V$ and independent from the $\Gamma_i$ (see, e.g., \cite[Proposition 3.1]{aizenman2013resonant}). 
We will also need to consider the $\eta=0$ analogue of \eqref{e:rdeplus}, given by
\be\label{e:rde}
\Gamma \overset{d}{=} \frac{1}{ V_0  - E - t^2 \sum_{i=1}^K \Gamma_i}.
\ee 
The following lemma, which shows the existence of at least one solution to \eqref{e:rde}, is proved in \Cref{s:pEexists}.
\begin{lemma}\label{l:pEexists}
For every $E \in \R$, there exists a real-valued solution to the recursive distributional equation \eqref{e:rde}. 
\end{lemma}

\begin{definition}\label{d:pEdef}
We denote the density for the solution to \eqref{e:rde} by $p_E(w) \, dw$.\footnote{It is straightforward to show that such a density exists since $V_0$ is assumed to have a density.}
  (If there are multiple solutions, we select one arbitrarily.)
For any $E \in \R$, we let $\rho_{E}$ denote the density of $V_0 - E - t^2 \sum_{i=1}^{K-1} \Gamma_i$ and $p^{(M)}_E$ denote the density of $t^2 \sum_{i=1}^{M} \Gamma_i$ for any $M \in \N$, where the $\Gamma_i$ are mutually independent random variables, each with the density $p_E(w)\, dw$. For any $z \in \bbH$, we let $\rho_{z}$ denote the density of $V_0 - z - t^2 \sum_{i=1}^{K-1} \Gamma_i$, the denominator of \eqref{e:rdeplus}.
\end{definition}

\subsection{Localization and Delocalization Criteria}\label{e:eigenvalueconventions}
We begin by recalling the definition of the free energy functional associated with $\hamiltonian$, denoted by $\phi$ below.

	\begin{definition} 
		
		\label{moment1} 
	
	 For any integer $L \ge 1$, we define
	\begin{flalign}
		\label{sumsz1} 
		\Phi_{L} (s; z) = \mathbb{E} \Bigg[ \displaystyle\sum_{v \in \mathbb{V}  (L)} \big| R_{0v} (z) \big|^s \Bigg]; \qquad \varphi_L (s; z) = L^{-1} \log \Phi_L (s; z), 
	\end{flalign}
	
\noindent where $\mathbb{V}  (L)$ is the set of vertices of $\mathbb T$ at distance $L$ from the root $0$. 
	When they exist, we also define the limits  
	\begin{flalign}
		\label{szl} 
		\varphi (s; z) = \displaystyle\lim_{L \rightarrow \infty} \varphi_L (s; z); \qquad \varphi (s; E) = \displaystyle\lim_{\eta \rightarrow 0} \varphi (s; E + \mathrm{i} \eta).
	\end{flalign}
	
	\end{definition} 

The following lemma is an immediate consequence of \cite[Theorem 3.2]{aizenman2013resonant}, and states certain fundamental properties of $\phi$. It applies to the operator $\hamiltonian$, and the associated quantities in \eqref{sumsz1} and \eqref{szl}, defined using generic inverse disorder parameters $t>0$ (and not just the particular choice in \eqref{e:tKdef}). 

	 \begin{lemma}[{\cite[Theorem 3.2]{aizenman2013resonant}}] \label{l:aizenman} 	
	 	Fix $s, \epsilon \in (0,1)$ and compact interval $I\subset \R$. The following statements hold for all $t \in (\epsilon, 1)$.\begin{enumerate}
\item For all $z \in \bbH$, the  function $\varphi (s; z)$ is (weakly) convex and nonincreasing in $s \in (0, 1)$. 
\item For all $z\in \bbH$, the  limit $\varphi (s; z) = \lim_{L \rightarrow \infty} \varphi_L (s; z)$ exists and is finite. 
\item There exists a constant $C( s,\epsilon, I) > 1$ such that for all $z \in \bbH$ with $E \in I$ and $\eta \in (0,1)$, 
	 		\begin{flalign}
	 			\label{e:limitr0j2} 
	 			\big| \varphi_L (s; z) - \varphi (s; z) \big| \le \displaystyle\frac{C}{L}.
	 		\end{flalign}
\item We have 

\begin{equation}\label{e:crudephi}
s \cdot  \E \Big[\log \big| R_{0{0}} (z) \big|  \Big] + \log K \le \varphi(s ; z) \le -\frac{ s \log K}{2} + \log K.
\end{equation}
\item Fix $E\in \bbR$. If the limit $\varphi (s; E) = \lim_{\eta \rightarrow 0} \varphi (s; E + \mathrm{i} \eta)$ exists, then the limit $\varphi_L (s; E) = \lim_{\eta \rightarrow 0} \varphi_L (s; E + \mathrm{i} \eta)$ does also. In this case, the limits in $L$ and $\eta$ in the second equality of \eqref{szl} commute, namely, $\varphi (s; E) = \lim_{L\rightarrow\infty} \varphi_L (s; E )$.
\item For almost all $E\in \R$ (with respect to Lebesgue measure), the limit \[ \varphi (s; E) = \lim_{\eta \rightarrow 0} \varphi (s; E + \mathrm{i} \eta)\] exists for all $s\in (0,1)$. 
\end{enumerate}
	 \end{lemma} 
	 
\begin{definition}
For all $z \in \overline{\bbH}$, we set
\bex
\varphi (1; z) = \lim_{s \rightarrow 1} \varphi (s; z),
\eex
when it exists.
The existence of the limit for all $z \in \bbH$ follows from the first part of  \Cref{l:aizenman}. The existence of the limit for almost all $z \in \R$ follows from additionally using the final part of \Cref{l:aizenman}.
\end{definition}

The next lemma provides criteria for spectral localization and delocalization in terms of $\phi$.
The first statement is \cite[Theorem 2.5]{aizenman2013resonant}. The second statement follows  from combining \cite[Proposition 2.2]{aizenman2013resonant} and \cite[Proposition 4.1]{aizenman2013resonant}.
The third statement is the well-known criterion of Simon and Wolff \cite{SCSRP}; the precise version we use is a consequence of \cite[Theorem 5.7]{aizenman2015random}. 
We observe that by the usual theory of the Stieltjes transform (see, e.g., \cite{aizenman2015random}), the limit 
\begin{equation}\label{e:stieltjeslimit}
R_{00}(E) = \lim_{\eta \rightarrow 0} R_{00} (E + \iu \eta)
\end{equation}
exists for almost all $E\in \R$, with respect to Lebesgue measure. In the lemma, the qualifier \emph{almost all} is with respect to Lebesgue measure.
\begin{lemma}[{\cite{aizenman2013resonant,aizenman2015random}}]\label{l:criteria}
Fix $K>1$. Consider the operator $\hamiltonian  = - t \adjacency + \V$, fix an interval $I \subset \R$ and suppose that the common density of the non-zero entries of $V$ is $\L$-regular. Then the following three statements hold.
\begin{enumerate}
\item  For almost all $E\in \R$ such that  $\phi(1;E) > 0$, we have $\Im R_{00}(E) > 0$ almost surely. 
\item  If for almost all $E\in I$, we have $\Im R_{00}(E) > 0$ with positive probability, then the spectrum of $\hamiltonian$ is almost surely absolutely continuous in $I$.
\item If for almost all $E\in I$, we have $\phi(1;E) < 0$, then the spectrum of $\hamiltonian$ is almost surely pure-point in $I$.
\end{enumerate}
\end{lemma}

The following lemma complements the first part of the previous lemma. In particular, it implies that if $\phi(1;E) < 0$, then $\Im R_{00}(E + \iu \eta)$ vanishes as $\eta$ tends to zero. 
We establish it in 
\Cref{s:continuitypreliminary}.
Together with \Cref{l:criteria}, it shows that the sign of $\phi(1;E)$ determines both spectral (de)localization and the behavior of $R_{00}(E)$. 
\begin{lemma} \label{p:imvanish}
	Fix $E \in \mathbb{R}$ and $s, t\in(0, 1)$.
	 Let $\{\eta_j\}_{j=1}^\infty$ be a decreasing sequence such that $\lim_{j\rightarrow\infty} \eta_j =0$, and suppose that \bex\limsup_{j \rightarrow \infty} \varphi (s; E + \mathrm{i} \eta_j) < 0.\eex
	  Then $\lim_{j\rightarrow \infty} \Imaginary R_{00} (E + \mathrm{i} \eta_j) = 0$ in probability.
	\end{lemma}

\subsection{Transfer Operator}\label{s:transferoperatorintro}
We now introduce an integral operator whose leading eigenvalue determines the value of $\phi(1;E)$, under the assumption that $\Im R_{00}(E)$ is real. 
Following  \cite{bapst2014large}, we define the operator $F = F_{K,t,s,E}$ on functions $u : \R \rightarrow \R$ by
\be\label{e:Fdef}
(F u)(x) = 
\frac{t^{2-s}}{|x|^{2-s}}
\int_{-\infty}^\infty
\rho_E \left( - y - \frac{t^2}{x} \right) u(y) \, dy,
\ee
where $\rho_E$ is given by \Cref{d:pEdef}. 
We will consider $F$ as an operator a certain weighted $L^\infty$ space. 
\begin{definition}\label{d:X} Define $w \colon \R \rightarrow \R$ by $
w(x) = 1 + |x|^{2 - s }$. The Banach space $\mathcal X$ is defined as the set of functions $f \colon \R \rightarrow \R$ whose norm $\| f \| = \| f w \|_\infty$ is finite. 
\end{definition}
In what follows, $\| \cdot \|$ always denotes the norm defined in \Cref{d:X}. 

The next lemma, which shows that $F$ has a single positive eigenvector corresponding to a eigenvalue equal to its spectral radius, is shown in \Cref{s:krconclusion}.

\begin{lemma}\label{c:KR}
For every $s\in (0,1)$, $t>0$, $K>1$, and $E\in \R$, $F$ is a bounded operator $\mathcal X \rightarrow \mathcal X$, and there exists an eigenvector $v=v_{K,t,s,E}\in \mathcal X$ of $F$ whose corresponding eigenvalue $\lambda=\lambda_{K,t,s,E}$ is positive. Further, 
$v(x) > 0$ for all $ x\in \R$, $\lambda$ is the spectral radius of $F$, and $F$ has no other non-negative eigenvectors.
\end{lemma}
We often abbreviate $v_{K,t,s,E}$ by $v_{s,E}$,  or even just $v_E$, and similarly for $\lambda_{K,t,s,E}$ and $F_{K,t,s,E}$. When we refer to $v_E$, we always normalize it so that 
\be\label{e:vnormalization}
\| v_E \|_1 = \frac{1}{1-s}.
\ee

The following result connects $\phi$ and $\lambda_{s,E}$. It is proved in \Cref{s:provebootstrap}.

	\begin{lemma}\label{l:bootstrap}
The following holds for almost every $E \in \R$.
Let $\{\eta_j\}_{j=1}^\infty$ be a decreasing sequence such that $\lim_{j\rightarrow\infty} \eta_j =0$. Suppose that
$\lim_{j \rightarrow\infty} \Im  R_{00}(E + \iu \eta_j) = 0$ in probability.
Then
\bex
\lim_{j \rightarrow \infty} \phi(s; E+ \iu \eta_j) =\log K +  \log \lambda_{s,E}.
\eex
\end{lemma} 

In combination with \Cref{l:criteria}, the previous lemma suggests that regions where $\lambda_{s,E} < K^{-1}$ for all $s$ in some interval $(s_0,1)$ correspond to localization, and regions where $\lambda_{s,E} > K^{-1}$ for all $s$ in some interval $(s_0,1)$ correspond to delocalization. 
Additionally, from the statement of \Cref{t:main}, away from energies $E_0$ where $\rho(E_0) = (4g)^{-1}$, we expect delocalization when $\rho(E) >  (4g)^{-1}$ and localization when $\rho(E) < (4g)^{-1}$. The next lemma connects these two descriptions, and completely determines the phase diagram 
for all energies except those in small intervals bracketing the mobility edges. 
In particular, for $s$ sufficiently close to $1$, the eigenvalue $\lambda_{s,E}$ is greater than $K^{-1}$ when $\rho(E)$ is greater than $(4g)^{-1}$ and bounded away from it, and  $\lambda_{s,E}$ is less than $K^{-1}$ when $\rho(E)$ is less than $(4g)^{-1}$ and bounded away from it.

With the notation of \Cref{t:main}, we define the control parameter 
\be \label{e:varpi}
 \varpi = \frac{1}{10^3 \L^2}
 \ee 
 and set
\[
D_+ = \big\{ E \in [ - \L , \L] : \rho(E) \ge 1/(4g) + \varpi/\L^9  \big\}, \quad D_- = \big\{ E \in [ - \L , \L] : \rho(E) \le  1/(4g) - \varpi/\L^9 \big\}.
\]
We prove the following lemma in \Cref{s:testvectoranalysis}.
\begin{lemma}\label{l:bulklambda}
Adopt the notation and hypotheses of \Cref{t:main}. There exists $K_0(\L)>1$ such that the following holds for all $K \ge K_0$. There exists $s_0(K) \in (0,1)$ such that for all $s \in (s_0,1)$, we have $\lambda_{s,E} > (1 + \L^{-100})/K$ for all $E \in D_+$ and  $\lambda_{s,E} <  (1 - \L^{-100})/K$ for all $E \in D_-$.
\end{lemma}

The next result establishes the monotonicity of $\lambda_{s,E}$ in $E$ on the intervals not covered by the previous lemma. This is the most significant technical novelty of the paper, and the primary input in the proof of \Cref{t:main}. It is proved in \Cref{s:monotonicitysub}. To state it, we define, for any $E_0 \in \R$, 
 \be\label{e:istar}
  I_\star  = [E_0 - \varpi , E_0 + \varpi].
 \ee
\begin{theorem}\label{t:mainlambda}
For every real number $\L> 1$, there exists a constant $K_0 = K_0 (\L) > 1$ such that the following holds for every $\rho$ that is $\L$-regular. Fix $g \in \R$ such that 
\be
 \frac{1}{ 4 \| \rho \|_{\infty}} + \frac{1}{\L}  < g <\L,
\ee 
fix  $E_0 \in \R$ such that $|E_0|\le 2\L$ and $\rho(E_0) = (4g)^{-1}$, and set $I_\L  =  [E_0 - \L^{-1}, E_0 + \L^{-1}]$ and $t = g(K \log K)^{-1}$. 
Suppose that 
\be\label{e:geL}
\rho'(E)  \ge \frac{1}{\L}
\ee 
for all $E \in I_\L$, or 
\be\label{e:leL}
\rho'(E)  \le  - \frac{1}{\L}
\ee 
for all $E \in I_\L$. 
Then there exists a constant  $s_0(K, \L)\in(0,1)$ such that for all $s \in (s_0,1)$, the function $f_s(e)= \lambda_{K,t,s,e}$ satisfies the following property. 
If \eqref{e:geL} holds, then for all $e_1, e_2 \in I_\star$ such that $e_1 < e_2$, 
\be
f_s(e_2) \ge f_s(e_1), \qquad \liminf_{s\rightarrow 1^-} \big( f_s(e_2) - f_s(e_1) \big)> 0.
\ee
Further, if \eqref{e:leL} holds  for all  $e_1, e_2 \in I_\star$ such that $e_1 < e_2$, 
\be
f_s(e_1) \ge f_s(e_2), \qquad \liminf_{s\rightarrow 1^-}\big( f_s(e_1) - f_s(e_2)  \big)> 0.
\ee
\end{theorem}

\subsection{Bootstrap for the Localized Regime}\label{s:bootstrapintro}
Our proof of the claim about spectral localization in \Cref{t:main} is facilitated through two distinct, but related, inductive arguments, which treat the bounded and unbounded components of $\mathcal M^\circ$ differently; we denote these components by $\mathcal M^\circ_1, \dots , \mathcal M^\circ_{n+1}$. For energies $E$ in a bounded component $\mathcal M^\circ_i$  where localization should occur, we first prove the claim for energies in $\mathcal M^\circ_i \cap D_-$, then extend it inductively to all energies in $\mathcal M^\circ_i$. The following lemma provides the initial bound for the induction. It is proven in \Cref{s:disorderbootstrap}. 

\begin{lemma}\label{l:bulkphi}
Adopt the notation and hypotheses of \Cref{t:main}. There exists $K_0(\L)>1$ such that the following holds for all $K \ge K_0$. For almost all $E \in D_-$, we have 
\be \limsup_{\eta \rightarrow 0} \phi(1;E+ \iu\eta) < 0.
\ee
\end{lemma}

The next lemma states the continuity lemma that we use for the induction. We prove it in \Cref{s:continuitylemmaA2}. 
\begin{lemma}\label{l:phicontinuity}
	Fix $s \in (0,1)$,
	$\kappa, \omega >0$,  and a compact interval $I \subset \bbR$. There exists a constant $\delta(s,\kappa, \omega, I) >0$ such that the following holds. For every $E_1, E_2 \in I$ such that $|E_1 - E_2 | \le \delta$ and 
\bex
\limsup_{\eta \rightarrow 0} \varphi (s; E_1 + \mathrm{i} \eta) < - \kappa,
\eex
we have	
\bex
\limsup_{\eta \rightarrow 0} \varphi (s; E_2 + \mathrm{i} \eta) 
< 
-\kappa + \omega.
\eex
\end{lemma}

For the unbounded components of $\mathcal M^\circ$, we use a similar induction argument with the base case provided by the following lemma, which asserts that $\phi$ is negative for sufficiently large $|E|$. It is shown in \Cref{s:bigEproof}. 
\begin{lemma}\label{l:bigE}
For all $K>1$,  there exists $\mathfrak B(K) >0$ such that for all $E \in \R$ satisfying $|E| > \mathfrak B$, we have 
\begin{equation}
\limsup_{\eta \rightarrow 0^+} \phi(1; \eta) < - 1.
\end{equation}
\end{lemma}

\subsection{Proof of Main Result} \label{s:mainproofintro}

We now combine the previous results of this section to prove \Cref{t:main}.

\begin{proof}[Proof of \Cref{t:main}]
If necessary, we begin by increasing the value of $\L$ so that $\mathcal M \subset [-\L/2, \L/2]$. 
Note that all hypotheses of the theorem become weaker as $\L$ increases. Indeed, if they are satisfied for a given density $\rho$ and parameter $\L >0$, they are satisfied for $\rho$ and $\L'$ if $\L' > \L$; in particular, if $\rho$ is $\L$-regular, then it is also $\L'$-regular, by \Cref{d:Lregular}. 
Further, none of the conclusions get weaker when increase $\L$ to $\L'$. Hence, increasing the value of $\L$ is permissible. 

Throughout this proof, we let $s\in (0,1)$ be a parameter. We consider only values of $s$ such that $s \in (s_0(K), 1)$, where $s_0$ is taken large enough so that the conclusions of \Cref{l:bulklambda} and \Cref{t:mainlambda} hold. 
Set $f_s(E)= \lambda_{K,t,s,E}$. We define 
\begin{align*}
\begin{split}
\mathcal A_- &=\left \{E \in [-\L, \L] :  \liminf_{s \rightarrow 1^-} f_s(E) < K^{-1} \right\},\\
\mathcal A_+ &=\left \{E \in [-\L, \L] :  \liminf_{s \rightarrow 1^-} f_s(E) >K^{-1} \right\},\\
D_0 &= [ -\L, \L] \setminus \big( \interior(\mathcal A_-) \cup  \interior(\mathcal A_+)  \big),
\end{split}
\end{align*}
where $\interior(S)$ denotes the interior of a set $S \subset \R$.  
Additionally, we define
\[
\widetilde D_+ = \big\{ E \in [ - \L , \L] : \rho(E) \ge 1/(4g) + \varpi/\L^8  \big\}, \quad \widetilde D_- = \big\{ E \in [ - \L , \L] : \rho(E) \le  1/(4g) -  \varpi/\L^8  \big\}.
\]
We note that by \Cref{l:bulklambda}, we have $\widetilde D_+  \subset \interior ( \mathcal A_+)$, so $D_0 \cap \widetilde D_+ = \varnothing$. Similarly, $D_0 \cap \widetilde D_- = \varnothing$. Then
\be
[-\L, \L] \setminus \big( \widetilde D_+ \cup \widetilde D_-) \subset \bigcup_{\mathfrak E \in \mathcal E} [\mathfrak E - \varpi, \mathfrak E + \varpi],
\ee
since $\rho$ is $\L$-Lipschitz (by the third part of \Cref{d:Lregular}). This implies that 
\be
D_0
\subset \bigcup_{\mathfrak E \in \mathcal E} [\mathfrak E - \varpi, \mathfrak E + \varpi].
\ee 

Since $ \interior(\mathcal A_-) $ and $ \interior(\mathcal A_+)$ are open, for any $\mathfrak E \in \mathcal E$, the set 
\be
 [\mathfrak E - \varpi,  \mathfrak E + \varpi ] \cap \big( \interior(\mathcal A_-) \cup  \interior(\mathcal A_+)  \big)
\ee 
must have the form 
\be 
[\mathfrak E - \varpi , a ) \cup (b, \mathfrak E + \varpi].
\ee 
for some $a,b \in \R$ such that $a\le b$, by the definitions of $\mathcal A_-$ and $\mathcal A_+$ and the monotonicity of $f_s(E)$ on $ [\mathfrak E - \varpi, \mathfrak E + \varpi] $ provided by \Cref{t:mainlambda}. Further, we must have $a=b$, or $ \liminf_{s \rightarrow 1^-} f_s(E)$ would be equal to $K^{-1}$ on the open interval $(a,b)$, contradicting \Cref{t:mainlambda}. Hence $D_0$ consists of a set of isolated points, with exactly one point in each interval $[\mathfrak E - \varpi , \mathfrak E + \varpi ]$. Setting $\mathcal M = D_0$, this establishes the first claim of the theorem.

We note that the previous reasoning establishes that $\mathcal M^\circ$ is a disjoint union of a finite number of connected components. 

By increasing $\L$ to some value $\L'$, using \Cref{l:bulklambda},  and recalling that $\rho$ is $\L$-Lipschitz (and satisfies the assumptions of the theorem, so that it is monotonic at all points in $\mathcal E$), we may suppose that on components $I$ belonging to $\mathcal A_+$, we have that $I\cap D_+$ is a nonempty interval. Similarly, we may suppose that on components $I$ belonging to  $\mathcal A_-$, we have that $I \cap D_-$ is a nonempty interval.

For the second claim of the theorem, we begin by considering the bounded components of $\mathcal M^\circ$. We first consider the assertion regarding absolutely continuous spectrum.
Let $Y\subset \R$ be the measure zero set on which the first part of \Cref{l:criteria} does not hold, or the limit \eqref{e:stieltjeslimit} does not exist, or the limit defining $\phi(1;E)$ does not exist. 
Fix a bounded, connected component $I\subset \mathcal M^\circ$ for which $\rho'(\ell) > 0$. By the second part of \Cref{l:criteria}, it suffices to show that $\Im R_{00}(E_0) > 0$ with positive probability for every $E_0 \in I \setminus Y$. 
Suppose, for the sake of contradiction, that there exists $ E_0 \in I \setminus Y$ such that $\Im R_{00}(E_0) = 0$ almost surely.
By the definition of $R_{00}(E_0)$, we have
\be\label{contradict2}
\lim_{\eta \rightarrow 0} \Im  R_{00}(E_0 + \iu \eta) = 0
\ee
in probability. 
From the definition of $\phi$ from \eqref{szl}, and \Cref{l:bootstrap}, we have
\be\label{e:phismainproof}
\phi(s; E_0) =  \lim_{\eta\rightarrow 0}  \phi(s; E_0+ \iu \eta) = \ln K + \log \lambda_{s,E_0}
\ee
for all $s \in (s_0, 1)$. Note that $I$  lies in $\mathcal A_+$, by the assumption that $\rho'(\ell)>0$, \Cref{t:mainlambda}, and the fact that $\mathcal M$ consists of isolated points. 
Then $E_0 \in \mathcal A_+$, and by the definition of $\mathcal A_+$, 
\[
\liminf_{s \rightarrow 1^{-} } \ln \lambda_{s,E_0} > - \ln K.
\]
This implies that $\phi(1; E) > 0$ after taking $s \rightarrow 1$ in \eqref{e:phismainproof}. 
Then by the first part of \Cref{l:criteria}, we have $\Im R_{00} (E_0)>0$ almost surely, contradicting \eqref{contradict2}.  After using the second part of \Cref{l:criteria}, this completes the proof of the conclusion regarding absolutely continuous spectrum for the bounded components of $\mathcal M^\circ$.

We now consider the bounded components of $\mathcal M^\circ$ where $\rho'(\ell) < 0$. Fix such an open interval $I$, and note that $I \subset \mathcal A_-$ (by \Cref{t:mainlambda}). 
For any $s \in (s_0 ,1)$, set
\be\label{e:provided}
I_{\mathrm{loc},s } =  \{ E \in I : \lambda_{s,E}  <  K^{-1} \}.
\ee
By \Cref{l:bulklambda}, for all all $s \in ( s_0(K), 1)$, we have $D_- \cap I \subset I_{\mathrm{loc},s }$.  

We next establish that for every $E\in I \setminus Y$, we have $\phi(1;E) < 0$.
By \Cref{l:bulkphi}, we have $\limsup_{s\rightarrow 1^{-}} \phi(s;E)  < 0$ for all $E \in D_-$ so it remains to consider energies in $I \setminus  D_-$. 
Fix $E' \in I \setminus ( D_- \cup Y)$ and $s \in (s_0, 1)$ such that $E' \in I_{\mathrm{loc},s }$; this choice of $s$ is possible by the definition of $\mathcal A_-$. 
We recall that we may suppose that $ I \cap D_-$ is an interval. Then, without loss of generality, we assume that $E' < \inf ( D_- \cap I )$. Let $E_0 = \inf  (D_- \cap I )$.   

By the definition of $I_{\mathrm{loc},s }$, we have $\lambda_{s,E'} < K^{-1}$.
Set
\bex
\kappa' = \kappa'(E',s) = \inf_{E \in [E', E_0]} \big( 1 - K  \lambda_{s,E} \big),
\qquad \kappa = \kappa(E',s) = \frac{\min ( \kappa', 1  )}{2}.
\eex
Recall that $\ell = \inf I$. Since $E \mapsto \lambda_{s,E}$ monotonic and strictly less than $K^{-1}$ on $(\ell, \ell + \varpi]$, by \Cref{t:mainlambda}, we have $\kappa' >0$. (We are also using that $\inf D_-  \le  \ell + \varpi$, by the definition of $D_-$ and the Lipschitz continuity of $\rho$ provided by the third part of  \Cref{d:Lregular}.)

Using \Cref{l:phicontinuity} on the interval $[E', E_0]$, with $\omega$ in that statement equal to $\kappa$, let $\delta$ be the constant given by \Cref{l:phicontinuity}.  
Let $\{E_i\}_{i=1}^M$ be a set of real numbers such that $E_M = E'$ and 
\bex
 0<  E_{i} - E_{i+1}  < \delta
\eex
for all $i \in \unn{0}{M-1}$. 
We claim that for every $i\in \unn{0}{M}$, we have 
\be\label{bob}
\limsup_{\eta \rightarrow 0} \exp\big(\phi(s; E_i + \iu \eta)\big) =  K \lambda_{s,E_i} .
\ee
We will show the claim by induction on $i$. The base case $i=0$  follows from our assumption that $E_0 \in D_-$,  \Cref{l:bulkphi}, and \Cref{l:bootstrap}.  
Next, for the induction step, we assume that the induction hypothesis holds for some $i\in \unn{0}{M-1}$, and we will show it holds for $i+1$. Using the induction hypothesis at $i$ and the definition of $\kappa$, we have 
\bex
\limsup_{\eta \rightarrow 0} \exp\big(\phi(s; E_i + \iu \eta)\big)  =  K \lambda_{s,E_i} \le 1 - 2 \kappa.
\eex
By \Cref{l:phicontinuity} and the definition of $\delta$, 
\begin{align}\label{p19}
\limsup_{\eta \rightarrow 0} \varphi (s; E_{i+1} + \mathrm{i} \eta)
&< \limsup_{\eta \rightarrow 0} \varphi (s; E_i + \mathrm{i} \eta) + \kappa\le  \log (1 - 2 \kappa)  +  \kappa \le - \kappa,
\end{align}
where the last inequality follows from the elementary bound $\ln(1-x) < -x$. 
Using \Cref{p:imvanish}, we see that \eqref{p19} implies
\be\label{jcvg}
\lim_{\eta \rightarrow 0} \Im R_{00}(E_{i+1} + \iu \eta) = 0.
\ee
Let $\{\eta_j\}_{j=1}^\infty$ be a sequence such that
\be\label{otter}
\lim_{j \rightarrow \infty}
\exp \left( \phi(s; E_{i+1} + \iu \eta_j) \right)= \limsup_{\eta\rightarrow 0} \exp \left( \phi(s; E_{i+1} + \iu \eta) \right).
\ee
Then using \eqref{jcvg} and \Cref{l:bootstrap}, equation \eqref{otter} implies that 
\bex
\limsup_{\eta \rightarrow 0} \exp \big( \phi(s; E_{i+1} + \iu \eta) \big) = \lim_{j \rightarrow \infty}
\exp \big( \phi(s; E_{i+1} + \iu \eta_j) \big) = K \lambda_{E_{i+1},s}.   
\eex 
This completes the induction step and shows that \eqref{bob} holds for all $i\in \unn{0}{M}$. In particular, taking $i=M$, we have 
\be
\limsup_{\eta \rightarrow 0} \exp\big(\phi(s; E' + \iu \eta)\big) \le  K \lambda_{E',s}.
\ee
By \Cref{p:imvanish}, and the fact that $\lambda_{E', s} < K^{-1}$ (as $E' \in I_{\mathrm{loc},s }$),  this implies $ \lim_{\eta \rightarrow 0} \Im R_{00} (E' + \iu \eta) = 0$.
Then by \Cref{l:bootstrap}, we have 
\be\label{e:takeliminf}
\phi(s ;E') = 
 \lim_{\eta \rightarrow 0} \phi(s; E' + \iu \eta ) = \ln K + \lambda_{s,E'}.\ee 
Since $\phi(s;E')$ is nonincreasing in $s$ by the first part of \Cref{l:aizenman}, taking the limit infimum as $s \rightarrow 1^{-}$ in \eqref{e:takeliminf}  and recalling  $E' \in \mathcal A_-$ 
gives $\phi(1 ; E ) < 0$. 

Recall that the choice of $E' \in I \setminus (D_- \cup Y)$ was arbitrary (and we already dealt with the case of $E' \in D_-$). We conclude that for every $E \in I \setminus Y$, we have $\phi(1 ; E ) < 0$.
We deduce from the third part of \Cref{l:criteria} that the spectrum of $\hamiltonian$ is almost surely pure-point on $I$. This concludes the proof for the bounded components of $\mathcal M^\circ$.

It remains to handle the unbounded components of $\mathcal M^\circ$. The proof here uses a bootstrap argument similar to the proof for bounded components, except \Cref{l:bigE} is used instead of \Cref{l:bulkphi} for the starting point of the induction. Since the details are virtually identical except for this change, we omit them. This completes the proof of the theorem.
\end{proof}
\section{Miscellaneous Preliminaries}\label{s:misc}

This section provides a collection of miscellaneous identities and estimates that will be used later. 
In \Cref{s:treenot}, we provide some fundamental notation for vertices and paths on $\mathbb{T}$. In \Cref{s:resolventid}, we state a variety of identities and estimates concerning the resolvent $\boldsymbol{R}$ and analogous resolvent operators on subgraphs of $\mathbb{T}$. \Cref{s:integralbounds} records two integral bounds that will be used in the proofs for results in this section and the next.  \Cref{s:prelimself} contains estimates on the densities $p_E$ and  $\rho_E$ defined in \eqref{s:rez}. 
\subsection{Tree Notation}\label{s:treenot}
We start by defining notation for the tree $\mathbb{T}$. Recall that $0$ denotes the root of $\mathbb{T}$. The \emph{length} of any vertex $v \in \bbV$ is its distance from the root under the standard geodesic metric.  We let $\mathbb{V} (\ell)$ denote the set of vertices of length $\ell$. We write $v \sim w$ if there is an edge between $v, w \in \mathbb{V}$. The \emph{parent} of any vertex $v \in \mathbb{V}$ that is not the root is the unique vertex $v_- \in \mathbb{V}$ such that $v_- \sim v$ and $\ell (v_-) = \ell (v) - 1$.\footnote{If $v = 0$, then its parent $0_-$ is defined to be the empty set.} 
 A child of $v$ is any vertex $w \in \mathbb{V}$ whose parent is $v$. We let $\mathbb{D} (v) \subset \mathbb{V}$ denote the set of children of $v$. 
		
	A \emph{path} on $\mathbb{T}$ is a sequence of vertices $\mathfrak{p} = (v_0, v_1, \ldots , v_k)$ such that $v_i$ is the parent of $v_{i + 1}$ for each $i \in \unn{0}{k - 1}$. 
	We write $v \preceq w$ (equivalently, $w \succeq v$) if there exists a path (which may be empty) with starting and ending vertices $v$ and $w$, respectively. We  write $v \prec w$ (equivalently, $w \succ v$) if $v \preceq w$ and $v \ne w$. For any integer $k \ge 0$ and vertex $v \in \mathbb{V}$ with $\ell (v) = \ell$, we additionally let $\mathbb{D}_k (v) = \big\{ u \in \mathbb{V} (\ell + k) : v \preceq u \big\}$. Observe in particular that $\mathbb{D}_1 (v) = \mathbb{D} (v)$ and $\mathbb D_\ell(0) = \mathbb{V}(\ell)$.
\subsection{Identities for Resolvent Entries} \label{s:resolventid}

For any subset of vertices $\mathcal{U} \subseteq \mathbb{V}$, we let $\hamiltonian^{(\mathcal{U})}$ denote the operator obtained from $\hamiltonian$ by setting all entries corresponding to a vertex in $\mathcal{U}$ to $0$. More specifically, we let $\hamiltonian^{(\mathcal{U})} : L^2 (\mathbb{V}) \rightarrow L^2 (\mathbb{V})$ denote the self-adjoint operator defined by first setting, for any vertices $v, w \in \mathbb{V}$,
		\begin{flalign*}
			 \big\langle \delta_v, \hamiltonian^{(\mathcal{U})} \delta_w \big\rangle =\big\langle \delta_v, \hamiltonian \delta_w \big\rangle  , \quad \text{if $v, w \notin \mathcal{U}$}; \qquad  
			 \big\langle \delta_v, \hamiltonian^{(\mathcal{U})} \delta_w \big\rangle = 0, \qquad \text{otherwise};
		\end{flalign*}
	
		\noindent and extending to $L^2 (\mathbb{V})$ by linearity and using the density of finitely supported vectors in $L^2(\mathbb V)$. For any complex number $z \in \mathbb{H}$, we  denote the associated resolvent operator $\boldsymbol{R}^{(\mathcal{U})} : L^2 (\mathbb{V}) \rightarrow L^2 (\mathbb{V})$ and its entries by 
		\begin{flalign*} 
			\boldsymbol{R}^{(\mathcal{U})} = \boldsymbol{R}^{(\mathcal{U})} (z) = \big( \hamiltonian^{(\mathcal{U})} - z \big)^{-1}; \qquad R_{vw}^{(\mathcal{U})} = R_{vw}^{(\mathcal{U})} (z) = \big\langle \delta_v, \boldsymbol{R}^{(\mathcal{U})} \delta_w \big\rangle, \qquad \text{for any $v, w \in \mathbb{V}$}.
		\end{flalign*}

		The following lemma provides identities and estimates on the entries of $\boldsymbol{R}^{(\mathcal{U})}$. The first statement is \cite[Proposition 2.1]{klein1998extended}, the second is a consequence of the first, and the third is \cite[Equation (3.37)]{aizenman2013resonant}.

\label{EquationsResolvent}
		\begin{lemma}[{\cite{klein1998extended,aizenman2013resonant}}]
		
		\label{q12} 
		
		Fix $z = E + \mathrm{i} \eta \in \mathbb{H}$ and $\mathcal{U} \subset \mathbb{V}$.
		
		\begin{enumerate}
			\item For any vertex $v \in \mathbb{V}$, we have the \emph{Schur complement identity}
			\begin{flalign}
				\label{qvv}
				R_{vv}^{(\mathcal{U})} = R_{vv}^{(\mathcal{U})} (z) = - \Bigg( z + \displaystyle\sum_{w \sim v}t^2  R_{ww}^{(\mathcal{U}, v)} \Bigg)^{-1}.
			\end{flalign}
				
			\noindent Moreover, $R_{ww}^{(v)} (z)$ has the same law as $R_{00} (z)$, for any $w \in \mathbb{D}(v)$.
			
			\item For any vertices $v, w \in \mathbb{V}$, we deterministically have 
			\begin{flalign*} 
				\big| R_{vw}^{(\mathcal{U})} (z) \big| \le \eta^{-1}.
			\end{flalign*} 
		
			\item For any vertex $v \in \mathbb{V}$, we have the \emph{Ward identity} 
			\begin{flalign}
				\label{sumrvweta} 
				\displaystyle\sum_{w \in \bbV \setminus \mathcal{U}} \big| R_{vw}^{(\mathcal{U})} \big|^2 = \eta^{-1} \cdot \Imaginary R_{vv}^{(\mathcal{U})}.
			\end{flalign}
		
		\end{enumerate}
		
		\end{lemma} 

We also require the following product expansion, which appears as \cite[(3.4)]{aizenman2013resonant}.
	\begin{lemma}[{\cite[Equation (3.4)]{aizenman2013resonant}}]
		
		\label{rproduct} 
		
		Fix $v, w \in \mathbb{V}$ with $v \preceq w$ that are connected by a path $\mathfrak{p}$ (with $v$ as the starting vertex and $w$ as the ending vertex) consisting of $m+1$ vertices. For any subset $\mathcal{U} \subset \mathbb{V}$, we have
		\begin{flalign*}
			R_{vw}^{(\mathcal{U})} = (-1)^m \cdot R_{vv}^{(\mathcal{U})} \cdot \displaystyle\prod_{v \prec u \preceq w} t R_{uu}^{(\mathcal{U}, u_-)} = (-1)^m \cdot R_{ww}^{(\mathcal{U})} \cdot \displaystyle\prod_{v \preceq u \prec w} t  R_{uu}^{(\mathcal{U}, u_+)}.
		\end{flalign*}
	
	\end{lemma}

The following fractional moment bound for off-diagonal resolvent entries is \cite[(12)]{bapst2014large}.

\begin{lemma}[{\cite[Equation (12)]{bapst2014large}}]\label{l:uniformlyintegrable}
Fix an integer $L \ge 0$, and let $v_0, v_1, \dots, v_L$ be a path of vertices in $\bbV$ with $v_0 = 0$ and $v_j \in \bbV(j)$ for $j \in \unn{1}{L}$. Then for all $s , t\in (0,1)$ and $z \in \bbH$, 
\begin{equation}
\E\Big[ \big|R_{0v_L }(z) \big|^s \; \Big| \; \{V_w\}_{w \notin \{v_0, \dots , v_L\}} \Big] 
\le \left( \frac{ 2^s \| \rho \|_\infty^s }{1-s} \right)^{L+1} t^{sL}.
\end{equation}
\end{lemma}

\subsection{Integral Bounds}\label{s:integralbounds}
We now state two useful integral bounds. \Cref{l:32} is used in the proof of \Cref{l:rhoebound} later in this section, and both bounds are used in \Cref{s:KRapp}. Their proofs are given in \Cref{s:integralboundproofs}. 
\begin{lemma}\label{l:32}
For all $s \in [0,1]$ and $A\in \R$, define
\be
I(A) = 
\int_{-\infty}^\infty \frac{dx }{(1 + (x-A)^2)(1+|x|^{2-s} )}.
\ee
There exists a  constant $C>0$ (not depending on $s$ or $A$) such that 
\be \label{e:desiredA} 
I(A)
\le C  \min \big( 1, |A|^{-(2-s)} \big).
\ee 
\end{lemma}

\begin{lemma}\label{l:integrallowerbound}
For all $s \in (0,1)$, $A\in \R$, and $B>0$, set $J(B) = [-B, B]^c$ and define
\be
I^\circ(A,B) = 
\int_{J(B)} \frac{dx }{(1 + (x-A)^2) |x|^{2-s} }.
\ee
There exists a  constant $c(B)>0$ such that for all $A\in \R$ and $s\in (0,1)$, 
\be  \frac{c}{1 + |A|^{2-s}} \le 
I^\circ(A, B).
\ee 
\end{lemma}

\subsection{Preliminary Estimates on the Self-Energy}\label{s:prelimself}
We now collect some useful results on the densities $p_E$, $p_E^{(M)}$, and $\rho_E$, which were defined in \Cref{s:rez}. Our first lemma follows immediately from the fact that $\rho_E$ is a sum of independent random variables, and hence can be represented as a convolution of probability densities; a corresponding representation for $p_E$ follows by a change of variables from its definition. Therefore, we omit the proof.
\begin{lemma}
For all $K > 1$ and all $E \in \R$, we have 
\be
\rho_E(x) = \int_{-\infty}^\infty  \label{e:bananav3}
p^{(K-1)}_{E}(y) \rho(x + y +E) \, dy,
\ee
and 
\be\label{e:loweraux1}
p_E(x) = \frac{1}{x^2} \int_{-\infty}^\infty
p^{(K)}_E (y) \rho\left( \frac{1}{x} + y +E \right) \,dy .
\ee 
\end{lemma}

We now collect several useful facts on the densities $p_E$ and $p_E^{(M)}$ that were shown in \cite{bapst2014large}. We recall our standing assumption that the density $\rho$ of the nonzero entries of $\V$ is $\L$-regular in the sense of \Cref{d:Lregular}. The first part of the lemma was shown in \cite[(A10)]{bapst2014large}. 
The second part is \cite[(A2)]{bapst2014large}. 
The third part follows from  \cite[(A7)]{bapst2014large}.\footnote{The assumption that $\rho$ is $\L$-regular allows us to take $\varsigma=1$ in that reference.}
The fourth part is a straightforward modification of \cite[(A4)] {bapst2014large}. 
The fifth part is shown in \cite{bapst2014large}  in the text following \cite[(A9)]{bapst2014large}.
\begin{lemma}[{\cite[Appendix A]{bapst2014large}}]\label{l:bapstcollection}
The following statements hold for all integers $K > 1$.
\begin{enumerate}
\item For all $E\in \R$, the function $\rho_E$ is Lipschitz on $\R$ with Lipschitz constant $\L$. 
\item For all $x, E \in \R$, we have 
\be\label{e:bapsta2}
p_{E}(x) \le \frac{ \| \rho \|_\infty}{x^2}.
\ee 
\item  For every $v >0$, there exists $C(v)>0$ such that for all $E \in [-v, v]$ and $x\in [-1,1]$, we have 
\be\label{e:bapsta7}
p_{E}(x) \le C ( 1 + K^3 t^2\| \rho\|_\infty). 
\ee
\item For all $M \in \Zplus$, $k\in \unn{0}{M}$ and $x, E_1, E_2 \in \R$, we have 
\begin{equation}\label{e:bapsta}
 \big(p_{E_1} ^{(k)} * p_{E_2} ^{(M-k)}\big)(x)  \le  \frac{M^3 t^2 \| \rho \|_\infty }{|x|^2}.
 \end{equation}
\item For all  $v >0$, we have 
\be\label{e:pelow}
\inf_{x \in [- v, v ] } p_E (x) > 0.
\ee 
\end{enumerate}
\end{lemma}

The following lemma provides a quantitative approximation of $\rho_E(x)$ by $\rho(x+E)$. It was shown  in  \cite[(A21)]{bapst2014large}. 

\begin{lemma}[{\cite[Equation (A21)]{bapst2014large}}]\label{l:bapstuniform}
Fix $G, \delta, \epsilon>0$. There exists $K_0(G,\delta, \epsilon) >0$ such that for all $K \ge K_0$, $g \in [ 0, G]$, and $e, E \in \R$, 
\be
\inf_{e' \in [-\delta, \delta]} \rho(e + e' + E) - \eps \le \rho_E (e) \le \sup_{e' \in [-\delta, \delta]}  \rho(e + e' + E) + \eps
\ee 
\end{lemma}

We now provide upper and lower bounds on the densities $\rho_E$ and $p_E$. The proof of these bounds is deferred to \Cref{s:densityproof}. 
\begin{lemma}\label{l:rhoebound}
For all $E \in \R$ and $K>1$, there exist constants $C(E,K),c(E,K)>0$ such that the following claims hold.
\begin{enumerate}
\item For all $x \in \R$, 
\be\label{e:pEquadratic}
 \frac {c }{ 1 + |x|^2}\le  p_E(x) \le \frac {C }{ 1 + |x|^2},
\ee 
\item For all $x \in \R$,
\be\label{e:rhoequadratic}
\frac{ c }{1 + | x  + E  |^2}
\le 
\rho_E(x) \le \frac{ C }{1 + | x  + E  |^2},
\ee 
\item We have 
\be\label{e:rhoEupper}
\| \rho_E \|_\infty \le \| \rho \|_\infty.
\ee 
\end{enumerate}
\end{lemma}

\section{Origin of the Transfer Operator}\label{s:origintransfer}

The goal of this section is to explain the origin of the operator $F$ defined in \eqref{e:Fdef}, and how it can be used to compute the quantity $\Phi_L$ (from \eqref{sumsz1}) as $\Im z \rightarrow 0^+$, under the assumption that the imaginary part of the boundary value $R_{00}(E)$ vanishes.
We begin in \Cref{s:productexpansion} by deriving a useful product expansion for $\Phi_L (s; z)$. We take its limit as $\Imaginary z$ tends to $0$ in \Cref{EtaSmall}, and we show in \Cref{s:transferoperator} that this quantity can be written in terms of $F$.

\subsection{Product Expansion}\label{s:productexpansion}

Throughout this section, we fix a complex number ${z = E + \mathrm{i} \eta} \in \mathbb{H}$ and an integer $L \ge 1$.  We typically abbreviate $R_{vw} = R_{vw} (z)$ for $v, w \in \mathbb{V}$ and $R_{00}=R_{00}(z)$. We let $W = W(z)$ denote a complex random variable with law $\rho_z$ (as defined in \Cref{s:rez}). 
For each integer $j \in \unn{0}{L}$, let $W_j = W_j (z)$ denote a complex random variable with law $W$, with the variables $( W_j )_{j=0}^L$ mutually independent. We also let $t \in (0,1)$ be a free parameter; our results do not depend on the choice of $t$. 

Next, we fix some infinite path $\mathfrak{p} = (v_0, v_1, \ldots ) \in \mathbb{V}$ with $v_0 = 0$. The following definition recursively introduces a set of random variables $\{ R_i(z ; L; \mathfrak{p} )\}_{i=0}^L$. We often abbreviate $R_i(z ; L)  = R_i (z; L;  \mathfrak{p})$.

\begin{definition} 
	
\label{srz} 

For $i \in \unn{0}{L+1}$, we define the random variables 
$R_i(z;L)$ recursively, by setting 
\begin{flalign*}
R_{L+1} (z; L) = R_{v_{L+1} v_{L+1}}^{(v_{L})}, \quad \text{and} \quad R_i(z;L) =  \big( W_{i}(z) - t^2 R_{i+1}(z ; L) \big)^{-1}, \quad \text{for $i \in \unn{0}{L}$}.
\end{flalign*}
\end{definition} 

The  definition of $R_i$ is analogous in form to the Schur complement identity stated in \eqref{qvv}.  The term in the denominator represented here by $t^2 R_{i+1} (z; L)$ has been written separately from the rest of the sum, which is represented by $W_i(z)$.

\begin{lemma}\label{l:expansionz}

For all $v \in \mathbb{V}_L$ and $z \in \bbH$, 
\be\label{productN}
\E\big[ |R_{0v}|^s \big] 
=
\E \Bigg[
\big| R_L (z; L ) \big|^s \cdot 
\prod_{j=1}^{L}
|t|^s \big| R_{j-1}(z; L) \big|^s \Bigg] 
.
\ee
\end{lemma}

\begin{proof}
We may suppose without loss of generality that $v \in \mathfrak{p}$. 
From
\Cref{rproduct}, we obtain for any $i$ that
\[
\big| R^{(v_{i-1})}_{v_i v } \big| = \big| R^{(v_{i-1})}_{v_i v_i} \big| \cdot |t| \cdot \big| R^{(v_i)}_{v_{i+1} v} \big|.
\]
Iterating this expansion yields 
\[
\E\big[ |R_{0v}|^s \big] = \E\left[ |R_{00}|^s  \prod_{i=1}^L | t |^s |R^{(v_{i-1})}_{v_{i}v_i}|^s \right]
= 
\E\left[ |R_{v_L v_L}^{(v_{L-1})} |^s  \prod_{i=1}^{L} |t |^s |R^{(v_{i-2})}_{v_{i-1}v_{i-1}}|^s \right].
\]
This completes the proof after using the definition of the $R_j$ and expanding each resolvent entry $R_{v_i v_i}^{(v_{i-1})}$ using \eqref{qvv} to see that the joint distributions of $(R_j)_{j=0}^L$ and $(R^{(v_{i-1})}_{v_i v_i})_{i=0}^L$ are identical. \end{proof}

\subsection{The Small $\eta$ Limit of the Product Expansion} \label{EtaSmall}

In this section we analyze the limit of $\E[ |R_{0v}|^s] $ as $\Im z$ tends to $0$. To that end, we adopt the notation of \Cref{s:productexpansion}.  We start by fixing a complex number $r_{L+1} \in \overline{\mathbb{H}}$ and a sequence of complex numbers $\bm{w} = (w_0,w_1, \ldots , w_{L}) \in \overline{\mathbb{H}}^{L+1}$.  We define quantities $r_i \in \overline{\mathbb{H}}$ recursively for $i\in \unn{0}{L}$ by
\be\label{ui}
r_i =  \big( w_{i}  -  t^2 r_{i+1} \big)^{-1}.
\ee
Observe that $r_i$ is a function of the parameters $ \{ r_{L+1} \} \cup \{ w_{i}, \dots, w_{L} \}$ only. We additionally set
\bex
T(r_{L+1}, \bm{w})=  \left( \prod_{i=0}^L
  \frac{t}{w_i - t^2 r_{i+1}}  \right) .
\eex
Together with \eqref{productN} and the definition of $R_i (z; L)$ from \Cref{srz}, we have for any $z \in \bbH$ that
\begin{align}\label{productN3}
\begin{split}
\E\big[ |R_{0v_L}|^s \big] 
&=
\E \Bigg[
\big| R_L (z; L) \big|^s \cdot 
\prod_{j=1}^{L}
|t |^s \big| R_{j-1}(z; L) \big|^s \Bigg] \cdot
 \\
&=
\frac{1}{t^s} 
\int_{\mathbb{C}^{L+2}}
\big|T(r_{L+1} , \bm{w})\big|^s
p_z(r_{L+1} ) \, d r_{L+1} \prod_{i=0}^{L} \rho_z(w_i) \,d w_i  
 .
\end{split}
\end{align}

We next introduce an analog of the right side of \eqref{productN3}, defined on the real line. 

\begin{definition} 
	
	\label{amoment2} 
	
	For any  $E \in \bbR$, define 
\begin{flalign}\label{productN2}
\Upsilon_L(s;E) =\frac{1}{t^s} 
\int_{\mathbb{C}^{L+2}}
\big|T(r_{L+1} , \bm{w})\big|^s
p_E(r_{L+1}) \, d r_{L+1} \prod_{i=0}^{L} \rho_E(w_i) \,d w_i  . \end{flalign}
\end{definition}

The following lemma shows that $\Upsilon_L$ can be interpreted as a boundary value of \eqref{productN3}, supposing that the there is a subsequence $(\eta_j)_{j=1}^\infty$ tending to zero along which $R_{00}(E+\eta_j)$ has a purely imaginary limit (which is suggestive of spectral localization). 

\begin{lemma}\label{l:PhiEA}
Fix $s\in (0,1)$, $E \in \bbR$, and an integer $L \ge 1$.
Suppose there exists a decreasing sequence of positive real numbers $(\eta_1, \eta_2, \ldots ) \subset (0, 1)$ tending to $0$ and an almost surely real random variable $R_{00} (E)$ such that $\lim_{k \rightarrow \infty} R_{00}(E + \iu \eta_k) = R_{00}(E)$ weakly. Then, for any $v_L \in \bbV(L)$, 
\bex
\lim_{k \rightarrow \infty} \E\Big[ \big|R_{0v_L} (E + \iu \eta_k) \big|^s \Big]  =\Upsilon_L(s ; E ).
\eex
\end{lemma}
\begin{proof}
 Setting $z_k = E + \mathrm{i} \eta$, we have $\lim_{k \rightarrow \infty} \rho_{z_k}(u) \, du  =  \rho_{E}(u) \, du$ from the hypothesis that $\lim_{k \rightarrow \infty} R_{00}(E + \iu \eta_k) = R_{00}(E)$ (applied to the definition of $\rho_E$ in \Cref{s:rez}, and using the continuous mapping theorem). This in turn implies the weak convergence
\begin{align}\label{replace!}
\lim_{k \rightarrow \infty} 
p_{z_k}(r_{L+1}) \, d r_{L+1}  \prod_{i=0}^{L} \rho_{z_k} (w_i) \,d w_i 
 =p_E(r_{L+1}) \, d r_{L+1} \prod_{i=0}^{L} \rho_E(w_i) \,d w_i 
.
\end{align}
 
Observe that  \eqref{productN3} may be viewed as the expectation of a random variable $|T(r_{L+1}, \bm{w}; z)|^s$, where the law of $T(r_{L+1}, \bm{w}; z )$ is induced by the $L+2$ mutually independent random variables $(r_{L+1}, \bm{w})$, with $r_{L+1}$ having law $p_{z}(r_{L+1}) \, d r_{L+1} $ and each $w_i$ having law $\rho_E(w_i) \,d w_i $.
Further observe that for any $s' \in (0,1)$ there exists a constant $C(s', K, L) >0$ such that 
\[
\E\Big[ \big|R_{0v_L} (z_k) \big|^{s'} \Big] \le C
\]
for all $k \in \Zplus$, by \Cref{l:uniformlyintegrable}. 
Choosing $s'>s$,  we have $ \E \big[ |T(r_{L+1}, \bm{w};  E + \iu \eta)|^{s_0} \big] < C$, where $C> 1$ is independent of $\eta \in (0,1)$. 
It follows that the sequence of random variables $( |T(r_{L+1}, \bm{w}; z_k)|^s)_{k=1}^\infty$ is uniformly integrable, which justifies the exchange of limits 
\begin{align*}
\lim_{k \rightarrow \infty}  \E\big[ |R_{0v_L}(z_k) |^s \big]  
&=\frac{1}{t^s} \int_{\mathbb{C}^{L+2} } \lim_{k \rightarrow \infty}
\big|T(r_{L+1} , \bm{w})\big|^s
p_{z_k}(r_{L+1}) \, d r_{L+1}  \prod_{i=0}^{L} \rho_{z_k} (w_i) \,d w_i 
\\
&=\frac{1}{t^s} \int_{\mathbb{R}^{L+2} }
\big|T(r_{L+1} , \bm{w})\big|^s
p_E(r_{L+1}) \, d r_{L+1} \prod_{i=0}^{L} \rho_E(w_i) \,d w_i 
.
\end{align*}
We used \eqref{replace!} for the second equality. 
The conclusion follows from the definition \eqref{productN2}.
\end{proof}

\subsection{Integral Operator}\label{s:transferoperator}
We next show that $\Upsilon_L(s;E)$ can be computed by iterating a certain integral operator.
We  first define a sequence of functions $\{g_i(x;E)\}_{i=0}^{L+1}$ in the following way. We set $g_{L+1}(x ; E) = p_E(-x/t^2)$, and for $i \in \unn{0}{L}$ we set
\be\label{gidef}
g_i(x;E) = \frac{t^{2-s}}{|x|^{2-s}}\int_{\mathbb{R}^2} g_{i+1}(y;E)\rho_E\left( - y - \frac{t^2}{x} \right) \, dy  .
\ee
\begin{lemma}\label{l:Erecursion}
Fix an integer $L \ge 1$ and real number and $E \in \bbR$. We have 
\begin{flalign}\label{lsbig}
\Upsilon_L(s;E)  = \frac{1}{t^{2+s}}\int_{-\infty}^{\infty} g_0(x ; E) \, dx.
\end{flalign}
\end{lemma}
\begin{proof}
Iterating \eqref{gidef}, we obtain that the right-hand side of \eqref{lsbig} equals
\be\label{e:bapstquestion1}
\frac{1}{t^{2+s}}\int_{\mathbb{R}^{L+2}}
p_E(-y_{L+1}/t^2)
\left|  \prod_{i=0}^{L}
\frac{t}{y_i} \right|^{2-s} 
  \prod_{i=0}^{L}  \rho_E\left( - y_{i+1} - \frac{t^2}{y_i} \right)\, dy_{i+1} \cdot dy_0 .
\ee
We now make the change of variables $r_{L+1} = - y_{L+1}/t^2$ and 
$w_i  = - y_{i+1} - \frac{t^2}{y_i}$ for $i\in\unn{0}{L}$.
We have 
\be
dy_{L+1} = -  t^2 \, d r_{L+1}, \qquad d w_i = \frac{t^2}{y_i^2} \, d y_i.
\ee 
Inserting these identities into \eqref{e:bapstquestion1}, we find that it equals 
\be\label{e:bapstquestion2}
t^{-s} \int_{\mathbb{R}^{L+2}}
p_E(r_{L+1})
\left|  \prod_{i=0}^{L}
\frac{t}{y_i} \right|^{-s} 
  \prod_{i=0}^{L}  \rho_E(w_i) \, dw_{i}  \cdot dr_{L+1} .
\ee
Recalling the definition of $r_i$ from \eqref{ui}, we find that $y_i = -t^2 r_i $ for all $i \in \unn {0}{L+1}$ by starting with our choice  $r_{L+1} = - y_{L+1}/t^2$  of the change of variables, and inducting backwards from $i=L+1$ to $i=0$ using the definition of $w_i$.
Then \eqref{e:bapstquestion2} becomes 
\be
t^{-s} \int_{\mathbb{R}^{L+2}}
\left|   \prod_{i=0}^L
  \frac{t}{w_i - t^2 r_{i+1}} \right|^s \prod_{i=0}^{L}  \rho_E\left( w_i  \right)\, dw_{i}  \cdot 
 p_E(r_{L+1}) \, d r_{L+1} .
\ee
This expression is \eqref{productN2}, which proves the lemma.
\end{proof}
\begin{definition}
Recall the definition of the Banach space $\mathcal X$ from \Cref{d:X}. 
We  define the linear operator $S = S_{s,E}$ on $\mathcal X$ by 
\bex
S(f) =  \frac{1}{t^{2+s}}\int_{-\infty}^\infty   f(y) \, dy.
\eex
\end{definition}
\begin{remark}\label{r:SFinite}
By the definition of the norm on $\mathcal X$, we have
\bex
\big|f(y) \big|
\le \| f \| \cdot \big(1 + |y|^{2-s} \big)^{-1}.
\eex
It follows that
\begin{align}\label{Sfupper}
\big| S(f) \big| 
&\le  \frac{ \| f \| }{t^{2+s}}
\int_{-\infty}^\infty \frac{dy}{1 + |y|^{2-s}}  =  \frac{C  \| f \|}{t^{2+s}},
\end{align}
so $S(f)$ is finite for any $f \in \mathcal X$. 
\end{remark}
The following corollary is an immediate consequence of \Cref{l:Erecursion} and the definition of $F$ in \eqref{e:Fdef}. We therefore omit its proof. 
\begin{corollary}\label{c:transfer}
Fix an integer $L \ge 1$ and real number and $E \in \bbR$. We have 
\begin{equation}
\Upsilon_L(s;E) = S \big( F^{L+1} (g_{L+1}) \big).
\end{equation}
\end{corollary}

\section{Application of the Krein--Rutman Theorem}\label{s:KRapp}

The goal of this section is to establish \Cref{c:KR}, which states that the operator $F$ (see \eqref{e:Fdef}) has a positive leading eigenvalue with an associated positive eigenfunction. We show this by applying the Krein--Rutman theorem, after several preliminary estimates to establish that the necessary hypotheses hold. In \Cref{s:bounded}, we show that $F$ is bounded. In \Cref{s:Fasymptotics}, we provide asymptotic expansions for functions in the image of $F$ for large and small $x \in \R$. \Cref{s:compact} contains the proof of \Cref{l:Fcompact}, which shows that $F^2$ is compact. (The reason for working with $F^2$ will be explained below.) 
In \Cref{s:krconclusion}, we apply the Krein--Rutman theorem, with the previous estimates of this section as inputs, to prove \Cref{c:KR}. 

In this section, we let $t \in (0,1)$ be a free parameter (instead of fixing $t$ to the value given in \eqref{e:tKdef}).

\subsection{Boundedness}\label{s:bounded}
We recall the definition of the Banach space $\mathcal X$ from \Cref{d:X}. We first show that $F$ is bounded as an operator from $\mathcal X$ to $\mathcal X$. We recall that the weight function $
w(x) = 1 + |x|^{2 - s }$ was defined in \Cref{d:X}. 

\begin{lemma}\label{l:Fbounded}
For every $s\in (0,1)$, $t>0$, $K>1$, and $E\in \R$, there exists a constant $C(s,t,E, K) > 0$ such that for all $u\in\mathcal{X}$,  
\be
\| F u \| \le C \| u \|.
\ee 
\end{lemma}
\begin{proof}
By the definition of the norm for $\mathcal X$, we have that for all $x\in \R$,
\[
\big |u(x)\big|
\le \| u\| \cdot \big(1 + |x|^{2-s} \big)^{-1}.
\]
Inserting this bound in the definition of $F$ gives
\begin{align}
\begin{split}
(F u)(x) \cdot w(x) &= 
\frac{t^{2-s} (1 + |x|^{2 - s }) }{|x|^{2-s}}
\int_{-\infty}^\infty
\rho_E \left( - y - \frac{t^2}{x} \right) u(y) \, dy\\
&\le 
\| u\| \cdot \frac{t^{2-s} (1 + |x|^{2 - s }) }{|x|^{2-s}}
\int_{-\infty}^\infty
\rho_E \left( - y - \frac{t^2}{x} \right) \frac{1}{1 + |y|^{2-s}} \, dy.
\end{split}
\end{align}
Set \be
\delta = \min\left( \frac{t^2}{9|E|}, \frac{t^2}{9} \right).
\ee 
We first consider the regime $|x| \ge \delta$. Then there exists $C(s,t,E,K)>0$ such that 
\begin{align}
\begin{split}
&\frac{t^{2-s} (1 + |x|^{2 - s }) }{|x|^{2-s}}
\int_{-\infty}^\infty
\rho_E \left( - y - \frac{t^2}{x} \right) \frac{1}{1 + |y|^{2-s}} \, dy\\ &
\le  C \int_{-\infty}^\infty
\rho_E \left( - y - \frac{t^2}{x} \right) \frac{1}{1 + |y|^{2-s}} \, dy\\ 
&\le C \int_{-\infty}^\infty
\frac{1}{1 + |y|^{2-s}} \, dy \le C.
\end{split}
\end{align}
The first inequality follows from our assumed lower bound for $x$, the second inequality follows from \eqref{e:rhoEupper} and the assumption that $\rho$ is $\L$-regular, and the last inequality follows since the relevant integral is finite. 

Next, we consider the regime $|x| \le \delta$.
We  have 
\begin{align}
\begin{split}
&\frac{t^{2-s} (1 + |x|^{2 - s }) }{|x|^{2-s}}
\int_{-\infty}^\infty
\rho_E \left( - y - \frac{t^2}{x} \right) \frac{1}{1 + |y|^{2-s}} \, dy\\
&\le C |x|^{-(2-s)} 
\int_{-\infty}^\infty
\rho_E \left( - y - \frac{t^2}{x} \right) \frac{1}{1 + |y|^{2-s}} \, dy\\
&\le  
C |x|^{-(2-s)} 
\int_{-\infty}^\infty\frac{dy}{(1 +  (E - y - t^2 x^{-1} )^2 ) (1 + |y|^{2-s})} \\
&\le C |x|^{-(2-s)} |E - t^2 x^{-1}|^{-(2-s)}\le C.
\end{split}
\end{align}
The first inequality follows from the assumed upper bound on $x$, and  the second inequality uses  \eqref{e:rhoequadratic}. 
The third inequality follow  follows from \Cref{l:32}. 
In the fourth inequality, we used 
\be
\left| \frac{t^2}{x} \right| \ge \max(9, 9|E|)
\ee 
to ensure that
\be\label{e:t2lower}
|E - t^2 x^{-1}| \ge \frac{t^2}{2|x|}
\ee 
by considering the cases $|E| \le 1$ and $|E| > 1$ separately. 
\end{proof}

\subsection{Asymptotics}\label{s:Fasymptotics} 
We now prove an asymptotic expansions for $(Fu)(x)$ as $x\rightarrow \infty$ and $x \rightarrow 0$.

We begin with the expansion as $x \rightarrow \infty$. We define $D\colon \mathcal X \rightarrow \R$ by
\be\label{e:ds}
D_{s}(u) = \int_{-\infty}^\infty
\rho_E \left( r  \right) u(-r ) \, dr.
\ee
Note that 
\begin{equation}\label{e:Dbounded}
\big| D_s(u) \big| \le C \| u\|
\end{equation}
 for some constant $C(E,K, \L) >0$  (by, e.g., the reasoning in \Cref{r:SFinite} and the upper bound in \eqref{e:rhoequadratic}). 
\begin{lemma}\label{l:powerseries}
For every $s  \in (0,1)$,  $t>0$, $K>1$, $E\in \R$,  $u\in\mathcal{X}$, and $x\in \R \setminus \{0 \}$, there exists $C(s)>0$ such that 
\be
\left| (Fu)(x)  - \frac{t^{2-s} D_{s}(u)}{|x|^{2-s}} \right| \le \frac{C t^{4-s} \L \| u\|}{|x|^{3-s}}.
\ee 
\end{lemma}
\begin{proof}
We find
\begin{align}
\begin{split}
&\left|\int_{-\infty}^\infty
\rho_E \left( r  \right) u(-r ) \, dr
- \int_{-\infty}^\infty
\rho_E \left( r  \right) u(-r - t^2 x^{-1}) \, dr\right| \\
&=\left| \int_{-\infty}^\infty
\left( \rho_E \left( r  \right) 
- 
\rho_E ( r  - t^2 x^{-1} ) \right) u(-r ) \, dr\right|\\
&\le 
\int_{-\infty}^\infty
\big| \big( \rho_E \left( r  \right) 
- 
\rho_E ( r  - t^2 x^{-1} )  \big) u(-r )\big| \, dr\\
&\le \frac{t^2\L}{|x|} \int_{-\infty}^\infty \big| u(-r )\big| \, dr\\
&\le \frac{t^2\L}{|x|} \int_{-\infty}^\infty \| u \| \big( 1 + |r|^{2-s} \big)^{-1} \, dr  \le \frac{C t^2\L \| u\|}{|x|},
\end{split}
\end{align}
where $C=C(s)>0$ is a constant depending on $s$. 
The equality follows by a change of variables, the third line follows from the triangle inequality, the fourth line follows from the first part of \Cref{l:bapstcollection}, and the final equality follows from  the definition of the norm $\| \cdot \|$. 
This implies the conclusion after recalling the definition of $F$.
\end{proof}

We also have the following asymptotic expansion around $x=0$ for functions in the image of $F^2$. We consider $F^2$ since $F^2u = F(Fu)$, enabling the use of the previous lemma to control $Fu$ for large $x$ in the integral defining $F$. 

\begin{lemma}\label{l:powerseries2}
For every $s \in (0,1)$, $t>0$, $K>1$, and $E\in \R$, there exists a   constant $C(s, t,E, K)  > 0$ such that for all $u\in\mathcal{X}$ and $x\in \R$, 
\be
\left| (F^2u)(x) -  D_{s}(Fu) \right| \le C \| u \|  |x|^{s/4}.
\ee
\end{lemma}
\begin{proof}
Set $\kappa = 1/2 + s/8$. 
We assume that $x>0$; the case $x<0$ is similar. 
Set
\begin{equation*}
J_1 = \left(-\infty, -\frac{1 }{x^{1-\kappa}} \right),\quad 
J_2 = \left(-\frac{1}{x^{1-\kappa}} , \frac{1}{x^{1-\kappa}} \right),
\quad
J_3 =\left(\frac{1}{x^{1-\kappa}}, \infty \right).
\end{equation*}
We begin with the integral $J_2$, which is the leading-order contribution. Set $v = Fu$. For all $x>0$ such that $t^2 > 2 x^\kappa$ and $r\in J_2$,
\be\label{e:t2lower2}
|r + t^2 x^{-1}| \ge \frac{t^2}{2x}.
\ee 
Using \Cref{l:powerseries}, we have for $D_s = D_s(u)$ that 
\begin{align}\label{e:calc1}
\begin{split}
&\left|\int_{J_2}  \rho_E(r) \big( v( - r - t^2 x^{-1}) -  t^{2-s} D_{s} |-r - t^2 x^{-1}|^{-(2-s)} \big)\, dr\right| \\
&\le C \L t^{4-s} \| u \| \int_{J_2} \rho_E(r) |r + t^2 x^{-1}|^{-(3-s)} \, dr \\
&\le C \L t^{s-2} |x|^{3-s} \| u \| \int_{-\infty}^\infty \rho_E(r)  \, dr \\
&\le C |x|^{3-s} \| u \|.
\end{split}
\end{align}
The second inequality follows from \eqref{e:t2lower2}, and the third inequality follows because $\rho_E$ is a probability density. We also absorbed $\L$ and $t^{s-2}$ into the constant in the last inequality, noting that it now depends on $t$, $K$,  and $\L$. 

Note that for any $r \in J_2$ and $y \in \R$ such that $|y| \ge 2 x^{-1 + \kappa}$, we have by a Taylor expansion (using  $|r/y| \le 1/2$) that
\begin{align}
\begin{split}
\left|\frac{1}{|y + r  |^{(2-s)} }  - | y|^{-2 + s} \right|  &= \frac{1}{ | y|^{2-s} }\left|\frac{1}{ \big|1   + (r/y)\big |^{(2-s)} }  -1  \right|\\
&\le  \frac{1}{ | y|^{2-s} }  \left( \frac{C |r| }{ |y| } \right) \le C | r| | y|^{-3 + s}.
\end{split}
\end{align}
Applying this bound, we find that there exists $c(s,K)>0$ and $C(s,t,K)>0$ such that for all $x \in (0,c) $ and $r \in J_2$,
\be
\left|\frac{1}{|-r - t^2 x^{-1}|^{(2-s)} } - t^{-2(2-s)} |x|^{2-s} \right|
\le C|x|^{2-s + \kappa}.
\ee 
Then arguing as in \eqref{e:calc1}, and using that the length of $J_2$ is $2 |x|^{-1 + \kappa}$, we find 
\begin{align}\label{e:calc2feb}
\begin{split}
\left|\int_{J_2}  \rho_E(r)  D_{s} \big( t^{-(2-s)} |x|^{2-s} -   t^{2-s}|-r - t^2 x^{-1}|^{-(2-s)} \big)\, dr\right| 
&\le C|x|^{-1 + \kappa} \cdot  |x|^{2-s + \kappa} \| u \| \\ &= C  |x|^{1-s + 2\kappa} \| u \|,
\end{split}
\end{align}
where we used \eqref{e:Dbounded} to control $D_s$. 
Further, by \eqref{e:rhoequadratic}, 
\be
\left| 1 -  \int_{J_2} \rho_E(r) \,dr \right| = 
\left|  \int_{-\infty}^{\infty} \rho_E(r)\,dr -  \int_{J_2} \rho_E(r) \,dr  \right| \le C x^{1-\kappa}.
\ee 
Using \eqref{e:calc1}, \eqref{e:calc2feb}, \Cref{l:Fbounded}, the previous line, and our choice of $\kappa$, we find 
\be\label{e:J2bound}
\left|
\int_{J_2} \rho_E(r)  v( - r - t^2 x^{-1})
 - t^{s-2} D_{s} |x|^{2-s} \right| \le C \| u \|   \left(|x|^{1- s + 2\kappa} +  |x|^{3-s-\kappa}  \right).
\ee

We now consider the integrals over $J_1$ and $J_3$. Adjusting $c$ downward if necessary, our assumption that $x<c$ implies that 
\begin{equation}\label{e:xbounds}
\frac{ t^2}{2|x|} \le | t^2 x^{-1}  +E |,
\end{equation}
and
\begin{equation}\label{e:rbounds!}
| r + E | \ge \frac{ |r|}{2}
\end{equation}
for all $r \in J_1$.
We further define 
\[
\mathcal B = \big \{ r \in \R :   t^2 x^{-1}   \ge 2 |r|  \big\}.
\]
Then on $J_1$, we  have that 
that
\begin{align}\label{e:J1bound}
\begin{split}
\int_{J_1}  \rho_E(r)  v( - r - t^2 x^{-1}) \,dr 
&= 
\int_{J_1 \cap \mathcal B}  \rho_E(r)  v( - r - t^2 x^{-1}) \,dr \\
& \quad +\int_{J_1 \cap \mathcal B^c}  \rho_E(r)  v( - r - t^2 x^{-1}) \,dr   \\
&\le 
C  \int_{J_1 \cap \mathcal B}  \frac{1}{ |E +r |^2} \cdot  \frac{\| u\|}{ | - r - t^2 x^{-1}|^{2-s} } \,dr \\
& \quad +  C \int_{J_1 \cap \mathcal B^c}  \frac{ 1}{1 + | E + r |^2} \cdot  v( - r - t^2 x^{-1}) \,dr   \\
&\le 
C |x|^{2-s}  \| u \|\int_{J_1 }  \frac{dr }{ |E +r |^2}\\
& \quad + C x^2  \int_{J_1 \cap \mathcal B^c}   v( - r - t^2 x^{-1}) \,dr  \\
&\le C (  |x|^{3 - s-  \kappa}  + |x|^2 ) \| u \|.
\end{split}
\end{align}
The first inequality follows from using \eqref{e:rhoequadratic} in both integrals, the fact that for all $x\in \R$, 
\[
\big|f(x)\big| \le \frac{\|f \|}{1 + |x|^{2-s}},
\]
by the definition of the norm on $\mathcal X$ in \Cref{d:X}, and \Cref{l:Fbounded} to show that \[\| v \| \le C \|u \|.\]
 The second inequality follows from using the definition of $\mathcal B$ to show that \[|-r - t^2x^{-1}|^{2-s} \ge c |t^2 x^{-1}|^{2-s}\] in the first integral, and  \eqref{e:xbounds} in the second integral. The third inequality follows from using \eqref{e:rbounds!} in the first integral, then integrating directly, and bounding the second integral by $C \| v \|$, then using \Cref{l:Fbounded}.

Similarly, we have on $J_3$ for $x<c$  that,
\begin{align}\label{e:J3bound}
&
\int_{J_3}  \rho_E(r)  v( - r - t^2 x^{-1})
\le 
C (  |x|^{3 - s -  \kappa}  + |x|^2 ) \| u \|.
\end{align}
Recalling the definition of $F$, combining our estimates on $J_1$, $J_2$, and $J_3$ (in \eqref{e:J1bound}, \eqref{e:J2bound}, and \eqref{e:J3bound}), 
 we conclude that for all $x$ such that $|x|<c$, 
\begin{align}
\big| (F v)(x) -
  D_{s}(Fu) \big|
\le C \| u \|  |x|^{s/4}. 
\label{e:thebound}
\end{align}
This implies that \eqref{e:thebound} holds for all $x \in \R$ after increasing $C$ (which is permissible by \eqref{e:Dbounded}, \Cref{l:Fbounded}, and the elementary bound $\| u \| \le \| u \|_\infty$). This is the desired conclusion.
\end{proof}

\subsection{Lower Bound} 
By the definition of the norm on $\mathcal X$, every non-negative function $u \in \mathcal X$ satisfies 
\[
u(x) \le \frac{ \| u \|_\infty}{1 + |x|^{2-s}}.
\]
The following lemma shows that a similar lower bound is obtained after applying $F$ twice. As with the proof of \Cref{l:powerseries2}, the choice of $F^2$ (instead of $F$) allows us to use \Cref{l:powerseries} in the proof. Since this positivity property is needed for applying the Krein--Rutman theorem to obtain the existence of a leading eigenvalue, we mainly consider $F^2$ in what follows.

We define the subset $\mathcal Y \subset\mathcal X$ by 
\be\label{e:mathcalY}
\mathcal Y = \left\{ u \in \mathcal X  : u \ge 0, \int_{-\infty}^\infty u(x) \, dx > 0 \right\}.
\ee 
\begin{lemma}\label{l:Flowerbound}
For every $s \in (0,1)$,  $t>0$, $K>1$, $E\in \R$, and $u\in \mathcal{Y}$, there exists a constant $c(s, t, u, E, K) > 0$ such that for all  $x\in \R$
\be
 F^2 u (x) \ge  \frac {c}{1 + |x|^{2-s}}.
\ee 
\end{lemma}
\begin{proof}
Fix $u \in \mathcal X$. All constants in this proof will depend on $s$, $t$, $u$, $E$, and $K$, but we omit this from the notation.
By \Cref{l:powerseries2}, there exists $c_0>0$ such that for all $x\in \R$ with $|x|< c_0$, $F^2 u(x) > c_0$. 
By \Cref{l:powerseries}, 
there exist constants $B, c>0$ such that for all $x\in \R$ with $|x| > B$,
\be
(Fu)(x) \ge \frac{c}{|x|^{2-s}}.
\ee 
Set $J= [ -B, B]^c$. Then, by \eqref{e:rhoequadratic}, we obtain for all $|x| \ge c_0$ that 
\begin{align}
\begin{split}
(F^2 u)(x) &= 
\frac{t^{2-s} }{|x|^{2-s}}
\int_{-\infty}^\infty
\rho_E \left( - y - \frac{t^2}{x} \right) (Fu)(y) \, dy\\
&\ge 
\frac{c  }{|x|^{2-s}}
\int_J 
\frac{1}{( 1 +  (y + t^2 x^{-1} - E )^2 )|y|^{2-s} }  \, dy\\
&\ge \frac{c  }{|x|^{2-s}} \cdot \frac{c}{1 + | t^2 x^{-1} - E |^{2-s}}\\ 
&\ge \frac{c  }{|x|^{2-s}} . 
\end{split}
\end{align}
The third line follows from 
\Cref{l:integrallowerbound}.
In the last line,  we used that $| t^2 x^{-1} - E | \le C$ for $|x| \ge c_0$.
This completes the proof.
\end{proof}

\subsection{Compactness}\label{s:compact}
Next, we show that $F^2$ is compact. 
Let $\mathcal B_1 \subset \mathcal X$ denote the unit ball in $\mathcal X$. We recall that the weight function $w(x) = 1 + |x|^{2-s}$ was defined in \Cref{d:X}. 

\begin{lemma}\label{l:equicontinuous}
For every $s \in (0,1)$,  $t>0$, $K>1$, and $E\in \R$,  the set of functions
\be
\mathcal V = \{ w\cdot F^2 u : u \in \mathcal B_1 \}
\ee 
is uniformly equicontinuous on $\R$. 
\end{lemma} 
\begin{proof}
In what follows, we always suppose that $u \in \mathcal B_1$ (so that $\| u \| \le 1$). We also omit the dependence of all constants on the parameters $s$, $t$, $E$, and $K$. 
Set $v = F^2 u$, and let $\kappa >0$ be a parameter to be chosen later. By definition, 
\begin{align}
\begin{split}\label{e:febthediff}
 & w (x) \cdot v(x)  - w(y) \cdot v (y) \\
 &=
 \frac{w(x) t^{2-s} }{|x|^{2-s}}
\int_{-\infty}^\infty
\rho_E \left( - r - \frac{t^2}{x} \right) (Fu)(r) \, dr
-
\frac{w(y) t^{2-s} }{|y|^{2-s}}
\int_{-\infty}^\infty
\rho_E \left( - r - \frac{t^2}{y} \right) (Fu)(r) \, dr
.
\end{split}
\end{align}
Then there exists $C(\kappa) >0$ such that for all $x,y \in \R$ with $|x| > \kappa$ and $|y| \ge \kappa$, 
\be\label{e:febdiff2}
\left|
\frac{w(x) t^{2-s} }{|x|^{2-s}} - \frac{w(y) t^{2-s} }{|y|^{2-s}}
\right| \le C |x -y|.
\ee
Further, the first claim in \Cref{l:bapstcollection} gives
\be\label{e:febdiff3}
\left|
\int_{-\infty}^\infty
\left( \rho_E \left( - r - \frac{t^2}{x} \right) - \rho_E \left( - r - \frac{t^2}{y} \right)  \right) (Fu)(r) \, dr \right| \le C|x-y|.
\ee 
when $|x| \ge \kappa$ and $|y| \ge \kappa$. To obtain \eqref{e:febdiff3}, we also used that 
\be
\int_{-\infty}^\infty \big| (Fu)(r)\big|  \, dr \le C \| F u \| \le C \| u \| \le C,
\ee 
by the definition of the norm for $\mathcal X$, the assumption that $\| u \| \le 1$, and \Cref{l:Fbounded}. 
Bounding the difference in \eqref{e:febthediff} using \eqref{e:febdiff2}, \eqref{e:febdiff3}, and \Cref{l:Fbounded}, and recalling that $\| u \| \le 1$ by assumption, we conclude that $w \cdot v$ is Lipschitz on $|x| > \kappa $ for some Lipschitz constant $L_0(\kappa)>1$. 

To show that $\mathcal V$ is uniformly equicontinuous, we must show that for every $\eps >0$, there exists $\delta >0$ such that for all $v \in \mathcal V$, we have  
\be
\big |v(x) - v(y) \big| \le \eps
\ee 
if $|x-y| < \delta$. Fix $\eps >0$, and choose $\kappa>0$ so that for all $x \in \R$ such that  $|x| < 2\kappa$, 
\be \label{e:orange1} \big |v(x) -  D_{s}(Fu) \big| \le \eps/3.
\ee 
Such a choice of of $\kappa$ is possible by \Cref{l:powerseries2}. Finally, choose $\delta = \kappa/(3 L_0)$. Then uniform equicontinuity is implied by \eqref{e:orange1} and the Lipschitz continuity on $|x| \ge \kappa$ derived above.
\end{proof}

\begin{lemma}\label{l:Fcompact}
For every $s\in (0,1)$,  $t>0$, $K>1$, and $E\in \R$, the operator $F^2 : \mathcal X \rightarrow \mathcal X$ is compact. 
\end{lemma}
\begin{proof}
We claim that the image of the unit ball $\mathcal B_1 \subset \mathcal X$ under $F^2$ is relatively compact in $\mathcal X$, which implies the conclusion of the lemma. Let $\{ u_n\}_{n=1}^\infty$ be an infinite sequence of functions such that $f_n \in \mathcal B_1$ for all $n \in \mathbb{N}$. It suffices to exhibit a convergent subsequence of $\{ F^2 u_n\}_{n=1}^\infty$. 

Let $\{ v_n \}_{n=1}^\infty$ denote a subsequence of $\{ u_n\}_{n=1}^\infty$ such that the sequence of real numbers $\{ D_s(v_n) \}_{n=1}^\infty$ converges (to a finite limit).
Such a subsequence exists since $\| u_n \| \le 1 $ for all $n\in \mathbb{N}$ by assumption, so the $D_s(v_n)$  are uniformly bounded (by \eqref{e:Dbounded}). Denote this limit by $\tilde D_s$. 

By \Cref{l:powerseries}, for every  $k>1 $, there exists $M_k > 1$ such that 
\be\label{e:theboundatinfinity}
\sup_{|x| > M_k}   \left| w(x) \cdot  (F^2 v_n )(x)  - w(x) \cdot \frac{t^{2-s} \tilde D_{s}  }{ |x|^{(2-s)} }  \right|  < \frac{1}{k  }.
\ee 
for all $n \in \mathbb{N}$. Note that the functions $w \cdot (F^2 v_n)$ are uniformly bounded, since $F$ is bounded on $\mathcal X$ by \Cref{l:Fbounded} and $\|  v_n \| \le 1$ for all $n \in \mathbb{N}$, and $\| v_n\|_\infty \le \| v_n\|$. These functions are also uniformly equicontinuous by \Cref{l:equicontinuous}.

By the Arzel\'a--Ascoli theorem, for every $k>1 $, there exists a subsequence $\{v^{(k)}_n \}_{n=1}^\infty$ of $\{ v_n \}_{n=1}^\infty$ such that  $\{F^2 v^{(k)}_n \one_{|x| < 2M_k} \}_{n=1}^\infty$ converges in $\mathcal X$. 
By successive applications of the Arzel\'a--Ascoli theorem, we may arrange these subsequences so that $\{v^{(k+1)}_n \one_{|x| < 2M_k} \}_{n=1}^\infty$ is a subsequence of $\{v^{(k)}_n \one_{|x| < 2M_k }\}_{n=1}^\infty$ for every $k \in \mathbb{N}$. We can also require that $\{v^{(k)}_n \one_{|x| < 2M_k }\}_{n=1}^\infty$ lies in a ball of radius $k^{-1}$ (by eliminating elements at the beginning of the subsequence).

By \eqref{e:theboundatinfinity},  the sequence $\{F^2 v^{(k)}_n \one_{|x| \ge  M_k} \}_{n=1}^\infty$ lies in a ball of radius $k^{-1}$ in $\mathcal X$. Then the diagonal sequence $\{F^2 v^{(n)}_n \}_{n=1}^\infty$ is Cauchy in $\mathcal X$, and the theorem is proved. 
\end{proof}

\subsection{Conclusion}\label{s:krconclusion}

We begin by recalling some terminology from the theory of Banach spaces. Given a Banach space $X$, a subset $K \subset X$ is called a \emph{cone} if it is a closed, convex set such that $c \cdot K \subseteq K$ for all $c > 0$ and $K \cap (-K) = \{0\}$. A cone   $K$  with a  nonempty interior  $K^0$ is called a \emph{solid cone}. An operator $S : X \to X$ is  \emph{positive} if $Sx \in K$ for every $x \in K$; it is \emph{strongly positive} if $Sx \in K^0$ for each $x \in S \setminus \{0\}$. 

We require the following corollary of the Krein--Rutman theorem, proved in \cite{chang2005methods}.
\begin{lemma}[{\cite[Theorem 3.6.12]{chang2005methods}}]\label{c:KRchang}
Let $X$ be a Banach space, $K \subset X$ be a solid cone, and $S : X \to X$ be a compact, strictly positive, linear operator. Then the spectral radius $r(S)$ satisfies $r(S)>0$, and $r(S)$ is a simple eigenvalue with eigenvector $v \in K^0$. Further, $S$ has no other eigenvalue with an eigenvector in $K \setminus \{ 0 \}$, and $|\lambda| < r(S)$ for all eigenvalues $\lambda \neq r(S)$. 
\end{lemma}

We aim to apply the proceeding corollary to $F^2$. 
Our argument is facilitated by the following lemma.

\begin{lemma}\label{l:Fstronglypositive}
Let $\mathcal{K} = \{f \in \mathcal{X} : f \geq 0\} \subset \mathcal{X}$ denote the cone of nonnegative functions in $\mathcal{X}$. Then for every $s \in (0,1)$,  $t>0$, $K>1$, and $E\in \R$, the operator $F^2$ is strongly positive (with respect to $\mathcal K$).
\end{lemma}
\begin{proof}
Note that $\mathcal{K}\setminus \{ 0 \}$ is the same as $\mathcal Y$ (defined in \eqref{e:mathcalY}). 
By \Cref{l:Fbounded} and \Cref{l:Flowerbound}, for every $u \in \mathcal{Y}$, there exist constants $C_u, c_u >0$ such that 
\be
c_u \le (F^2u)(x)  (1 + |x|^{2-s} ) \le C_u
\ee 
for all $x\in \R$. 
This implies that the open ball $\{ v \in \mathcal X : \| v - F^2 u \|  < c_u/2\}$ is contained in $\mathcal K$. This shows that $F^2$ is strongly positive.
\end{proof}
\begin{proof}[Proof of \Cref{c:KR}]
By \Cref{l:Fcompact} and \Cref{l:Fstronglypositive}, the operator $F^2$ satisfies the hypotheses of \Cref{c:KRchang}. Since $F^3 = F(F^2)$ can be written as the composition of a compact operator and a bounded operator, so does $F^3$, and similarly for $F^6$, and \Cref{c:KRchang} applies to these operators as well. Then there exist positive functions $v_2, v_3 : \R \rightarrow \R$, normalized so that $\| v_2 \| = \| v_3 \| =1$, and $\mu_2, \mu_3 \ge 0$ such that $F^2 v_2 = \mu_2 v_2$ and $F^3 v_3 = \mu_3 v_3$. We have 
\[
F^6 v_2  = (F^2)^3 v_2 =  \mu_2^3 v_2, \qquad F^6 v_3 = (F^3)^2 v_3 = \mu_3^2 v_3,
\]
Because $F^6$ has a unique positive eigenvector, we must have $v_2 = v_3$, and there must exist $\lambda >0$ such that $\lambda^2 = \mu_2$ and $\lambda^3 = \mu_3$. Since 
\[
\lambda^3 v_3= F^3 v_3 = F(F^2 v_3) = \lambda^2 (F v_3),
\]
we find that $v_3$ is a positive eigenvector of $F$ with eigenvalue $\lambda$. Moreover, $F$ cannot have any other non-negative eigenvectors, since this would imply that $F^2$ has multiple non-negative eigenvectors, contradicting the conclusion of \Cref{c:KRchang}. Finally, we see that the spectral radius of $F$ is bounded above by $\lambda$, since \Cref{c:KRchang} shows the spectral radius of $F^2$ is $\lambda^2$. It is also bounded below by $\lambda$, since $v_3$ is a eigenvector, so it must equal $\lambda$, as claimed. 
\end{proof}

\section{Eigenvalue Bounds}\label{s:evaluebounds}

The goal of this section is to prove \Cref{l:bulklambda}, which states that the leading eigenvalue $\lambda_{K,t,s,E}$ of the operator $F_{K,t,s,E}$ defined at \eqref{e:Fdef} is approximately $K^{-1}$ near the predicted locations of the mobility edges, up to error terms that decay as $K$ grows large. In \Cref{s:preliminaryiterations}, we recall an elementary lemma that allows $\lambda_{K,t,s,E}$ to be bounded above and below by considering the iterates $F^nu$ for arbitrary (non-negative) test functions $u$. \Cref{s:gooddensitybound} contains \Cref{l:Kuniformbound}, which establishes a bound on the density $\rho_E$ (from \Cref{s:rez}) that is independent of the branching number $K$ (and decays in $|x|$). 
\Cref{s:testvectoranalysis} considers the action of $F$ on $u_1$, $u_2$, and $u_3$, and uses this information to deduce \Cref{l:bulklambda}. 

\subsection{Preliminary Lemma}\label{s:preliminaryiterations} 

We require the following lemma, which is \cite[Proposition 1]{bapst2014large}. 
\begin{lemma}[{\cite[Proposition 1]{bapst2014large}}]\label{p:l1pf}
Let $(X, \|\cdot\|_X)$ be a normed vector space with a partial order $\preceq$ compatible with the multiplication by non-negative real  numbers, meaning $u \preceq v$ and $\lambda \geq 0$ imply $\lambda u \preceq \lambda v$. Suppose also that that for every $u, v \in X$, 
if 
$0 \preceq u \preceq v$, then $\|u\|_X \leq \|v\|_X$.

Let $F:X \rightarrow X $ be a linear operator preserving $\preceq$ (meaning $u \preceq v$ implies $F(u) \preceq F(v)$). Let $u$ and $v$ be non-negative vectors ($0 \preceq u, 0 \preceq v$), different from the zero vector, satisfying 
\[
F(u) \preceq \lambda u, \quad F(v) \succeq \mu v 
\]
for some $\lambda, \mu > 0$. Then for all vectors $p \in X$ such that there exist $a, b > 0$ for which
\[
a v \preceq p \preceq b u, 
\]
we have 
\[
\mu \leq \liminf_{n \to \infty} \|F^n(p)\|_X^{1/n} \leq \limsup_{n \to \infty} \|F^n(p)\|_X^{1/n} \leq \lambda. 
\]
\end{lemma}
We will apply this lemma to $F$ considered as an operator on $\mathcal X$, with $\preceq$ defined as $u \preceq v$ if $u(x) \le v(x)$ for almost all $x \in \R$. It is straightforward to see that $F$ preserves this order, since $F$ has a nonnegative kernel.

\subsection{Density Bound}\label{s:gooddensitybound}
As a preliminary step, we prove a bound for the fractional moment of a sum of independent  variables with densities $p_{E_i}$.
In will be used in the proof of the next lemma, in conjunction with Markov's inequality, to estimate the probability that this sum is large. 
 In particular, we state the lemma in a way that allows for these variables to be defined for different energies $E \in \R$. 
\begin{lemma}\label{l:QE1E2bound}
Fix $G> 0$ and suppose that $g \in [0 , G]$. Fix $r \in (0,1)$. 
There exists a constant $C(r,G) > 0$ such that the following holds. For every $E_1, E_2\in \R$ and integer $K > 1$, and $j \in \{1, 2, \dots, K-1\}$, let $X_1, \dots X_{K-1}$ be independent random variables such that $X_i$ has density $p_{E_1}$ for $i \le j$, and $X_i$ has density $p_{E_2}$ otherwise. Set 
\be
Q_{E_1, E_2}(j, K) = \sum_{i=1}^{K-1} X_i.
\ee 
Then for every $A >0$, we have
\be\label{e:QE1E2bound}
\P \big( | Q_{E_1, E_2}(j, K)  | \ge A ) \le \frac{ CK  }{ r A^{ 1- r} }.
\ee
\end{lemma} 
\begin{proof}
Let $q\in(0,1)$ be a parameter, and for all $i$ such that $1 \le i \le K-1$, let $\tilde E_i = E_1$ if $i \le j$ and $\tilde E_i = E_2$ otherwise. 
For all $i$, we have by \eqref{e:bapsta2} that
\begin{align}
\begin{split}
\label{e:Xfractionalmoment}
\E\big[
|X_i|^q 
\big] &= \int_\R |x|^q p_{\tilde E_i} (x) \, dx \\
&\le  \int_{ | x | < 1}  p_{\tilde E_i} (x) \, dx +   \| \rho \|_\infty  \int_{  |x| \ge 1}  |x|^{q- 2} \, dx\\
& \le  (1- q)^{-1} ( 1 + \|\rho \|_\infty).
\end{split}
\end{align}
To bound the first integral in the second line, we used that $p_{\tilde E_i}$ is a probability density. 
For notational convenience, we abbreviate $Q = Q_{E_1, E_2}(j, K) $.
Then \eqref{e:Xfractionalmoment} implies that 
\begin{align}
\E\big[
\left|
Q 
\right|^q
\big]
&\le  \sum_{i=1}^{K-1}
\E\big[ | X_i|^q 
\big]  \le  \frac{C K  }{1 -q },
\end{align}
where the first inequality follows from the elementary inequality $(a+b)^q \le a^q + b^q$ (holding for $a,b \ge 0$). 
Markov's inequality then implies that
\be\label{e:remarkov}
\P \big( |Q | \ge A ) \le \frac{ C  K }{(1-q) A^q  }.
\ee 
Setting $q = 1- r$, we obtain  \eqref{e:QE1E2bound}, as desired.
\end{proof}
We now prove an upper bound on the density $\rho_E$ that is uniform in the branching number $K$ of $\bbT$. 
\begin{lemma}\label{l:Kuniformbound}
Fix $G> 0$ and suppose that $g \in [0 , G]$. There exists a constant $C(G) > 0$ such that the following holds. 
For every $E_1, E_2 \in \R$ and $k, K \in \Zplus$ with $k \le K$, and every $x\in \R$,
\be\label{e:pEuniform}
\int_{x-2}^{x+2} \big(p_{E_1} ^{(k)} * p_{E_2} ^{(K-k-1)}\big)(y) \, dy \le \frac{C}{1 + |x |^{4/3}}.
\ee
Further, for every $E\in \R$ and $K > 1$, 
\be\label{e:rhoEuniform}
\rho_E (x)  \le \frac{C}{1 + |x +E |^{4/3}}.
\ee 
\end{lemma}
\begin{proof}
Recalling the notation of \Cref{l:QE1E2bound}, we note that $p_{E_1} ^{(k)} * p_{E_2} ^{(K-k-1)}$ is the density of $t^2Q_{E_1,E_2}(k, K)$. We suppose without loss of generality that $x>0$. For $x>3$, we have 
\begin{align*}
\int_{x-2}^{x+2}
\big(p_{E_1} ^{(k)} * p_{E_2} ^{(K-k-1)}\big)(y) \, dy &\le 
\P\big(
t^2Q_{E_1,E_2}(k, K) \ge  x - 2 
\big)\\
& \le \P\big(
Q_{E_1,E_2}(k, K) \ge  K^2(x - 2)
\big)\\
&\le  \frac{ C_r }{ r  K^{1-2r} | x-2|^{1-r} } \le  \frac{ C_r }{ r  K^{1-2r} | x|^{1-r} }.
\end{align*}
When $|x| \le K^{2}$, the previous bound implies 
\be
\int_{x-2}^{x+2}
\big(p_{E_1} ^{(k)} * p_{E_2} ^{(K-k-1)}\big)(y) \, dy  \le \frac{C_r}{r  |x|^{ 3/2 - 2r}}.
\ee
Choosing $r$ so that $3/2 - 2r = 4/3$, this gives
\be\label{e:jan1}
\int_{x-2}^{x+2} \big(p_{E_1} ^{(k)} * p_{E_2} ^{(K-k-1)}\big)(y) \, dy  \le \frac{C}{  |x|^{ 4/3}}.
\ee 
Next, when $|x|\ge K^2$, we observe that \eqref{e:bapsta} implies that
\be\label{e:jan2}
\int_{x-2}^{x+2} \big(p_{E_1} ^{(k)} * p_{E_2} ^{(K-k-1)}\big)(y) \, dy 
\le  \frac{  \| \rho \|_\infty}{|x|^{3/2}} \le \frac{C}{|x|^{4/3}}.
\ee
Further, we have 
\be\label{e:xless3}
\int_{x-2}^{x+2} \big(p_{E_1} ^{(k)} * p_{E_2} ^{(K-k-1)}\big)(y) \le  
\int_{-\infty}^\infty \big(p_{E_1} ^{(k)} * p_{E_2} ^{(K-k-1)}\big)(y) \,dy \le 1,
\ee
since $p_{E_1} ^{(k)} * p_{E_2} ^{(K-k-1)}$ is a probability density, which justifies \eqref{e:pEuniform} for $|x| \le 3$. 
We conclude from \eqref{e:jan1}, \eqref{e:jan2}, and \eqref{e:xless3} that \eqref{e:pEuniform} holds,

Next, we note that 
\begin{align}
\rho_E(x) &= \int_{-\infty}^\infty \label{e:banana2}
p^{(K-1)}_{E}(y) \rho(x - y +E) \, dy \\
&\le \int_{-\infty}^\infty
p^{(K-1)}_{E}(y) \cdot \frac{1}{1 + | x - y  +E |^2 }  \, dy \\ &\le  \frac{C}{1 + |x+E|^{4/3}},
\end{align}
The first equality is from the definition of $\rho_E$, the second line follows from the assumption that $\rho$ is $\L$-regular, and the third line follows from \Cref{l:32} and \eqref{e:pEuniform}. This establishes \eqref{e:rhoEuniform}. 
\end{proof}
The following lemma is a direct consequence of \Cref{l:Kuniformbound}. 
It will be applied in \Cref{s:testvectoranalysis} to understand the action of $F$ on various test vectors. 
\begin{lemma}\label{l:supbounds}
Fix $G > 0$. There exists a constant $C(G)>0$ such that for all $K \ge 2$,  $g \in [0, G]$, $E \in [ - G, G]$, and $s \in (5/6,1)$,
\be\label{e:supbound1}
\sup_{z \in \R} \sup_{|y| \le 1}
\rho_E(y+z) | z|^{2- s} \le C
\ee
and
\be\label{e:supbound2}
\sup_{z \in \R} \int_{|y| \ge 1} \rho_E(y+ z) \frac{|z|^{2-s}}{|y|^{2-s } } \,dy \le C.
\ee 
\end{lemma}
\begin{proof}
To prove \eqref{e:supbound1}, we note that by \Cref{l:Kuniformbound} it suffices to control
\be
\sup_{|E| \le G} \sup_{z \in \R} \sup_{|y| \le 1}
\frac{| z|^{2- s} }{ 1 + | y + z + E |^{4/3 } } \le \sup_{z \in \R} \sup_{|E| \le G+1} \frac{| z|^{2- s} }{ 1 + | z  + E|^{4/3} }.
\ee
We may suppose that $G>10$. There exists a constant $C(G)>0$ such that the above supremum is bounded on $[ - 2G, 2G]$, by continuity. When $|z| \ge 2G$, we have $|z +E | \ge |z|/2$ for $|E| \le G+1$, implying 
\begin{align}
\sup_{z \in \R} \sup_{|E| \le G+1|} \frac{| z|^{2- s} }{ 1 + | z  + E|^{4/3} }
\le \sup_{z \in \R}  \frac{| z|^{2- s} }{ 1 + | z /2 |^{4/3} } \le \sup_{z \in \R}  \frac{| z|^{7/6} }{ 1 + | z /2 |^{4/3} }  \le C,
\end{align}
establishing \eqref{e:supbound1}.

For \eqref{e:supbound2}, we first use \Cref{l:Kuniformbound} to see that
\begin{align}
\begin{split}
\sup_{z \in \R} \int_{|y| \ge 1} \rho_E(y+ z) \frac{|z|^{2-s}}{|y|^{2-s } } \,dy 
& \le C \cdot  \sup_{z \in \R} \int_{|y| \ge 1}   \frac{|z|^{2-s}} {\big( 1 + | y + z + E |^{4/3} \big)|y|^{2-s } }  \,dy .
\end{split}
\end{align}
When $|y| \ge |z|/2$, we have 
\begin{align}
\begin{split}
 \int_{|y| \ge |z|/2 }   \frac{|z|^{2-s}} {\big( 1 + | y + z + E |^{4/3} \big)|y|^{2-s } }  \,dy
 & \le 
 C  \int_{|y| \ge |z|/2}   \frac{dy } { 1 + | y + z + E |^{4/3}}  \\
 &\le  C  \int_{-\infty}^\infty   \frac{dy } { 1 + | y  |^{4/3}} \le C.
 \end{split}
\end{align}
We now suppose that $|y|\le |z|/2$. For brevity, we set $\mathcal S(z) = \{  y \in \R : |y| \ge 1, |y| \le |z|/2\}$. There exists $C(G)>0$ such that 
\begin{align}
\begin{split}
\sup_{|z| \le 10 G} \int_{\mathcal S }  \frac{|z|^{2-s}} {\big( 1 + | y + z + E |^{4/3} \big)|y|^{2-s } }  \,dy \le C,
\end{split}
\end{align}
by the dominated convergence theorem and the continuity of the integrand. For $|z| \ge 10 G$ and $|y| \le |z|/2$, we have $|y + z + E | \ge |z|/4$, implying 
\begin{align}
\begin{split}
\sup_{|z| \ge 10 G } \int_{\mathcal S }  \frac{|z|^{2-s}} {\big( 1 + | y + z + E |^{4/3} \big)|y|^{2-s } } 
&\le
C \cdot \sup_{|z| \ge 10 G } \int_{\mathcal S }  \frac{|z|^{2-s}} { | z |^{4/3} |y| }  \le C |z|^{-1/7},
\end{split} 
\end{align}
where we used the definition of $\mathcal S$ and evaluated the integral directly. This completes the proof. 
\end{proof}

\subsection{Test Vector Analysis}\label{s:testvectoranalysis}
Recall the definition of $\Delta$ from \eqref{e:tKdef}. 
Set 
\be\label{e:theus}
u_1(x) = \Delta^{-1} \one\{ |x| \le  \Delta \},
\quad 
u_2(x) = |x|^{-(2-s)} \one\{   \Delta  \le  |x| \le 1 \},
\quad
u_3(x) = |x|^{-(2-s)} \one\{   1  \le  |x| \}.
\ee
We will bound the leading eigenvalue of $F$ (from \eqref{e:Fdef}) from above using \Cref{p:l1pf} by studying the action of $F^n$ on the function $u_1 + u_2 + u_3$ as $n$ tends to infinity. To do this, the following lemmas study the action of $F$ on each of $u_1$, $u_2$, and $u_3$ individually. We will see that the dominant contribution to the upper bound comes from $Fu_2$. Therefore, to achieve a lower bound using \Cref{p:l1pf}, it suffices to prove a lower bound on $Fu_2$ only.

We begin with the lower bound on $Fu_2$. We note that the calculations for the lemmas in this subsection are similar to those in \cite[Section III]{bapst2014large}.
\begin{lemma}\label{l:lowerbound}
Fix $\eps >0$. There exist constants $\alpha_0(\epsilon) >0$ and $C>0$ such that the following holds for all $\alpha \in (0, \alpha_0] $, $s\in (5/6,1)$, $K>1 $, and $t, E \in \R$. 
For all $x \in \R$ such that $|x|  \ge \Delta$, 
\[
(F u_2)(x) 
\ge 
\frac{t^{2-s}}{ |x|^{2-s}}\left( 2 \big( \rho_E(0) - \eps ) \left( \frac{1 - \Delta^{s-1}}{s-1}  \right) -C \right).
\]
\end{lemma}
\begin{proof}
By the Lipschitz continuity of $\rho_E$ (provided by \Cref{l:bapstcollection}),
\begin{align*}
(F u_2)(x) & =
\frac{t^{2-s}}{ |x|^{2-s}}\left( 
\int_{\Delta\le | y| \le 1 }
\rho_E \left( - y - \frac{t^2}{x} \right) \frac{1}{|y|^{2-s}} \, dy
\right)\\
&\ge 
\frac{t^{2-s}}{ |x|^{2-s}}
\left( 
\int_{\Delta\le | y| \le 1 }
\left( \rho_E \left(  - \frac{t^2}{x} \right)  - C |y| \right) \frac{1}{|y|^{2-s}} \, dy\right)\\
&\ge 
\frac{t^{2-s}}{ |x|^{2-s}}\left( 2 \big( \rho_E(0) - \eps ) \left( \frac{1 - \Delta^{s-1}}{s-1}  \right) -C \right),
\end{align*}
where the last inequality follows from an explicit integration.
\end{proof}

We now provide upper bounds on $Fu_1$, $Fu_2$, and $Fu_3$. 
We start with $Fu_1$. 
\begin{lemma}\label{l:u1upper}
Fix $\eps ,G>0$.  There exist constants $\alpha_0(\epsilon) >0$ and $C(G)>0$ such that the following holds for all $\alpha \in (0, \alpha_0]$, $g \in [0, G]$, $s\in ( 5/6, 1)$, $K\ge 2$,  and $E \in [-G, G]$. 
For all $x \in \R$,
\begin{align}
(F u_1)(x) \le& 
C  \Delta t^{-(2-s)}  u_1(x) 
+  C t^{2-s} \big ( u_2(x) + u_3(x) \big).
\end{align}
\end{lemma}
\begin{proof}
We begin with the case $|x| \ge \Delta$. 
By \eqref{e:rhoEupper}, 
\begin{align*}
(F u_1)(x) 
&=
\frac{1}{\Delta}
\frac{ t^{2-s}}{|x|^{2-s}}
\int_{|y| \le \Delta}
\rho_E \left( - y - \frac{t^2}{x} \right)  \, dy \le \frac{C t^{2-s}}{|x|^{2-s}}.
\end{align*}

Next, we consider the case $|x| \le \Delta$. 
We obtain
\begin{align*}
(F u_1)(x) &= 
\frac{1}{\Delta}
\frac{ t^{2-s}}{|x|^{2-s}}
\int_{|y| \le \Delta}
\rho_E \left( - y - \frac{t^2}{x} \right)  \, dy\\
&\le 
\frac{2  t^{2-s}}{|x|^{2-s}}
\left(\sup_{|y| \le \Delta} \rho_E \left( - y - \frac{t^2}{x} \right)  \right)  \\
&= \frac{2  t^{2-s}}{|x|^{2-s}} t^{-2(2-s)} |x|^{2 -s} 
\left(\sup_{|y| \le \Delta}  \rho_E \left( - y - \frac{t^2}{x} \right)  \right)  \left( \frac{t^2}{|x|}\right)^{2-s}\\
&\le C t^{-(2-s)} \le C  \Delta t^{-(2-s)}  u_1(x) .
\end{align*}
The first inequality follows from bounding the integrand by its supremum, and the second inequality follows from  \Cref{l:supbounds} with $z =-  t^2/x$. Combining the two cases completes the proof.
\end{proof}

\begin{lemma}\label{l:u2upper}
Fix $\eps ,G>0$. There exist constants $\alpha_0(\epsilon) >0$ and $C(G)>0$ such that the following holds for all $\alpha \in (0, \alpha_0]$, $s\in (5/6,1)$, $g\in [0, G]$, $K>1 $,  and $E \in [-G,G]$. 
For all $x \in \R$, 
\begin{align*}
(F u_2)(x) 
 \le  &C \Delta t^{-(2-s)}
\left( \frac{1 - \Delta^{s-1}}{s-1}  \right) 
 u_1(x) \\
 & + t^{2-s} \left( 2 \big( \rho_E(0) + \eps ) \left( \frac{1 - \Delta^{s-1}}{s-1}  \right) +C \right) \big( u_2(x) + u_3(x) \big).
\end{align*}
\end{lemma}
\begin{proof}
First, suppose that $|x|  \le \Delta$. Then  letting $z = -t^2/x$, we obtain
\begin{align*}
(F u_2)(x) &= \frac{t^{2-s}}{ |x|^{2-s}}\left( 
\int_{\Delta\le | y| \le 1 }
\rho_E \left( - y - \frac{t^2}{x} \right) \frac{1}{|y|^{2-s}} \, dy
\right)\\
&\le 
\frac{ t^{2-s}}{|x|^{2-s}}
\left(\sup_{|y| \le 1} \rho_E \left( - y - \frac{t^2}{x} \right)  \right) 
\int_{\Delta \le |y| \le 1}
\frac{1}{|y|^{2-s}} \, dy \\
&= 
\frac{ 2t^{2-s}}{|x|^{2-s}} 
\left( \frac{1 - \Delta^{s-1}}{s-1}  \right)  \cdot
\sup_{|y| \le 1} \rho_E \left( - y - \frac{t^2}{x} \right) \\
&= 2 t^{-(2-s)}
\left( \frac{1 - \Delta^{s-1}}{s-1}  \right) 
\cdot
\sup_{|y| \le 1} \rho_E \left( y + z  \right) |z|^{ 2-s } \\
&\le 
C \Delta t^{-(2-s)}
\left( \frac{1 - \Delta^{s-1}}{s-1}  \right) 
 u_1(x).
\end{align*}
The second equality follows from explicit integration, and the second inequality follows from \Cref{l:supbounds} and the definition of $u_1$. 

Next, suppose that $|x| \ge \Delta$. 
Because $\rho_E$ is Lipschitz (by \Cref{l:bapstcollection}),
\begin{align*}
(F u_2)(x) & =
\frac{t^{2-s}}{ |x|^{2-s}}\left( 
\int_{\Delta\le | y| \le 1 }
\rho_E \left( - y - \frac{t^2}{x} \right) \frac{1}{|y|^{2-s}} \, dy
\right)\\
&\le 
\frac{t^{2-s}}{ |x|^{2-s}}
\left( 
\int_{\Delta\le | y| \le 1 }
\left( \rho_E \left(  - \frac{t^2}{x} \right)   + C |y| \right) \frac{1}{|y|^{2-s}} \, dy\right)\\
&\le 
\frac{t^{2-s}}{ |x|^{2-s}}\left( 2 \big( \rho_E(0) +  \eps ) \left( \frac{1 - \Delta^{s-1}}{s-1}  \right)  + C \right),
\end{align*}
where the last inequality follows by explicit integration. Combining the two cases completes the proof. 
\end{proof}

\begin{lemma}\label{l:u3upper}
Fix $\eps ,G>0$.  There exist constants $\alpha_0(\epsilon) >0$ and $C(G)>0$ such that the following holds for all $\alpha \in (0, \alpha_0]$, $g \in [0,G]$, $s\in (5/6,1)$, $K>1 $,  and $E \in [-G,G]$. 
For all $x \in \R$, 
\begin{align*}
(F u_3)(x) 
 \le &  C \Delta t^{-(2-s)}  u_1(x)  + C t^{2-s} \big ( u_2(x) + u_3(x) \big). 
\end{align*}
\end{lemma}
\begin{proof}
First, suppose that $|x| \le \Delta$. Letting $ z = -t^2/x$ and using \Cref{l:supbounds}, we obtain
\begin{align*}
(F u_3)(x) & =\frac{t^{2-s}}{ |x|^{2-s}}\left( 
\int_{| y| \ge 1 }
\rho_E \left( - y - \frac{t^2}{x} \right) \frac{1}{|y|^{2-s}} \, dy
\right)\\
&\le 
t^{-(2-s)}
\left( 
\int_{| y| \ge 1 }
\rho_E \left( y+ z  \right) \frac{|z|^{2-s}}{|y|^{2-s}} \, dy
\right)\\
&\le C t^{-(2-s)} 
=C \Delta t^{-(2-s)}  u_1(x).
\end{align*}
Next, suppose that $|x| \ge \Delta$. 
Using \eqref{e:rhoEuniform} and the fact that $|t^2/x| \le \alpha \le 1$, we have 
\begin{align*}
(F u_3)(x) & =
\frac{t^{2-s}}{ |x|^{2-s}}\left( 
\int_{ | y| \ge  1 }
\rho_E \left( - y - \frac{t^2}{x} \right) \frac{1}{|y|^{2-s}} \, dy
\right)\\
&\le  \frac{C t^{2-s}}{ |x|^{2-s}} 
\left( 
\int_{ | y| \ge 1 }
\frac{1}{1 + |y +E |^{4/3}}
 \frac{1 }{|y|^{2-s}} \, dy\right) =
\frac{C t^{2-s}}{ |x|^{2-s}}. 
\end{align*}
Combining the two cases completes the proof. 
\end{proof}

Combining our upper and lower bounds on the $Fu_i$, and using \Cref{p:l1pf}, we can deduce upper and lower bounds on the leading eigenvalue $\lambda_{s,E}$ of $F$ (recall \Cref{c:KR}). 
\begin{lemma}\label{l:rougheigenvector}
Fix $\eps >0$ and $G > 0$. There exist constants $\alpha_0(\epsilon) , C(G)>0$ such that the following holds for all $\alpha \in (0, \alpha_0]$.
There exists $K_0(\alpha)> 1$ such that for all $K \ge K_0$, there exists $s_0(K)\in (0,1)$ such that 
for all $s\in (s_0, 1)$, $g\in[0, G]$,  and $E \in [-G,G]$, 
\be
\left| 
\lambda_{s,E} - 2  t^{2-s} \rho_E(0)  \left( \frac{1 - \Delta^{s-1}}{s-1}  \right) 
\right| 
\le  t^{2-s} \left(  \epsilon \left( \frac{1 - \Delta^{s-1}}{s-1}  \right)  + C  \right).
\ee 
\end{lemma}
\begin{proof}
 \Cref{c:KR} shows that $v_E(x) > 0$ for all $x \in \R$. Then, using \Cref{l:Flowerbound} and the fact that $v_E$ is an eigenvector, and since $v_E \in \mathcal{X}$ implies that $\| v_E \|_\infty < \infty$, there exist constants $C,c >0$ such that 
\begin{equation}
c u_2(x) \le v_E(x) \le C\big(u_1(x) + u_2(x) + u_3(x)\big)
\end{equation}
for all $x \in \R$. 
The conclusion follows from combining \Cref{p:l1pf}, \Cref{l:lowerbound}, \Cref{l:u1upper}, \Cref{l:u2upper}, and \Cref{l:u3upper}, after recalling that $v_E$ is an eigenvector with eigenvalue $\lambda$.  The conditions $K \ge K_0(\epsilon)$ and $s > s_0(K)$ are necessary to apply \Cref{l:u1upper}. To obtain the desired bound, it suffices to choose $s_0(K)$ large enough so that
\be
C_0 \Delta t^{-(2-s)} \le  t^{2-s} \rho_E(0 ) \left( \frac{1 - \Delta^{s-1}}{s-1}  \right) 
\ee 
where $C_0$ denotes the constant from \Cref{l:u1upper}, where we recall \eqref{e:rhoEupper}, and notice that the difference quotient tends to $| \log \Delta | $ as $s$ goes to $1$ (which depends on $\alpha$). 
\end{proof}

\begin{proof}[Proof of \Cref{l:bulklambda}]
This is a straightforward consequence of \Cref{l:rougheigenvector}.
\end{proof}
We omit the proof of the following corollary, since it is an immediate consequence of \Cref{l:rougheigenvector}, the Lipschitz continuity of $\rho_E$ provided by the first part of \Cref{l:bapstcollection}, and \Cref{l:bapstuniform}. We recall the definitions of $\varpi$ and $I_\star=I_\star(E_0)$ given in \eqref{e:varpi} and \eqref{e:istar}, respectively. 

\begin{corollary}\label{c:roughlambda}
Fix  $\L>1$, $g>0$, and $E_0 \in \R$ satisfying the conditions of \Cref{t:mainlambda}. There exists  $K_0(\L) >0$ such that the following holds. 
For every $K \ge K_0$, there exists $s_0(K)\in(0,1)$ such that for all $E \in I_\star$ and $s \in (s_0, 1)$, we have for $t = g ( K \ln K)^{-1}$ that 
\be
\frac{1}{2K} \le \lambda_{s,E} \le \frac{2}{K}. 
\ee

\end{corollary}

\section{Estimates on the Self-Energy Density}\label{s:selfenergydensityestimates}

In this section, we prove \Cref{l:amoluniform}, which provides important bounds on the density $\rho_E$ defined in \eqref{s:rez}. The first, \eqref{e:Elipschitz}, states that $\rho_E(x)$ is Lipschitz continuous as a function of $E$, with a Lipschitz constant decaying in $|x|$. 
The others, \eqref{e:le11finite} and \eqref{e:le11finite2}, show that the monotonicity of $\rho(x)$ for $x$ near $E$ propagates to the monotonicity of $\rho_{E'}(x)$ for $E'$ near $E$ and $x$ near zero. 
\Cref{s:preliminarydensitybounds} provides some preliminary estimates on quantities related to $\rho_E$. 
\Cref{s:mainlipschitz} contains the proof of \Cref{l:amoluniform}.

 Throughout this section, we fix  $\L>1$, $g>0$, $\rho$, and $E_0 \in \R$ satisfying the conditions of \Cref{t:mainlambda}.
 For any real functions $f$ and $g$ and $k\in \N$, we let $f^{*k}$ denote the $k$-fold convolution of $f$ with itself, and $f*g$ denote the convolution of $f $ with $g$. 
Recall that $p_E$ denotes the probability density for a solution of \eqref{e:rde}. We let $Q_E$ denote the cumulative density function corresponding to $p_E$.
\subsection{Preliminary Density Bounds}\label{s:preliminarydensitybounds}

The  following estimate on $\rho_E$ is the goal of this section. It can be thought of as a kind of uniform control for its derivative in $E$. However, for technical reasons, we do not treat the derivative directly and instead work with finite differences. We prove it below, in \Cref{s:mainlipschitz}, after proving a set of preliminary bounds. 
\begin{proposition}\label{l:amoluniform}
Set $I_\L  =  [E_0 - \L^{-1}, E_0 + \L^{-1}]$. There exists $C,K_0 >0$ such that the following is true for all $K\ge K_0$. 
\begin{enumerate}
\item
We have 
\be\label{e:Elipschitz}
\big| \rho_{E_2}(x) - \rho_{E_1}(x)  \big| \le C |E_1 - E_2| \big(|x|+1\big)^{-\fourthirds } 
\ee
 for all $E_1,E_2 \in I_{2 \L}$ and $x \in \R$. 
\item 
Suppose that 
\be\label{e:ge11}
\rho'(E)  \ge \frac{1}{\L}
\text{ for all $E \in I_\L$ }.
\ee 
Then 
\be\label{e:le11finite}
\rho_{E_2}(x) - \rho_{E_1}(x) \ge \frac{E_2 - E_1}{2\L} 
\ee 
for all  $E_1,E_2 \in I_{2 \L}$,  and $x  \in [ - (4\L)^{-1}  ,(4\L)^{-1}  ]$.
\item Suppose that 
\be\label{e:le11}
\rho'(E)  \le  - \frac{1}{\L} \text{ for all $E \in I_\L$ }.
\ee 
Then
\be\label{e:le11finite2}
\rho_{E_2}(x) - \rho_{E_1} (x)  \le  - \frac{E_2  - E_1}{2\L} .\ee 
for all  $E_1,E_2 \in I_{2 \L}$,  and $x  \in [ - (4\L)^{-1}  ,(4\L)^{-1}  ]$.
 \end{enumerate}
 \end{proposition}

Our aim is control the difference $\rho_{E_1} - \rho_{E_2}$ using the integral representation \eqref{e:bananav3}.
We see from the form of \eqref{e:bananav3} that this will require control over $p^{(K-1)}_{E_1} - p^{(K-1)}_{E_2}$, so we begin by estimating this difference, starting with $p_{E_1} - p_{E_2}$.
Our next lemma provides a representation for this difference using \eqref{e:loweraux1}. To state it, we set
\[
M_{j} (w) = \big( p_{E_1}^{*(K-j)} *(Q_{E_1} -  Q_{E_2})  * p_{E_2}^{*(j-1)}\big)(w) , \qquad M(w) = 
 \sum_{j=1}^K   M_{j}(w).
\]
\begin{lemma}
For all $K > 1$ and all $x,E_1, E_2 \in \R$, 
\begin{align}
\begin{split}
&p_{E_1} (x)  - p_{E_2}(x)\label{e:derivative-first} \\
& = \frac{1}{x^2}  \int_{-\infty}^{\infty} \left( \rho \left(\frac{1}{x} + E_1 + t^2 y\right) - \rho \left(\frac{1}{x} + E_2 + t^2 y\right) \right) p_{E_1}^{*K} (y) \, dy \\
& \quad - \frac{t^2}{x^2}  \int_{-\infty}^{\infty} \rho' \left(\frac{1}{x} + E_2 + t^2 y\right) M(y) \,  dy. 
\end{split}
\end{align}
\end{lemma}
\begin{proof}
By \eqref{e:loweraux1}, for any $E_1, E_2 \in \R$, we have 
\begin{align}
\begin{split}
&p_{E_1} (x)  - p_{E_2}(x)\label{e:derivative-first-proof} \\
&= \frac{1}{x^2}  \int_{-\infty}^{\infty} \left( \rho \left(\frac{1}{x} + E_1 + t^2 y\right) - \rho \left(\frac{1}{x} + E_2 + t^2 y\right) \right) p_{E_1}^{*K} (y) \, dy  \\
& \quad + \frac{1}{x^2}  \int_{-\infty}^{\infty} \rho \left(\frac{1}{x} + E_2 + t^2 y\right) \left( \sum_{j=1}^K  p_{E_1}^{*(j-1)} *(p_{E_1} -  p_{E_2})  * p_{E_2}^{*(K-j)} \right) (y)\,  dy.
\end{split}
\end{align}
For the equality in \eqref{e:derivative-first-proof}, we used 
\be
 p_{E_1}^{*K} -  p_{E_2}^{*K}
= 
\sum_{j=1}^K  p_{E_1}^{*(K-j)} *(p_{E_1} -  p_{E_2})  * p_{E_2}^{*(j-1)}.
\ee 
The conclusion now follows from integration by parts in the second integral (using that $\rho$ vanishes at infinity eliminate the boundary terms).
\end{proof}

From the representation in \eqref{e:derivative-first}, we see that it will be essential to control $M$ in order to bound $p_{E_1} - p_{E_2}$. 
We first consider the quantity $Q_{E_1} - Q_{E_2}$ (coming from the definition of $M$), which we analyze using the following lemma. 

\begin{lemma}
For all $K > 1$ and all $w,E_1, E_2 \in \R$, 
\begin{align} \label{e:derivative-third}
\begin{split}
&Q_{E_2} (w)   -  Q_{E_1} (w)   = -  \int_{-\infty}^{\infty} p_{E_1}^{*K} (y) \int_{w^{-1}}^{0}   \big( \rho(x  + E_1 + t^2 y ) - \rho (x + E_2 + t^2 y) \big) \, dx   \, dy\\
& \quad   - t^2  \int_{-\infty}^{\infty}
\left( \rho \left(\frac{1}{w} + E_2 + t^2 y\right) - \rho \left( E_2 + t^2 y\right) \right)M (y) \, dy.
\end{split}
\end{align} 
\end{lemma}
\begin{proof}
 Integrating \eqref{e:derivative-first} from $x_1$ to $x_2$ yields
\begin{align}
\begin{split}
&\big(Q_{E_1} (x_2)  - Q_{E_2}(x_2)\big)  -  \big( Q_{E_1} (x_1)  - Q_{E_2}(x_1)  \big)\label{e:derivative-second} \\
&= \int_{-\infty}^{\infty} p_{E_1}^{*K} (y) \int_{x_1}^{x_2}  \frac{1}{x^2}  \left( \rho \left(\frac{1}{x} + E_1 + t^2 y\right) - \rho \left(\frac{1}{x} + E_2 + t^2 y\right) \right) \, dx   \, dy \\
& \quad  -  t^2  \int_{-\infty}^{\infty}\int_{x_1}^{x_2} \frac{1}{x^2} \rho' \left(\frac{1}{x} + E_2 + t^2 y\right) M(y) \, dx \, dy \\
&= -  \int_{-\infty}^{\infty} p_{E_1}^{*K} (y) \int_{x_1^{-1}}^{x_2^{-1}}   \big( \rho (x  + E_1 + t^2 y) - \rho (x + E_2 + t^2 y) \big) \, dx   \, dy \\
& \quad  -  t^2  \int_{-\infty}^{\infty}
\left( \rho \left(\frac{1}{x_1} + E_2 + t^2 y\right) - \rho\left(\frac{1}{x_2} + E_2 + t^2 y\right) \right) M (y) \, dy 
\end{split}
\end{align}
In the second equality, we made the change of variables $x \mapsto (x')^{-1}$ and performed the integration. 
Taking $x_2 = - \infty$ gives the conclusion.
\end{proof}

\subsection{Self-Consistent Integral Bound}

To control the last term in \eqref{e:derivative-first}, it is natural to attempt to control the integral over $| M(y) |$ over $\R$. However, due to the heavy tails of $M$, it is not clear that this integral exists. Instead, we control the integral of $| M(y) |$ multiplied by a decaying weight function to enforce convergence. 
Specifically, for all $p \in \R$, we define
\be F(p) = \int_{-\infty}^\infty \big| M(w) \big| ( 1 + |wt^2 - p|)^{-1/\L}  \, dw.
\ee  

To bound $F(p)$,  we need the following preliminary technical bound, whose proof is given in \Cref{s:weightedintegralest}.

\begin{lemma}\label{l:weightedintegraldensity} 
Set $I_\L  =  [E_0 - \L^{-1}, E_0 + \L^{-1}]$.
There exist constants $C , K_0> 0$ such that the following holds. 
For every $E_1, E_2 \in I_\L$, $K \ge K_0$,  and $x\in \R$, 
\begin{align}\label{e:integral1}
\int_{-\infty}^{\infty} (1+|w|t^2)^{-1/\L}  p_{E_1}^{*(K-j)}*  p_{E_2}^{*(j-1)} (w-x) \, dw \le (1 + |x|t^2 )^{-1/\L} .
 \end{align}
\end{lemma}

The next lemma bounds $F(p)$. Since $M$ is defined in terms of $Q_{E_1} - Q_{E_2}$, and our expression for this difference in \eqref{e:derivative-third} is in turn given in terms of $M$, our argument is based on bounding $F(p)$ by a multiple of itself. 
\begin{lemma}\label{l:intermediatejanuary}
Set $I_\L  =  [E_0 - \L^{-1}, E_0 + \L^{-1}]$. There exists $C,K_0 >0$ such that  for all $K\ge K_0$, $E\in I_\L$, and $p\in \R$,
\be\label{e:finalMbound}
F(p)  \le C K^{1 + 2/\L}  |E_1 - E_2|.
\ee 
\end{lemma}
\begin{proof}
The proof proceeds in three steps.\\
\paragraph{\textit{Step 1.}} 
In \eqref{e:derivative-third}, taking the convolution with $p_{E_1}^{*(K-j)} * p_{E_2}^{*(j-1)}$ gives 
\begin{align} \label{e:derivative-fourth}
& - M_j(w) = \int_{-\infty}^{\infty}\int_{-\infty}^{\infty}  p_{E_1}^{*K} (y) \big( p_{E_1}^{*(K-j)} * p_{E_2}^{*(j-1)}  \big)(w-x) H(x,y) \, dx \, dy  \\
& \quad  \notag +  t^2  \int_{-\infty}^{\infty} \int_{-\infty}^{\infty} 
\big( p_{E_1}^{*(K-j)} * p_{E_2}^{*(j-1)}  \big)(w-x) 
\left( \rho \left(\frac{1}{x} + E_2 + t^2 y\right) - \rho \left( E_2 + t^2 y\right) \right) \\
& \quad  \quad \quad  \times
M (y)\, dx \, dy \notag,
\end{align}
where we set 
\be\label{e:Hdef}
H(x,y) = \int_{x^{-1}}^{0}   \big( \rho(r + E_1 + t^2 y ) - \rho (r + E_2 + t^2 y) \big) \, dr.
\ee 
\paragraph{\textit{Step 2.}} 
In this step and the next, we bound the function  $F(p)$
by considering the terms resulting from both of the integrals in \eqref{e:derivative-fourth} separately, after integrating against the weight $( 1 + |wt^2 - p |)^{-1/\L} $. In this step, we consider the term from the first integral in \eqref{e:derivative-fourth}. We begin by noting a few useful bounds on $H$.
By a change of variables, we get
 \begin{align}
H(x,y) 
&= 
\int_{x^{-1}}^{0}    \rho(r + E_1 + t^2 y ) \, dr  - \int_{x^{-1} + E_2  - E_1}^{ E_2 - E_1}  \rho (r + E_1 + t^2 y)  \, dr\notag \\
&  =
\int_{x^{-1} }^{x^{-1} + E_2 - E_1}  \rho (r + E_1 + t^2 y)  \, dr
-
 \int_{0}^{E_2 - E_1}    \rho(r + E_1 + t^2 y ) \, dr .\label{e:Hunifbound}
\end{align}
Then, because $\| \rho \|_\infty \le \L$, we obtain from \eqref{e:Hunifbound} that
\be\label{e:Hnearzero}
\big | H(x,y) \big| \le C | E_1 - E_2 |. 
\ee 
Using \eqref{e:Hnearzero},  and considering the region $|x| \le 1$ in the first integral in \eqref{e:derivative-fourth}, we find that 
\begin{align}
&\int_{-\infty}^{\infty} \Bigg| \int_{-\infty}^{\infty}\int_{-1}^1  p_{E_1}^{*K} (y) \big( p_{E_1}^{*(K-j)} * p_{E_2}^{*(j-1)}  \big)(w-x) H(x,y) ( 1 + |wt^2 - p| )^{-1/\L}   \, dx \, dy\, \Bigg|  dw\notag \\
&\qquad \le C| E_1 - E_2|,\label{e:integralbound0}
\end{align}
where we used that $p_{E_1}^{*K} $ and $p_{E_1}^{*(K-j)} * p_{E_2}^{*(j-1)}  $ are probability densities.

Further, note that the mean value theorem the assumption that $ \rho$ is $\L$-regular ( \Cref{d:Lregular}) together imply that 
\begin{align}
\big|\rho(r + E_1 + t^2 y ) - \rho (r + E_2 + t^2 y)
\big| \le C|E_1 - E_2|.
\end{align}
Inserting this in \eqref{e:Hdef} gives 
\begin{equation}\label{e:newHbound}
\big| H(x,y) \big| \le C |E_1 - E_2 | |x|^{-1}.
\end{equation}

Using \eqref{e:integral1} and \eqref{e:newHbound}, and considering the region $|x| \ge 1$ in the first integral in \eqref{e:derivative-fourth}, we have after the change of variables $t^2 \tilde w = t^2 w -p$ that
\begin{align}
&\int_{-\infty}^{\infty} \Bigg| \int_{-\infty}^{\infty}\int_{[-1,1]^c}  p_{E_1}^{*K} (y) \big( p_{E_1}^{*(K-j)} * p_{E_2}^{*(j-1)}  \big)(w-x) H(x,y) ( 1 + | t^2 w  - p |)^{-1/\L}   \, dx \, dy\, \Bigg|  dw \notag \\
&\qquad \le C  \int_{-\infty}^{\infty}\int_{[-1,1]^c}  p_{E_1}^{*K} (y)  \big| H(x,y)  \big| ( 1 + |t^2x - p|)^{-1/\L}   \, dx \, dy \notag  \\
&\qquad \le  C |E_1 -E_2|  \int_{-\infty}^\infty \int_{[-1,1]^c}  p_{E_1}^{*K} (y)  |x|^{-1}  ( 1 +  |t^2x - p|)^{-1/\L}   \, dx \, dy\notag  \\
&\qquad \le C K^{1/\L} | E_1 - E_2|. \label{e:integral1bound}
\end{align}
In the last inequality, we noted that 
\begin{align}
\begin{split}\label{e:janintegral}
 & \int_{|x| \ge 1 }  |x|^{-1}  ( 1 +  |t^2x -p  |)^{-1/\L}    \, dx \\
  &= \int_{|x| \ge t^2}  |x|^{-1}  ( 1 +  |x -p  |)^{-1/\L}    \, dx \\
   & = \int_{ t^2 \le |x| \le 1 }  |x|^{-1}  ( 1 +  |x -p  |)^{-1/\L}    \, dx + \int_{ |x| \ge 1 }  |x|^{-1}  ( 1 +  |x -p  |)^{-1/\L}    \, dx  \\
 &\le C K^{1/\L} + C  \le C K^{1/\L} .
 \end{split}
\end{align}
where the first equality follows from a change of variables. 
Combining \eqref{e:integralbound0} and
\eqref{e:integral1bound}, we find 
\begin{align}
&\Bigg|\int_{-\infty}^{\infty}  \int_{-\infty}^{\infty}\int_{-\infty,}^{\infty}  p_{E_1}^{*K} (y) \big( p_{E_1}^{*(K-j)} * p_{E_2}^{*(j-1)}  \big)(w-x) H(x,y) ( 1 + |t^2w-p |)^{-1/\L}   \, dx \, dy\,  dw\Bigg|  \notag \\
&\le C K^{1/\L} | E_1 - E_2|. \label{e:step2conclusion}
\end{align}

\paragraph{\textit{Step 3.}} 
We now proceed to bound the second integral in \eqref{e:derivative-fourth}. 
We begin by deriving another useful estimate. 
Because $\rho$ is $\L$-regular, we have 
\be\label{e:ubound}
\big |\rho (x^{-1} +E+yt^2) - \rho (E+yt^2) \big| \le C u(x,y),
\ee 
where
\begin{equation}\label{e:cases}
u(x,y)=
    \begin{cases}
        1 & \text{if }|x| \le 100   (|y|t^2+1)^{-1}\\
       (1 + |x|)^{-1} \big(|y|t^2+1\big)^{-2/\L} & \text{if } |x| > 100  (|y|t^2+1)^{-1}.
    \end{cases}
\end{equation}
The found in the first regime of \eqref{e:cases} follows from the assumption that $\| \rho \|_\infty \le \L$. 
When $|x| \le 1$, the second case follows follows from the decay bound
\begin{equation}
 \rho (u+yt^2) \le \frac{C}{1 + | u + yt^2|^2 }, 
\end{equation}
which follows from \Cref{d:Lregular} and the assumption that $|E| \le 2 \L$, after considering the cases $|y|t^2 \le 4 \L$ and  $|y|t^2 \ge 4 \L$ separately. When $|x| >1$, the second case follows from similar reasoning applied to $\rho'$ (using the mean value theorem). 

Using \eqref{e:ubound}, we have 
\begin{align}\label{e:weightedprelim}
&\int_{-\infty}^{\infty} (1+|t^2 x - p| )^{-1/\L} \left|\rho\left(\frac{1}{x}+E+yt^2\right)-\rho(E+yt^2) \right|  \,dx  \\
&\le C \int_{-\infty}^{\infty} (1+|t^2 x - p| )^{-1/\L}  u(x,y) \, dx.
\end{align}
Breaking this integral into two pieces according to the size of $u$, we have 
\begin{align*}
&\int_{|x| \le 100 ( |y| t^2 +1)^{-1} } (1+|t^2 x - p| )^{-1/\L}  u(x,y) \, dx \le C( |y| t^2 +1)^{-1},
\end{align*}
since the integrand is bounded by $1$. Further, we have
\begin{align}
\begin{split}
&\int_{|x| \ge 100 ( |y| t^2 +1)^{-1} } (1+|t^2 x - p| )^{-1/\L}  u(x,y)\, dx  \\ 
&\le \big(|y|t^2+1\big)^{-2/\L}  \int_{|x| \ge 100 ( |y| t^2 +1)^{-1} }(1+|t^2 x - p| )^{-1/\L} |x|^{-1} \,dx \\
&\le C\big(|y|t^2+1\big)^{-2/\L}  \int_{ 100 ( |y| t^2 +1)^{-1} \le |x| \le 1}  |x|^{-1} \,dx \\
& \quad  +  C \big(|y|t^2+1\big)^{-2/\L}  \int_{|x| \ge 1  } |x - t^{-2} p| ^{-1/\L}  |x|^{-1}   \,dx  \\
&\le C K^{1/\L} \big(|y|t^2+1\big)^{-1/\L},
\end{split}
\end{align}
where in the last inequality, we directly evaluated the first integral and used \eqref{e:janintegral} to bound the second. 
Inserting the previous two bounds in \eqref{e:weightedprelim} gives 
\begin{multline}
\int_{-\infty}^{\infty} (1+|t^2 x  -p | )^{-1/\L} \left|\rho\left(\frac{1}{x}+E+yt^2\right)-\rho(E+yt^2) \right|  \,dx\\  \le C K^{1/\L} \big(|y|t^2+1\big)^{-1/\L}. 
\label{e:integral2}
\end{multline}
Inserting \eqref{e:integral1} and \eqref{e:integral2} into the second integral in \eqref{e:derivative-fourth} gives  
\begin{align}
\begin{split}
 &\int_{-\infty}^\infty  \Bigg| \int_{-\infty}^{\infty} \int_{-\infty}^{\infty} 
\big( p_{E_1}^{*(K-j)} * p_{E_2}^{*(j-1)}  \big)(w-x)  
\left( \rho \left(\frac{1}{x} + E_2 + t^2 y\right) - \rho \left( E_2 + t^2 y\right) \right)\\ &\quad \times 
M (y) ( 1 + |wt^2 - p |)^{-1/\L} \, dx \, dy  \Bigg| \, dw 
\\
&\le   \int_{-\infty}^{\infty} \int_{-\infty}^{\infty} 
(1+|t^2 x  -p | )^{-1/\L}  \left| \rho \left(\frac{1}{x} + E_2 + t^2 y\right) - \rho \left( E_2 + t^2 y\right) \right|
\big| M (y) \big|  \, dx \, dy    \\
& \le C K^{1 /\L}  \int_{-\infty}^\infty M(y) \big(|y|t^2+1\big)^{-1/\L} \, dy.\label{e:integral2bound}
\end{split}
\end{align}
Inserting \eqref{e:step2conclusion} and \eqref{e:integral2bound} into \eqref{e:derivative-fourth}, we obtain
\begin{align*}
 &\int_{-\infty}^\infty \big| M_j(w)\big|  ( 1 +  | t^2w -p |)^{-1/\L}  \, dw\\ & \le C K^{1/ \L} |E_ 1 - E_2| + C t^2 K^{ 1/ \L}  \int_{-\infty}^\infty \big| M(w)\big|  ( 1 + t^2 |w|)^{-1/\L}  \, dw .
\end{align*}
Summing over $j$ gives 
\begin{align*}
 &\int_{-\infty}^\infty \big| M(w)\big|  ( 1 + | t^2w -p |)^{-1/\L}  \, dw\\ & 
 \le CK^{1 + 1/\L} |E_ 1 - E_2| + C t^2  K^{ 1 + 1/ \L}  \int_{-\infty}^\infty \big| M(w)\big|  ( 1 + t^2 |w|)^{-1/\L}  \, dw .
\end{align*}
Since $t \le C K^{-1}$, this implies that for all $p\in \R$,  
\be\label{e:Fpbd}
F(p) \le C K^{1 + 1/\L} |E_ 1 - E_2|  + C K^{-1 + 1/\L} F(0).
\ee
Inserting $p=0$ gives $F(0) \le C K^{1 + 1/\L} |E_ 1 - E_2|$, and using this bound in \eqref{e:Fpbd} gives \eqref{e:finalMbound}. 
\end{proof}
We can now use the previous lemma to bound the difference $Q_{E_2}   -  Q_{E_1} $ from \eqref{e:derivative-third}. 

\begin{lemma}
Set $I_\L  =  [E_0 - \L^{-1}, E_0 + \L^{-1}]$. There exists $C,K_0 >0$ such that for all $K\ge K_0$, $w\in \R$, and $E_1, E_2 \in I_\L$, 
\begin{align}\label{e:largew}
 \big| Q_{E_2} (w)   -  Q_{E_1} (w) \big|
 \le \frac{C K^{2/ \L} |E _1 - E_2|}{1 + |w|}.
\end{align}
\end{lemma}
\begin{proof}

We begin with the case $|w| \le 1$. 
From \eqref{e:derivative-third}, we obtain from the assumption that $\rho$ is $\L$-regular and the assumption $|E_2| \le 2\L$ that 
\begin{align} 
\begin{split} \label{e:qdiff}
& \big| Q_{E_2} (w)   -  Q_{E_1} (w) \big| \\ & \le  \int_{-\infty}^{\infty} p_{E_1}^{*K} (y) \int_{-w^{-1}}^{0}   \big|  \rho(x  + E_1 + t^2 y ) - \rho (x + E_2 + t^2 y) \big| \, dx   \, dy \\
& \quad   + K t^2 \left| \int_{-\infty}^{\infty}
\left( \rho \left(\frac{1}{w} + E_2 + t^2 y\right) - \rho \left( E_2 + t^2 y\right) \right)M (y) \, dy \right|\\
&\le C  | E_1 - E_2 |  + K t^2 \big( F(0) + F( w^{-1})  \big) \le C K^{2/ \L} |E _1 - E_2|.
\end{split}
\end{align} 
To get the second inequality, we used the mean value theorem to show
\begin{align}
\begin{split}
&\int_{-w^{-1}}^{0}   \big|  \rho(x  + E_1 + t^2 y ) - \rho (x + E_2 + t^2 y) \big| \, dx \\ & \le 
| E_ 1 - E_2| \int_{-\infty}^\infty \big|\rho'\big( u(x) \big) \big| \, dx \le  C|E_1 - E_2|. 
\end{split}
\end{align}
in the integral in the second line, letting $u(x)$ denote the value in the interval 
\[
[x + E_ 1 + t^2 y , x + E_ 2 + t^2 y ]
\]
coming from the mean value theorem, and invoking \Cref{d:Lregular} and the fact that $ p_{E_1}^{*K} (y)$ is a probability density.
For the integral in the second line of \eqref{e:qdiff}, we used
\begin{align}  \left| \rho \left(\frac{1}{w} + E_2 + t^2 y\right)  \right| & \le \frac{1}{ 1 + | w^{-1} + E_2 + t^2 y |^{1/\L} } \le \frac{C}{(1 + | w^{-1}  + t^2 y | )^{1/\L} },
\end{align}
using the assumption that $E_2 \in I_\L$. 
We  used \Cref{l:intermediatejanuary} in the third inequality of \eqref{e:qdiff}. 
This completes our analysis of the case $|w| \le 1$. 

For $|w| \ge 1$, we use the mean value theorem. In the first integral in \eqref{e:qdiff}, we have 
\begin{align}
\begin{split}
&\int_{-w^{-1}}^{0}   \big|  \rho(x  + E_1 + t^2 y ) - \rho (x + E_2 + t^2 y) \big| \, dx \\ & \le 
| E_ 1 - E_2| | w|^{-1} \| \rho' \| \le  C |E_1 - E_2 | |w|^{-1},
\end{split}
\end{align}
where the last inequality follows from the third part of \Cref{d:Lregular}. 
The second integral is bounded as before. This completes the proof. 
\end{proof}

\subsection{Weighted Integral Estimate}\label{s:weightedintegralest}
We now prove \Cref{l:weightedintegraldensity}. 
\begin{proof}[Proof of \Cref{l:weightedintegraldensity}]
When $|w-x | \le |x|/2 $, we have $|w| \ge |x|/2$ and $(1+|w|t^2)^{-1/\L}  \le C  (1+|x|t^2)^{-1/\L} $. Hence 
\begin{align*}
&\int_{|w-x | \le  |x|/2   } (1+|w|t^2)^{-1/L}  p_{E_1}^{*(K-j)} *  p_{E_2}^{*(j-1)} (w-x) \, dw 
 \le C (1+|x |t^2)^{-1/\L},
\end{align*}
where we used the fact that $p_{E_1}^{*(K-j)}  * p_{E_2}^{*(j-1)} (w-x)$ is a probability density (as it is the convolution of probability densities). 
Further, when $|w -x  | \ge |x|/2$, we have by a change of variables that
\begin{align*}
\int_{|w-x | \ge  |x|/2   } (1+|w|t^2)^{-1/L}  p_{E_1}^{*(K-j)} * p_{E_2}^{*(j-1)} (w-x) \, dw \le 
 \int_{|u| \ge  |x|/2   }  p_{E_1}^{*(K-j)}  * p_{E_2}^{*(j-1)} (u) \, du,
\end{align*}
since $(1+|w|t^2)^{-1/L}  \le 1$.  
The convolution $p_{E_1}^{*(K-j)}  * p_{E_2}^{*(j-1)} (u)$ is the density of a random variable $\tilde Q$.  \Cref{l:QE1E2bound} gives
\be\label{e:remarkov2}
\P\big( |  \tilde Q | \ge |x|/ 2 \big) \le \frac{C K }{|x|^{3/4}},
\ee  This yields
\begin{align*}
 \int_{|u| \ge  |x|/2   }  p_{E_1}^{*(K-j)}  * p_{E_2}^{*(j-1)} (u) \, du &  = \P\big( |  \tilde Q | \ge |x|/ 2 \big) \\
& \le \min\left( 1, \frac{C K}{|x|^{3/4} } \right)\\
&\le  C ( 1 + t^2 |x|)^{-1/\L},
\end{align*}
where the last inequality comes from considering the cases $|x| \ge t^{-2}$ and $|x| \le t^{-2}$ separately. 
This completes the proof.
\end{proof}

\subsection{Proof of Main Lipschitz Bound}\label{s:mainlipschitz}
We now have all necessary inputs to prove \Cref{l:amoluniform}.
\begin{proof}[Proof of \Cref{l:amoluniform}]
The proof proceeds in four steps. Our starting point is \eqref{e:bananav3}, which gives
\begin{align}
\begin{split}
\rho_{E_2}(x) - 
\rho_{E_1}(x)
&= \int_{-\infty}^\infty \label{e:bananav2}
p^{*(K-1)}_{E_1}(y) \big(\rho(x - t^2y +E_2 )  - \rho(x - t^2 y +E_1 ) \big) \, dy\\
& \quad + 
\int_{-\infty}^\infty 
\big(p^{*(K-1)}_{E_2}(y) - p^{*(K-1)}_{E_1}(y)  \big)  \rho(x - t^2 y +E_2) \, dy. 
\end{split}
\end{align}
\\

\paragraph{\textit{Step 1.}} 
In this step, we bound the first integral in \eqref{e:bananav2}. 
Using that $\rho$ is $\L$-regular (recall \Cref{d:Lregular}) and  the mean value theorem, we have 
\begin{align}
\begin{split}
\label{e:piece_1}
&\left|  \int_{-\infty}^\infty
p^{*(K-1)}_{E_1}(y) \big(\rho(x - t^2y +E_2 )  - \rho(x - t^2 y +E_1 ) \big) \, dy \right| \\
&\le  C |E_1 - E_2| \int_{-\infty}^\infty
( 1 + |x - t^2 y| )^{-\fourthirds}
p^{*(K-1)}_{E_1}(y) \, dy\\
&=  C |E_1 - E_2| \int_{-\infty}^\infty
( 1 + |x -  y| )^{-\fourthirds}
p^{(K-1)}_{E_1}(y) \, dy.
\end{split}
\end{align}
By partitioning the region of integration into intervals of length $4$, and using  \eqref{e:pEuniform} and \eqref{e:integral1}, we obtain 
\begin{align}
\begin{split}\label{e:roughnew1}
&\int_{-\infty}^\infty
( 1 + |x -  y| )^{-\fourthirds}
p^{(K-1)}_{E_1}(y) \, dy\\
&= \sum_{k=-\infty}^\infty \int_{4k}^{4k+4} ( 1 + |x -  y| )^{-\fourthirds}
p^{(K-1)}_{E_1}(y) \, dy\\
&\le C \sum_{k=-\infty}^\infty \int_{4k}^{4k+4}  (1 + |x -  4k | )^{-\fourthirds} (1+ |y|)^{-4/3} \, dy\\
&\le \int_{-\infty}^\infty
( 1 + |x -  y| )^{-\fourthirds}
(1+ |y|)^{-4/3} \, dy \le |y|^{-\fourthirds}.
\end{split}
\end{align}
Since $p^{(K-1)}_{E_1}(y)$ is a probability density, we also have 
\be\label{e:roughnew2}
\int_{-\infty}^\infty
( 1 + |x -  y| )^{-\fourthirds}
p^{(K-1)}_{E_1}(y) \, dy \le 1.
\ee
Inserting \eqref{e:roughnew1} and \eqref{e:roughnew2} into \eqref{e:piece_1}, we obtain 
\be\label{e:step4-1}
\left|  \int_{-\infty}^\infty
p^{*(K-1)}_{E_1}(y) \big(\rho(x - t^2y +E_2 )  - \rho(x - t^2 y +E_1 ) \big) \, dy \right| 
\le  \frac{ C | E_1 - E_2| }{(1 + |x|)^{1 + 1/\L} }.
\ee
\paragraph{\textit{Step 2.}} 
In this step and the next, we consider the second integral  in \eqref{e:bananav2}.  This step is concerned with the regime $|x| \ge t^{-2}$ in \eqref{e:bananav2}. 
From \eqref{e:derivative-first}, 
we have 
\begin{align}\label{e:deltap}
\begin{split}
\big| p_{E_1} (x)  - p_{E_2}(x) \big| &\le  \frac{C |E_1 - E_2| }{x^2}  + \frac{t^2}{x^2} \int_{-\infty}^\infty \left| \rho'\left( \frac{1}{x}  + E_2 + t^2 y\right) M(y)\right| \, dy\\
&\le \frac{C |E_1 - E_2| }{x^2}    + \frac{C t^2}{x^2} 
\int_{-\infty}^\infty ( 1 + | t^2 y + x^{-1}  |)^{-1/\L} \big|  M(y) \big| \, dy\\
&\le \frac{C |E_1 - E_2| }{x^2} .
\end{split}
\end{align}
where we used $| \rho' | \le C$ to get the first term on the right side of the first line, the decay hypothesis on $\rho'$ from \Cref{d:Lregular} in the second line, and  \eqref{e:finalMbound} and $ t^2 K^{1+2/\L} < C$ to deduce the last line. We next note that 
\begin{align*}
\begin{split}
&
\left| \int_{-\infty}^\infty 
\big(p^{*(K-1)}_{E_1}(y) - p^{*(K-1)}_{E_2}(y)  \big)  \rho(x - t^2 y +E_2) \, dy \right|
\\ 
& = 
\left| 
\sum_{j=1}^K  
\int_{-\infty}^\infty  \int_{-\infty}^\infty \int_{-\infty}^\infty p_{E_1}^{*(K-j)}(u - w )  (p_{E_1}(w) -  p_{E_2}(w))  p_{E_2}^{*(j-1)}(y - u)
 \rho(x - t^2 y +E_2) \, dy \,du \, dw \right|\\
 & = 
\left| 
\sum_{j=1}^K  
\int_{-\infty}^\infty  \int_{-\infty}^\infty (p_{E_1}(w) -  p_{E_2}(w))  p_{E_1}^{*(K-j)} * p_{E_2}^{*(j-1)} (y-w)
 \rho(x - t^2 y +E_2) \, dw \, dy   \right|. 
\end{split}
\end{align*}
For all $r\in \R$, set 
\be
 R_w (r) = \frac{1}{K} \sum_{j=1}^K p_{E_1}^{*(K-j)} * p_{E_2}^{*(j-1)}(r) \rho\big(x - t^2 ( r + w) + E_2\big).
\ee 
Then
\begin{align}
\begin{split}\label{e:wregions}
&
\left| \int_{-\infty}^\infty 
\big(p^{*(K-1)}_{E_1}(y) - p^{*(K-1)}_{E_2}(y)  \big)  \rho(x - t^2 y +E_2) \, dy \right|
\\ 
& =  K \left|
\int_{-\infty}^\infty
\int_{-\infty}^\infty 
\big(p_{E_1}(w) - p_{E_2}(w)  \big)  R_w ( y -w ) \, dw \, dy
\right|.
\end{split}
\end{align}
In the last integral, we will consider the regions $|w| > t^{-2}$ and $|w| <t^{-2}$ separately. We begin with $|w|>t^{-2}$. 
Note that by the assumption that $\rho$ is $\L$-regular (see \Cref{d:Lregular}), we have
\begin{align}\label{e:translated}
\begin{split}
\left| \int_{-\infty}^{\infty}  R_w ( y -w )  \, dy \right| 
\le \frac{C}{( 1 + | x   - t^2 w |^{4/3} ) },
\end{split}
\end{align}
where the inequality follows as in \eqref{e:roughnew1} by breaking the integral into intervals of length four, applying \eqref{e:pEuniform}, and resumming.
Then
\begin{align}
\begin{split}\label{e:largewpen00}
&
K \left|
\int_{-\infty}^\infty
\int_{|w| \ge t^{-2}}
\big(p_{E_1}(w) - p_{E_2}(w)  \big)  R_w ( y -w ) \, dw \, dy
\right|
\\ 
&\quad  \le C K \int_{|w| \ge t^{-2}}  \frac{\big| p_{E_1}( w) - p_{E_2}(  w)  \big| }{ 1 + | x  - t^2 w |^{4/3}  }   \, dw .
\end{split}
\end{align}
Using \eqref{e:deltap}, we have 
\begin{align}
\begin{split}\label{e:largewpen0}
K  \int_{|w| \ge t^{-2}} \frac{\big| p_{E_1}( w) - p_{E_2}(  w)  \big| }{ 1 + | x  - t^2 w |^{4/3}  }   \, dw 
&\le 
K  |E_1  - E_2 | \int_{|w| \ge t^{-2}}
\frac{dw }{w^2 ( 1 + | x  - t^2 w |^{4/3} ) }\\
&\le  Kt^2  |E_1  - E_2 |  \int_{|u| > 1 } 
\frac{du }{u^2 ( 1 + | x  - u |^{4/3} ) }\\
&\le \frac{Kt^2 |E_1  - E_2 |}{  1 + |x|^{4/3}   }  .
\end{split}
\end{align}
In the last line, we used the bound
\be
 \int_{|u| > 1   } 
\frac{du }{u^2 ( 1 + | x  - u |^{4/3} ) } \le \int_{-\infty}^\infty \frac{2 du }{(1 + u^2 ) ( 1 + | x  - u |^{4/3} ) } 
\le \frac{ C}{  1 + |x|^{4/3} },
\ee
which follows from \Cref{l:32}. We conclude from \eqref{e:largewpen00} and \eqref{e:largewpen0} that 
\begin{align}
\begin{split}\label{e:largewpen1}
K \left|
\int_{-\infty}^\infty
\int_{|w| \ge t^{-2}}
\big(p_{E_1}(w) - p_{E_2}(w)  \big)  R_w ( y -w ) \, dw \, dy
\right| \le  \frac{Kt^2 |E_1  - E_2 |}{ ( 1 + |x|^{4/3} )  } .
\end{split}
\end{align}

Next, we consider the region $|w|\le t^{-2}$ in \eqref{e:wregions}. 
Using integration by parts, we have
\begin{align}
\begin{split}
&\frac{d}{dw}  \left( \int_{-\infty}^\infty R_w ( y -w )  \, dy \right) \\ & = 
-  \frac{1}{K} \int_{-\infty}^\infty \sum_{j=1}^K \big(p_{E_1}^{*(K-j)} * p_{E_2}^{*(j-1)}(y-w)\big)' \rho\big(x - t^2 y + E_2\big) \,dy \\ 
&= -  \frac{t^2}{K} \int_{-\infty}^\infty \sum_{j=1}^K p_{E_1}^{*(K-j)} * p_{E_2}^{*(j-1)}(y-w) \rho'\big(x - t^2 y + E_2\big) \,dy,
\end{split}
\end{align}
which implies, as in \eqref{e:roughnew1} and  \eqref{e:translated}, that 
\begin{align}
\begin{split}\label{e:rdiff}
\left|\frac{d}{dw}  \left( \int_{-\infty}^\infty R_w ( y -w )  \, dy \right) \right|
\le \frac{Ct^2 }{( 1 + | x   - t^2 w |^{1 + 1/\L} ) }.
\end{split}
\end{align}
 Integrating by parts, we obtain
\begin{align}
\begin{split}\label{e:wpen}
&K  \left|
\int_{|w| \le t^{-2}}
\big(p_{E_1}(w) - p_{E_2}(w)  \big)  \left( \int_{-\infty}^\infty R_w ( y -w )  \, dy \right) \, dw
\right| \\
&\le K  \left|
\int_{|w| \le t^{-2}}
\big(Q_{E_1}(w) - Q_{E_2}(w)  \big)  \left[ \frac{d}{dw} \left( \int_{-\infty}^\infty R_w ( y -w )  \, dy \right) \right] \, dw
\right| \\
&\quad + K \left| \big(Q_{E_1}(t^{-2}) - Q_{E_2}(t^{-2} )  \big) \left( \int_{-\infty}^\infty R_{t^{-2}} ( y - t^{-2} )  \, dy \right) \right|\\
&\quad + K \left| \big(Q_{E_1}(-t^{-2}) - Q_{E_2}(-t^{-2} )  \big) \left( \int_{-\infty}^\infty R_{-t^{-2}} ( y + t^{-2} )  \, dy \right) \right|.
\end{split}
\end{align}
Using \eqref{e:largew} and \eqref{e:rdiff}, we obtain
\begin{align}\label{e:pen1}
\begin{split}
& K \left|
\int_{-t^{-2}}^{t^{-2}}
\big(Q_{E_1}(w) - Q_{E_2}(w)  \big)  \left[ \frac{d}{dw} \left( \int_{-\infty}^\infty R_w ( y -w )  \, dy \right) \right] \, dw
\right| \\
&\le Ct^2 K^{1+ 2/\L} |E _1 - E_2|  \int_{-t^{-2}}^{t^{-2}} \frac{dw}{(1+|w|)( 1 + | x   - t^2 w |^{1 + 1/\L} ) }\\
&\le  \frac{Ct^2 K^{1+ 2/ \L} |E _1 - E_2| } {(1 + | x  |)^{1 + 1/\L} }.
\end{split}
\end{align}
In the last line, we bounded the integral by noting that 
\begin{align*}
\begin{split}
  \int_{-t^{-2}}^{t^{-2}} \frac{dw }{(1+|w|)( 1 + | x   - t^2 w |^{1 + 1/\L} )) }
 & =  \int_{-1}^{1} \frac{t^{-2} du}{(1+|u t^{-2}|)( 1 + | x   - u |^{1 + 1/\L} ) } \le \frac{C}{(1 + |x|)^{1 + 1/\L} }
\end{split}
\end{align*}
and considering the regimes $|u| \le t^2$ and $|u| \ge t^2$ separately (and using our assumption that $|x| \ge t^{-2}$).
 Using \eqref{e:translated} and \eqref{e:largew}, we find
\begin{align}\label{e:pen2}
 K \left| \big(Q_{E_1}(t^{-2}) - Q_{E_2}(t^{-2} )  \big) \left( \int_{-\infty}^\infty R_{t^{-2}} ( y - t^{-2} )  \, dy \right) \right| 
\le \frac{Ct^{2} K^{1+ 3/ \L} |E _1 - E_2| }{ ( 1 + |x|)^{1 + 1/\L} }.
\end{align}
The last term of \eqref{e:wpen} can be handled similarly, and we obtain from \eqref{e:wpen} using \eqref{e:pen1} and \eqref{e:pen2} that 
\begin{align}\label{e:smallwpen}
K  \left|
\int_{|w| \le t^{-2}}
\big(p_{E_1}(w) - p_{E_2}(w)  \big)  \left( \int_{-\infty}^\infty R_w ( y -w )  \, dy \right) \, dw
\right| \le \frac{Ct^{2} K^{1+ 3/ \L} |E _1 - E_2| }{ ( 1 + |x|)^{1 + 1/\L} }.
\end{align}

Using \eqref{e:wregions}, \eqref{e:largewpen1}, and \eqref{e:smallwpen}, we find
\begin{align}
\begin{split}\label{e:wregions2}
\left| \int_{-\infty}^\infty 
\big(p^{*(K-1)}_{E_1}(y) - p^{*(K-1)}_{E_2}(y)  \big)  \rho(x - t^2 y +E_2) \, dy \right|
 \le \frac{Ct^{2} K^{1+ 3/ \L} |E _1 - E_2| }{ ( 1 + |x|)^{1 + 1/\L} }.
\end{split}
\end{align}
This bounds the second integral  in \eqref{e:bananav2} for $|x| \ge t^{-2}$. \\

\emph{Step 3.}  We now consider the case $|x| \le t^2$ in the second integral of \eqref{e:bananav2}.
Using integration by parts and \eqref{e:finalMbound}, we find
\begin{align}\label{e:step6}
\begin{split}
&\left| \int_{-\infty}^\infty 
\big(p^{*(K-1)}_{E_1}(y) - p^{*(K-1)}_{E_2}(y)  \big)  \rho(x - t^2 y +E_2) \, dy  \right| \\
 & = t^2 \left|  \int_{-\infty}^\infty 
M(y)  \rho'(x - t^2 y +E_2) \, dy \right| \le t^2  F( x ) \le C K^{-1 + 2/\L} |E_1 - E_2|.
\end{split}
\end{align}
Inserting this estimate in \eqref{e:bananav2} and using the result \eqref{e:step4-1}, we find 
\be
\big| \rho_{E_2}(x) - 
\rho_{E_1}(x) \big| \le  \frac{ C | E_1 - E_2| }{(1 + |x|)^{1 + 1/\L } }. 
\ee
This proves \eqref{e:Elipschitz}. \\

\emph{Step 4.}  
We now turn to the proof of \eqref{e:le11finite}. We begin by lower bounding the first integral in \eqref{e:bananav2}. We note by the change of variables $y' = t^2 y$, and the definitions of $p^{*(K-1)}_{E_1}(y)$ and $p^{(K-1)}_{E_1}(y)$ (corresponding to two random variables differing by a factor of $t^2$) that 
\begin{align}
\begin{split}\label{e:step4zero}
&\int_{-\infty}^\infty 
p^{*(K-1)}_{E_1}(y) \big(\rho(x - t^2y +E_2 )  - \rho(x - t^2 y +E_1 ) \big) \, dy\\
&= 
\int_{-\infty}^\infty 
p^{(K-1)}_{E_1}(y) \big(\rho(x - y +E_2 )  - \rho(x - y +E_1 ) \big) \, dy.
\end{split}
\end{align}
Set $J_{\L} = [ - \L/100, \L/ 100 ]$. Recalling the notation of \Cref{l:QE1E2bound}, we note that $p_{E_1} ^{(K)} $ is the density of $t^2Q_{E_1,E_1}(K, K)$.  
Using \Cref{l:QE1E2bound} with $r=1/4$, we have 
\begin{align}
\begin{split}\label{e:step40}
 \int_{J^c_\L} p^{(K-1)}_{E_1}(y)  \, dy& =  \P\big(
t^2Q_{E_1,E_1}(K, K) \ge   \L/100
\big)\\
& \le  \P\big(
 Q_{E_1,E_1}(K, K) \ge  K^2 \L/100
\big)\le  \frac{ C }{  K^{1/2} }.
\end{split}
\end{align}

Using the assumption that $\rho$ is Lipschitz, we have by \eqref{e:step40} that
\begin{align}
\begin{split}\label{e:step4one}
\left| \int_{J^c_\L} p^{(K-1)}_{E_1}(y) \big(\rho(x - y +E_2 )  - \rho(x - y +E_1 ) \big) \, dy\right| 
&\le C|E_1 - E_2|  \int_{J^c_\L} p^{(K-1)}_{E_1}(y)  \, du\\ & \le  \frac{C|E_1 - E_2| }{K^{1/2}}.
\end{split}
\end{align}
Further, using the assumption \eqref{e:ge11} and our assumptions on $E_1$, $E_2$, and $x$, we have by the mean value theorem that 
\begin{align}\label{e:step4two}
\begin{split}
& \int_{J_\L} 
p^{(K-1)}_{E_1}(y) \big(\rho(x - y +E_2 )  - \rho(x - y +E_1 ) \big) \, dy\\
&\ge \frac{ E_2 - E_1}{\L} \int_{J_\L}p^{(K-1)}_{E_1}(y)  \, dy \ge \frac{2}{3} \cdot \frac{ E_2 - E_1}{\L},
\end{split}
\end{align}
where in the last inequality, we used 
\[
 \int_{J_\L}p^{(K-1)}_{E_1}(y) = 1 -  \int_{J^c_\L}p^{(K-1)}_{E_1}(y) \,dy
\]
and \eqref{e:step40}. From \eqref{e:step4zero}, \eqref{e:step4one}, and \eqref{e:step4two}, we obtain 
\[
\int_{-\infty}^\infty 
p^{*(K-1)}_{E_1}(y) \big(\rho(x - t^2y +E_2 )  - \rho(x - t^2 y +E_1 ) \big) \, dy \ge \frac{ E_2 - E_1}{2\L}.
\]
Inserting this in \eqref{e:bananav2} and applying \eqref{e:step6}, 
we obtain
\begin{align}
\rho_{E_2}(x) - 
\rho_{E_1}(x)
&\ge \L^{-1}(E_2 - E_1)  -   C t^2 K^{1 + 2/\L}|E_1 - E_2|   \ge (2\L)^{-1}(E_2 - E_1).
\end{align}
This proves \eqref{e:le11finite}. The argument for \eqref{e:le11finite2} is similar and omitted.
This completes the proof. 
\end{proof}

Using \Cref{l:amoluniform}, we have the following analogue of \Cref{l:supbounds}. We omit the proof, since it is nearly identical to the proof of \Cref{l:supbounds} (using \eqref{e:Elipschitz} in place of \eqref{e:rhoEuniform}). 
\begin{lemma}\label{l:supboundscirc}
There exists $C, K_0>0$ such  for all $K\ge K_0$, $E \in I_\L$, and $s \in (1 - (5\L)^{-1} ,1)$,
\be
\sup_{z \in \R} \sup_{|y| \le 1}
 \big| \rho_{E_1} (y+z) -  \rho_{E_2}(y+z) \big| | z|^{2- s} \le C|E_1 - E_2|
\ee
and
\be
\sup_{z \in \R} \int_{|y| \ge 1} \big|\rho_{E_1} (y+z) -  \rho_{E_2}(y+z)  \big| \frac{|z|^{2-s}}{|y|^{2-s } } \,dy \le C|E_1 - E_2|.
\ee 
\end{lemma}

\section{Eigenvector Bounds}\label{e:evectorbounds}
 In this section, we study the shape of the leading eigenvector $v_{s,E}$ for the operator $F_E$ (from \eqref{e:Fdef}). 
 We begin in \Cref{s:parameterchoices} by defining various parameters that will be used throughout this section. Then, in \Cref{s:eigenvectorestimates}, we define an explicit approximate eigenvector $\tilde u$ and prove preliminary upper and lower bounds on $v_{s,E}$. Finally, in \Cref{s:approximateeigenvector}, 
 we show in \Cref{l:eigenvectorsandwich} that $F\tilde u$ is comparable to $F v_{s,E}$. 
 
 \subsection{Parameter Choices}\label{s:parameterchoices}
 Throughout this section, we fix  $\L>1$, $g\in (\L^{-1}, \L)$, $\rho$, and $E_0 \in \R$ satisfying the conditions of \Cref{t:mainlambda}. In particular, we have 
 \be\label{e:E0}
 \rho (E_0  )  =\frac{ 1 }{ 4 g}, \qquad |E_0| \le 2 \L
 \ee 
  Recall that by \Cref{l:bapstcollection}, the density $\rho_E$ is $\L$-Lipschitz, and recall the definitions of $\varpi$ and $I_\star$ from \eqref{e:varpi} and \eqref{e:istar}.
  Then using \Cref{l:bapstuniform}, with the choices $G = 3 \L $, $\eps = \varpi $, and $\delta = \varpi$,  there exists a constant $K_0(\L)>1$ (not depending further on $g$, $E_0$, or $\rho$), such that for all $K \ge K_0$, we have  
\be\label{e:peppercorn}
\rho_E(x) \ge \frac{1}{8 \L}
\ee
 for all $E \in I_\star$ and $x \in [- 2 \varpi , 2\varpi]$. 
 
Recall that $\alpha$ and $\Delta$ were introduced in \Cref{s:notation}.
For every $K>1$, let $\mathfrak s (K, \alpha)$ be any number such $\s > 9/10$, and for any $s\in (\mathfrak s ,1)$,  
\be\label{e:weirdlog}
\frac{1}{2} \log \left( \frac{1}{\Delta} \right)  \le \left( \frac{1 - \Delta^{s-1}}{s-1}  \right) \le 2 \log \left( \frac{1}{\Delta} \right).
\ee 
Such an $\mathfrak s $ exists because 
\be
\lim_{s \rightarrow 1^-}
\left( \frac{1 - \Delta^{s-1}}{s-1}  \right)  = \log \left( \frac{1}{\Delta} \right)  .
\ee 
By increasing $\mathfrak s$ if necessary, we have 
\be\label{e:tbounds}
\frac{t}{2} \le t^{2-s} \le 2t, \qquad  \ln K \le \frac{1}{1-s}
\ee 
for all $s \in (\mathfrak s, 1)$. 
Further increasing $\s$, we may suppose that $\s$ is greater than the value $s_0(K)$ given in \Cref{c:roughlambda}. 
In the results below, $\alpha_0$ will always be a constant that does  not depend on $K$ or $s$, and we will consider parameters $\alpha$ such that $\alpha \in (0, \alpha_0]$.

\subsection{Eigenvector Estimates}\label{s:eigenvectorestimates}
We now determine the shape of $v_E$. Define 
\be\label{e:tildeu}
\tilde u(x) 
 = \frac{1}{t^2 \ln K}  \one\{ |x| \le  \Delta \}
+  |x|^{-(2-s)} \one\{   \Delta  \le  |x|  \},
\ee 
and note that the definition of $\tilde u$ implicitly depends on the value of $\alpha$. Our guiding heuristic is that $v_E$ should be qualitatively similar (but not identical) to the vector $\tilde u$ (which is, in turn, similar to the test vector $u_1 + u_2 + u _3$ used in the previous section to provide an almost-optimal upper bound on the size of $\lambda_E$). To that end, \Cref{l:vupperone} and \Cref{l:sumac} provide complementary upper and lower bounds that show $v_E$ is proportional to $|x|^{-2+s}$ for $|x| \ge \varpi^{-1} t^2$. The regime of small $x$ is more subtle, because it is unclear where the exact transition between $t^{-2}$ behavior to $(t^2 \log K)^{-1}$ behavior is. In \Cref{l:sumac2}, we prove that $v_E(x)$ is lower bounded by a multiple of $(t^2 \log K)^{-1} $ for small $x$. For the upper bound at small $x$, we do not bound $v_E$ by $\tilde u$ directly, but instead  show that $(Fv)(x) \le (F\tilde u)(x)$, in \Cref{l:eigenvectorsandwich}. Since our argument establishing the monotonicity of the leading eigenvalue of $F$ in \Cref{s:conclusion} is based on comparing the iterates $F^nv_E$ and $F^n \tilde u$, this bound is sufficient for our purposes.

We recall our convention (stated in \Cref{s:proof}) that constants may depend on the choice of $\L$, and that this dependence is typically omitted from the notation. We begin with the upper bound on $v_E(x)$ for large $x$. 

\begin{lemma}\label{l:vupperone}
There exist 
$K_0 > 0$ and  $C>0$ such that the following holds.
For all $K\ge K_0$,  $E \in I_\star$,  $s \in (\s, 1)$, and $x\in \R$ such that $|x| \ge \varpi^{-1} t^2$, 
\be\label{e:vupperone}
v_E(x) \le \frac{C}{|x|^{2-s}}. 
\ee
\end{lemma}
\begin{proof}
The proof consists of three steps. \\ 
\paragraph{\textit{Step 1.}} 
 We first claim that there exists $K_0 > 0$ and  $C_0>0$ such that for all $K \ge K_0$, $s\in(\s, 1)$, and $E \in I_\star$, 
\be\label{e:stepA}
\frac{1}{ \log K} \int_{-\varpi}^\varpi v_{K, t , s,E}(x)\,dx \le  C_0.
\ee 
To show \eqref{e:stepA}, we will suppose that it does not hold and derive a contradiction. 
In this case, for every $C_0 > 0$, we  can find an increasing sequence $\{ K_j\}_{j=1}^\infty$  
and sequences $\{E_j\}_{j=1}^\infty$ and $\{s_j\}_{j=1}^\infty$ with  $E_j \in I_\star$ and $s_j \in (\s,1)$ such that 
\be\label{e:towardcontradiction1}
\lim_{j \rightarrow \infty} \frac{1}{\ln K_j} \int_{-\varpi}^\varpi v_{K_j, t, s_j, E_j} (x) > C_0.
\ee 
We henceforth abbreviate $v_{E_j} = v_{K_j, t, s_j, E_j}$.   
Suppose that $|x| \ge \varpi^{-1} t^2$. 
We find
\begin{align*}
\frac{2}{K_j} v_{E_j} (x) & \ge \lambda_{E_j} v_{E_j} (x)\\
& \ge \frac{t}{2|x|^{s-2}} \left(  \int_{-\varpi}^\varpi \rho_{E_j} \left(- y - \frac{t^2}{x}  \right) v_{E_j} (y)\, dy  + \int_{|y| \ge \varpi} \rho_{E_j} \left(- y - \frac{t^2}{x}  \right) v_{E_j}(y)\, dy\right)\\ 
& \ge \frac{t}{2|x|^{s-2}}   \int_{-\varpi}^\varpi \rho_{E_j} \left(- y - \frac{t^2}{x}  \right) v_{E_j} (y)\, dy\\ 
&\ge \frac{c t}{|x|^{s-2} } \int_{-\varpi}^\varpi v_{E_j} (y)\, dy \\
&\ge \frac{4}{K_j |x|^{s-2}}
\end{align*}
for sufficiently large $j$, after choosing $C_0 \ge 4 c^{-1}$.  The first inequality uses \Cref{c:roughlambda}, the second inequality follows from \eqref{e:tbounds}, the third inequality follows from the positivity of the second integral, the fourth inequality follows from the choice of $\varpi$  (see \eqref{e:peppercorn}), and the last inequality follows from taking $j$ sufficiently large in \eqref{e:towardcontradiction1} and using the choice of $t$ in \eqref{e:tKdef}.

Then 
\be\label{e:oregano}
v_{E_j}(x) \ge \frac{2}{|x|^{s-2}}
\ee
for $|x| \ge \varpi^{-1} t^2$. 
 Integrating  $v_E$ over $x \in \R$ and using \eqref{e:oregano} on $|x| \ge 1$ (and lower bounding by zero on  $|x| \le 1$), we obtain $\| v_{E_j} \|_1 \ge 2(1-s)^{-1}$. This is
 a contradiction, because we assumed that $\| v_{E_j} \|_1 = (1-s)^{-1}$ (in \eqref{e:vnormalization}). We conclude that \eqref{e:stepA} holds.\\ 
 
\paragraph{\textit{Step 2.}} 
Next, we claim that there exists a constant $C >0$ such that 
\be\label{e:prelim2}
\int_{|x| > \varpi } \frac{v_E(x)}{|x|} \, dx < C.
\ee
We begin by noting that, for $|x|> \varpi^{-1} t^2$ and $t =g (K \log K)^{-1}$ (from  \eqref{e:tKdef}), we  have $|t^2/ x | \le \omega$ for $K \ge K_0(\L)$. Moreover, since $E\in I_\star$ and $|E_0| \le 2 \L$, we have $|E| \le 3 \L$. Using these bounds in combination with \Cref{l:Kuniformbound}, we deduce the existence of a constant $C(\L) > 0$ such that 
\begin{equation}\label{e:newboundmarch0}
\rho_E (-y-t^2/x)
\le \frac{1}{1 + | E  -y-t^2/x |^{4/3}} \le \frac{C}{|y|},
\end{equation}
where the second inequality follows from considering the cases $|y| \le 4 \L +1$ and $|y| \ge 4\L +1$ separately.
Then for all $x$ such that $|x|>\varpi^{-1} t^2 $, we have by \Cref{c:roughlambda} that 
\begin{align}\label{e:marchintcalc}
\begin{split}
&\frac{1}{2K} v_E(x) \\ &\le \lambda_E v_E (x) \\
&= (Fv_E)(x) \\
&\le  2t|x|^{s-2} \left( \int_{|y|<\varpi} \rho_E (-y-t^2/x) v_E(y) \,dy +  \int_{|y|>\varpi} \rho_E (-y-t^2/x) v_E(y)\,  dy \right)\\
&\le 2 t|x|^{s-2} \left(C \log K   + C \int_{|y|>\varpi} \frac{v_E(y)}{|y|} \,dy\right),
\end{split}
\end{align}
where the second inequality uses \eqref{e:tbounds}, and in the last bound we used \eqref{e:rhoEupper}, \eqref{e:stepA}, \eqref{e:prelim2},  and \eqref{e:newboundmarch0}. 
Since $t = g (K \log K)^{-1}$, it follows that 
\begin{equation}\label{e:1f}
\frac{v_E(x)}{|x|} \le \frac{C|x|^{s-3}}{\log K}  \left(\log K + \int_{|y|>\varpi }  \frac{v_E(y)}{|y|}  \, dy\right).
\end{equation}
Integrating both sides of \eqref{e:1f} over $|x| > \varpi$, 
we obtain that there exists $K_0$ such that for all $K\ge K_0$, 
\begin{align}\label{e:bach0}
\begin{split}
\int_{|y|>\varpi}  \frac{v_E(y)}{|y|}  \, dy  &\le C +\frac{1}{2} \int_{|y|>\varpi}  \frac{v_E(y)}{|y|}  \, dy,
\end{split}
\end{align}
which implies \eqref{e:prelim2}.\\

\paragraph{\textit{Step 3.}} 
Our conclusion, \eqref{e:vupperone}, now follows from inserting \eqref{e:prelim2} into the last line of \eqref{e:marchintcalc}, recalling that $t =g (K \log K)^{-1}$ and taking $K\ge K_0(\L)$.
\end{proof}

\begin{lemma}\label{l:sumac}
There exist  $K_0 > 0$  and  $c>0$ such that the following holds.
For all $K\ge K_0$,  $E \in I_\star$,  $s \in (\s, 1)$, and $x\in \R$ such that $|x| \ge \varpi^{-1} t^2$,
\[
v_E(x) \ge  \frac{c}{|x|^{2-s}}. 
\]
\end{lemma}
\begin{proof}
The proof consists of two steps.\\
\paragraph{\textit{Step 1.}} 
First, we claim that there exist constants $K(g), c_0>0$ such that for all $K \ge K_0$, $s\in(\s, 1)$, and $E \in I_\star$, 
\be\label{e:stepA2}
\frac{1}{ \log K} \int_{-1 }^1  v_E(x)\,dx \ge c_0. 
\ee 
To show \eqref{e:stepA2}, we will suppose that it does not hold and derive a contradiction. 
In this case, for every $c_0 > 0$,  we can find an increasing sequence $\{ K_j\}_{j=1}^\infty$  
and sequences $\{E_j\}_{j=1}^\infty$ and $\{s_j\}_{j=1}^\infty$  with  $E_j \in I_\star$ and $s_j \in (\s,1)$ such that 
\be\label{e:pepper1}
\lim_{j \rightarrow \infty} \frac{1}{\ln K_j} \int_{-1}^1 v_{K_j, t, s_j, E_j} (x) < c_0.
\ee 
In particular, by taking $j$ large enough, 
we have for all $x \in \R$ that 
\begin{align*}
\frac{1}{2K_j } v_{E_j} (x) &\le \lambda_{E_j} v_{E_j} (x) \\
&= Fv_{E_j}(x) \\
&=  2t|x|^{s_j-2} \left( \int_{|y|<1} \rho_{E_j} (-y-t^2/x) v_E(y) \,dy +  \int_{|y|>1} \rho_{E_j} (-y-t^2/x) v_{E_j}(y)\,  dy \right)\\
&\le 
C t|x|^{s_j-2} \left( \int_{|y|<1}  v_{E_j} (y) \,dy +  \int_{|y|>1} \rho_{E_j} (-y-t^2/x)  \,  dy \right)\\
&\le C t|x|^{s_j-2} \left( \int_{|y|<1}  v_{E_j} (y) \,dy + 1 \right).
\end{align*}
The first line follows from \Cref{c:roughlambda}, the second line uses that $v_E$ is an eigenvector $F$, the third line uses the definition of $F$ and \eqref{e:tbounds},  the fourth line uses \eqref{e:rhoEupper} and  \Cref{l:vupperone}, and the last line follows because $\rho_{E_j}$ is a probability density.
Rearranging, we find that
\be\label{e:bluecars}
v_{E_j} (x) \le \frac{C |x|^{s_j-2} }{ \log K_j} \left( \int_{|y|<1}  v_{E_j} (y) \,dy + 1  \right) .
\ee
By taking $j$ sufficiently large and $c_0$ sufficiently small (relative to the constant $C$ in \eqref{e:bluecars}), we have 
\be\label{e:pepper2}
v_{E_j} (x) \le |x|^{s_j-2} \left( \frac{1}{3 }  + \frac{C}{ \ln K_j }\right) \le \frac{|x|^{s_j-2}}{4} .
\ee
Combining \eqref{e:pepper1} and \eqref{e:pepper2} (using the latter for $|x| \ge 1$), we obtain
\be
\| v_{E_j} \|_1 \le \frac{\ln K_j}{4} + \frac{1}{4 ( 1-s_j )} \le  \frac{1}{2( 1-s_j )}.
\ee 
for sufficiently large $j$, by the definition of $\s$. This contradicts our choice of normalization for $v_E$, $\| v_E\|_1 = (1-s)^{-1}$, and establishes \eqref{e:stepA2}.\\
\paragraph{\textit{Step 2.}} 
Using the bound from the previous step, we now complete the proof. 
By  \Cref{l:vupperone}, we find that there
exist constants $K_0, C(\delta)>0$ such that for all $K\ge K_0$ and $s\in (\s, 1)$, 
\begin{equation}
 \int_{\varpi<|x|<1} v_E(x)\, dx \le C.
\end{equation}
Using \eqref{e:stepA2}, we conclude that there exist $K_0, c_0> 0$ such that for all $K \ge K_0$, $s\in(\s, 1)$, and $E \in I_\star$, 
\be\label{e:stepA22}
  \frac{1}{ \log K} \int_{-\varpi}^\varpi v_{E}(x)\,dx \ge  c_0.
\ee 

Finally, for all $x\in \R$ such that  $|x| > \Delta$ and $K\ge K_0(g)$, we have using our choice of $\varpi$ (using \eqref{e:peppercorn}) \begin{align*}
\frac{2}{K} v_E(x) & \ge \lambda_E v_E(x)\\
& \ge \frac{t}{2|x|^{s-2}} \left(  \int_{|y| \le \varpi} \rho_E\left(- y - \frac{t^2}{x}  \right) v_E(y)\, dy\right)\\ 
& \ge \frac{c t}{2|x|^{s-2}} \left(  \int_{|y| \le \varpi}  v_E(y)\, dy\right)\\ 
&\ge \frac{c}{K | x|^{s-2}},
\end{align*}
where we used \eqref{e:tbounds} in the first two inequalities, and \eqref{e:stepA22} in the last inequality.
This implies the desired conclusion. 
\end{proof}

\begin{lemma}\label{l:sumac2}
Fix $\beta > 0$.  There exist  
$K_0(\beta)> 0$ and  $c(\beta)>0$ such that the following holds.
For all $K\ge K_0$,  $E \in I_\star$, $s \in (\s, 1)$,  and $x\in \R$ with $|x| \le  \beta t^2$,
\[
v_E(x) \ge  c(\log K)^{-1} t^{-2} .
\]
\end{lemma}
\begin{proof}
Suppose that $|x| \le \beta t^2$. 
We compute using the choice of parameter $A=\beta^{-1}/2 $ that 
\begin{align*}
v_E(x) & \ge  \frac{Kt  |x|^{s-2}}{4} \int_{-\infty}^{\infty}  v_E(y) \rho_E (-y-t^2/x)\,  dy\\ 
&\ge \frac{(\log K)^{-1}  |x|^{s-2}}{4} \int_{-t^2/x-A}^{-t^2/x + A} v_E(y) \rho_E (-y-t^2/x) \, dy \\
&\ge c  (\log K)^{-1} |x|^{s-2} \left| \frac{t^2}{x} \right|^{s-2} \int_{-t^2/x-A}^{-t^2/x+A} \rho_E (-y-t^2/x) \, dy \\
&=  c  (\log K)^{-1} |x|^{s-2} \left| \frac{t^2}{x} \right|^{s-2} \int_{-A}^{A} \rho_E (- y) \, dy \\
&\ge  c  (\log K)^{-1} t^{-2}.
\end{align*}
The first inequality uses \eqref{e:tbounds}, \Cref{c:roughlambda}, and the fact that $v_E$ is an eigenvector of $F$, along with the definition of $F$. The second inequality follows from restricting the domain of integration (using that the integrand is positive). The third inequality uses \Cref{l:sumac}, after noting that, in the case where $x >0$, 
\be
- \frac{t^2}{x} +A \le   - \frac{1}{\beta}  + \frac{1}{2\beta} \le - \frac{1}{\beta},
\ee
so that  \Cref{l:sumac} can be applied, and using
\be
\left|- \frac{t^2}{x}  - A \right|  =\frac{t^2}{x}  + A \le \frac{ 2t^2}{x}
\ee
to bound the argument of $v$ when applying \Cref{l:sumac}. Our use of  \Cref{l:sumac} requires $K$ large enough so that $\varpi^{-1} t^2 \le \beta^{-1}$, which gives rise to the requirement that we must have $K \ge K_0(\beta)$.
The $x<0$ case is similar.
The fourth line is a change of variables, and the final inequality uses 
 that our choice of $\varpi$ satisfies \eqref{e:peppercorn} to bound $\rho_E(-y)\ge (8\L)^{-1}$ for $|y| \le \varpi$ (and zero elsewhere). We note that this leads to a dependence of the constant $c$ in the last inequality on $A$, and hence on $\beta$. \end{proof}
 
 \begin{lemma}
  There exists  $\alpha_0 >0$ and $K_0 >0$ such that the following holds.
For all  $\alpha \in (0, \alpha_0]$, 
 there exists a constant $C(\alpha)> 0$ such that the following holds.
 For all $K \ge K_0$, 
 $E \in I_\star$, $s \in (\s, 1)$, 
and  $|x| \le  \Delta $,
\be\label{e:salt}
v_E(x) \le C t^{-2}.
\ee 
 \end{lemma}
 \begin{proof}
For notational convenience, we  let $\alpha>0$ be a parameter that will be decreased at various points throughout the proof, but only a finite number of times. We allow all constants to depend on $\alpha$ unless explicitly noted, and we use $\tilde C$ to denote constants that do not depend on $\alpha$.
Then  
we find that
\begin{align}
 v_E(x) &\le 2  Kt  |x|^{s-2} \int_{-\infty}^{\infty}  v_E(y) \rho_E (-y-t^2/x) \, dy\notag \\
&\le 2 \tilde C  (\log K)^{-1} |x|^{s-2} \left( \int_{|y|>\alpha t^2/ |x| } v_E(y) \rho_E (-y-t^2/x) \, dy + \int_{|y|<\alpha t^2/ |x| } v_E(y) \rho_E (-y-t^2/x) \, dy\right)\notag  \\
&\le \tilde C (\log K)^{-1} |x|^{s-2} \left( |\alpha t^2/ x |^{s-2} \int_{|y|>\alpha t^2/ |x| } \rho_E (-y-t^2/x) \, dy + \int_{|y|<\alpha t^2/ |x| } v_E(y) \rho_E (-y-t^2/x) \, dy\right) \notag \\
&\le \tilde C (\log K)^{-1} |x|^{s-2} \left(\alpha^{s-2} t^{2(s-2)}|x|^{2-s} + \int_{|y|<\alpha t^2/ |x| } v_E(y) \rho_E (-y-t^2/x) \, dy\right)\notag \\
&\le C (\log K)^{-1} t^{-2} + \tilde C (\log K)^{-1} |x|^{s-2} \int_{|y|<\alpha t^2/ |x| } v_E(y) \rho_E (y-t^2/x) \, dy.\label{e:paprika}
\end{align}
The first line follows from \eqref{e:tbounds} and \Cref{c:roughlambda}, the second line is a rewriting of the integral, the third line uses \Cref{l:vupperone} in the first integral, and the fourth line uses that $\rho_E$ is a probability measure. 
We now focus on bounding the integral in \eqref{e:paprika}.

We now consider separately the regimes $\varpi^{-1} \alpha t^2 \le |x|$ and $\varpi^{-1} \alpha t^2 \ge |x|$. We decrease $\alpha$ if necessary so that $\varpi^{-1} \alpha  < \alpha^{-1}/ 2$.

If $\varpi^{-1} \alpha t^2 \le |x|$,  then  $\alpha t^2/ |x| \le  \varpi$, \eqref{e:rhoEupper}, and \eqref{e:stepA} imply that
\be
\int_{|y|<\alpha t^2/ |x| } v_E(y) \rho_E (y-t^2/x) \, dy
\le \| \rho \|_\infty \int_{-\varpi}^{\varpi} v_E(y) \, dy \le \tilde C \ln K.
\ee
Inserting this in \eqref{e:paprika} and using $\varpi^{-1} \alpha t^2 \le |x|$, we obtain (using $s \ge \s$) 
\be
v_E(x) \le C t^{-2}.
\ee
Next, suppose that $ |x| \le\varpi^{-1} \alpha t^2$. Then 
\be
\frac{ t^2 }{x} \ge  \frac{\varpi}{\alpha}.
\ee 
By decreasing $\alpha$ if necessary, we have for all $| y | \le  \alpha t^2 /|x|$ that 
\be\label{e:necessaryforrhoe}
\left| 
 y  - \frac{t^2}{x} +E 
\right| \ge 2.
\ee 
By a change of variables, \Cref{l:vupperone}, and the inequality $|x|^{-1} \le |x|^{-5/6}$ for $|x| \ge 1$,
\begin{align*}
\int_{\varpi \le |y|<\alpha t^2/ |x|  } v_E(y) \, dy
&\le 
\tilde C \int_{\varpi  \le |y|<\alpha t^2/ |x|  } |y|^{-2 + s}  \, dy\\
&\le \tilde C \int_{1 \le |y|<\alpha t^2/ \varpi |x|  } |w|^{-2 + s}  \, dw \\
&\le \tilde C \int_{1 \le |y|<\alpha t^2/ \varpi |x|  } |w|^{-5/6}  \, dw\\
&\le  \tilde C  \left(\frac{\alpha t^2}{ \varpi |x|}\right)^{1/6}.
\end{align*}
Then using \Cref{l:Kuniformbound}, \eqref{e:necessaryforrhoe}, and the previous display,
\begin{align}
 \int_{\varpi  \le |y|<\alpha t^2/ |x| } v_E(y) \rho_E (y-t^2/x) \, dy 
 &\le \tilde C 
\left| \frac{x}{t^2} \right|^{4/3 }  \int_{\varpi \le |y|<\alpha t^2/ |x|  } v_E(y) \, dy\\
& \le \tilde C  \alpha^{1/6} \left| \frac{x}{t^2} \right|^{7/6 } \\
& = \tilde C  \alpha^{1/6} \left| \frac{x}{t^2} \right|^{2-s  } \left| \frac{x}{t^2} \right|^{s -5/6   } \\
&\le \tilde C  \alpha^{1 - s} \left| \frac{x}{t^2} \right|^{2-s} \le  \tilde C   \left| \frac{x}{t^2} \right|^{2-s}  .\label{e:integralp1}
\end{align}
In the last line, we used that our assumption $|x| \le \Delta$ implies that
\be\label{e:useassumption}
\left| \frac{x}{t^2} \right| \le \frac{1}{\alpha}, 
\ee 
Further, by \eqref{e:stepA} and \eqref{e:rhoEuniform}, 
\begin{align}
 \int_{ |y|<\varpi   } v_E(y) \rho_E (y-t^2/x) \, dy 
 &\le \tilde C 
\left| \frac{x}{t^2} \right|^{4/3 }  \int_{ |y|<\varpi  } v_E(y) \, dy\\
&\le \tilde C ( \log K) \left| \frac{x}{t^2} \right|^{4/3 }\\ 
&\le \tilde C ( \log K) \left| \frac{x}{t^2} \right|^{2-s}  \left| \frac{x}{t^2} \right|^{s - 2/3} \\
&\le \tilde C \alpha^{2/3-s}  ( \log K)  \left| \frac{x}{t^2} \right|^{2-s}.\label{e:integralp2}
\end{align}
The last line again used \eqref{e:useassumption}. 

Then \eqref{e:integralp1} and \eqref{e:integralp2} imply, together with  \eqref{e:tbounds}, that there exists $s_0(\alpha, K) \in (0,1)$ and a constant $\tilde C>0$ (not depending on $\alpha$) such that for all $s\in (\s, 1)$
\be
 \int_{ |y|<\alpha t^2/ |x| } v_E(y) \rho_E (y-t^2/x) \, dy  
\le \tilde C \alpha^{2/3 -s } t^{-2} |x|^{2-s} \ln K.
\ee
Inserting the previous bound into \eqref{e:paprika}, we obtain \eqref{e:salt}.
 \end{proof}

\subsection{Approximate Eigenvector}\label{s:approximateeigenvector}

The following bounds determines the shape of $v_E$, as described in the introduction to the previous subsection. We recall that $\tilde u$ was defined in \eqref{e:tildeu}. 
\begin{lemma}\label{l:eigenvectorsandwich}
 There exists  $\alpha_0 >0$ and $K_0 >0$ such that the following holds.
For all  $\alpha \in (0, \alpha_0]$, 
 there exists a constant $C(\alpha)> 0$ such that the following holds.
 For all $K \ge K_0$, 
 $E \in I_\star$, $s \in (\s, 1)$, 
and $x\in \R$,
\be\label{e:sandwichlowerbound}
 \tilde u(x) \le C v_E(x)
\ee
and 
\be\label{e:sandwichupperbound}
(F v_E)(x) \le C (F \tilde u)(x).
\ee 
\end{lemma} 
\begin{proof}
The proof consists of three  steps. \\
\paragraph{\textit{Step 1.}} 
 The lower bound \eqref{e:sandwichlowerbound} on $v_E$ holds for all sufficiently small $\alpha$ by 
\Cref{l:sumac} and \Cref{l:sumac2} (after choosing $\beta = \alpha^{-1}$ in \Cref{l:sumac2}). \\

\paragraph{\textit{Step 2.}} 
We now prove 
\be\label{e:peanutprelude}
C (F \tilde u)(x) \ge (Fv_E)(x).
\ee 
 for $|x| \ge \Delta$.
Suppose that $|x| \ge \Delta$. 
By  \eqref{e:peppercorn} and the Lipschitz continuity of $\rho_E$ (\Cref{l:bapstcollection}), we may choose  $\alpha_0(\L) >0$ such that for all $\alpha \in (0, \alpha_0]$, 
\begin{align*}
(F \tilde u )(x) & =
\frac{t^{2-s}}{ |x|^{2-s}}\left( 
\int_{\mathbb{R}}
\rho_E \left( - y - \frac{t^2}{x} \right) \tilde u (y) \, dy
\right)\\
&\ge 
\frac{t^{2-s}}{ |x|^{2-s}}
\left( 
\int_{\Delta\le | y| \le 1 }
\left( \rho_E \left(  - \frac{t^2}{x} \right)  -  C |y| \right) \frac{1}{|y|^{2-s}} \, dy\right)\\
&\ge 
\frac{C t^{}}{ |x|^{2-s}}\left( 2 c \left( \frac{1 - \Delta^{s-1}}{s-1}  \right) - C \right)\\
&\ge \frac{c}{K |x|^{2-s}} \ge \frac{ c v_E(x)}{K}.
\end{align*}
In the fourth line, we also used \eqref{e:weirdlog} and \Cref{l:vupperone}. We conclude using \Cref{c:roughlambda}, and using that $v$ is an eigenvector of $F$,  that \eqref{e:peanutprelude} holds for all $|x| \ge \Delta$.\\

\paragraph{\textit{Step 3.}} 
We now complete the proof by establishing \eqref{e:peanutprelude} for $|x| \le \Delta$.  First, note that
\begin{align*}
(F \tilde u )(x) & =
\frac{t^{2-s}}{ |x|^{2-s}}\left( 
\int_{\mathbb{R}}
\rho_E \left( - y - \frac{t^2}{x} \right) \tilde u (y) \, dy
\right)\\
&\ge \frac{t^{2-s}}{ |x|^{2-s}}\left( 
\int_{|y| \ge \Delta}
\rho_E \left( - y - \frac{t^2}{x} \right) \tilde u (y) \, dy
\right),
\end{align*}
where the inequality uses that the integrand is positive. 
By the definition of $\tilde u$ and \Cref{l:vupperone},
\begin{align}\label{e:peanut1}
\frac{t^{2-s}}{ |x|^{2-s}}\left( 
\int_{|y| \ge \Delta}
\rho_E \left( - y - \frac{t^2}{x} \right) \tilde u (y) \, dy
\right) \ge \frac{c t^{2-s}}{ |x|^{2-s}}\left( 
\int_{|y| \ge \Delta}
\rho_E \left( - y - \frac{t^2}{x} \right) v_E (y) \, dy
\right). 
\end{align}
Further, 
we have for $|x| \le \Delta$ that 
\begin{align}\label{e:beef1}
\begin{split}
& \frac{t^{2-s}}{ |x|^{2-s}}\left( 
\int_{|y| \ge \Delta}
\rho_E \left( - y - \frac{t^2}{x} \right) \tilde u (y) \, dy
\right) \\ & \ge \frac{t^{2-s}}{ |x|^{2-s}}\left( 
\int_{ - 3t^2/2x }^{ -  t^2/2x }
\rho_E \left( - y - \frac{t^2}{x} \right) \tilde u  (y) \, dy
\right) \\
&\ge 
 \frac {c t^{2-s}}{ |x|^{2-s}} \left|  \frac{x^{2-s} }{t^{2(2-s)}} \right|
\left( 
\int_{ - 3t^2/2x }^{ -  t^2/2x }
\rho_E \left( - y - \frac{t^2}{x} \right)  \, dy \right)\\
&\ge c t^{-1}\int_{- t^2/2x  }^{  t^2/2x } 
\rho_E (y) \, dy \\  &\ge  c t^{-1} \int_{- \alpha/2  }^{   \alpha/2 } 
\rho_E (y) \, dy  \ge c t^{-1}.
\end{split}
\end{align}
The first inequality follows from restricting the region of integration. The second inequality uses the definition of $\tilde u$. The third inequality is an algebraic simplification, and the last inequality uses \eqref{e:peppercorn}.

We also have, using $s > 2/3$, that
\begin{align}
\begin{split}
\frac{t^{2-s}}{ |x|^{2-s}}\left( 
\int_{|y| \le \Delta}
\rho_E \left( - y - \frac{t^2}{x} \right) v_E(y) \, dy
\right)& \le
\frac{C }{ t^{s} |x|^{2-s}}  
\int_{|y| \le  \Delta}
\rho_E \left( - y - \frac{t^2}{x} \right)  \, dy\\
&\le C  t^{-1} |x|^{-2 +s} t^2 \left| \frac{x}{t^2} \right|^{4/3} \\% \max_{|y| \le \Delta} \rho_E \left( - y - \frac{t^2}{x} \right) .\\
&\le C  |x|^{-2/3 + s} t^{-5/3}
\\
& \le  C t^{-1 + 2 (s - 1) }  
\le Ct^{-1}. \label{e:beef2}
\end{split}
\end{align}
The first inequality uses \eqref{e:salt}, the second inequality uses $|x| \le \Delta$ and \eqref{e:rhoEuniform}, and the last line follows from algebraic simplifications using $|x| \le \Delta$ and \eqref{e:tbounds}. 

We conclude, combining \eqref{e:beef1} and \eqref{e:beef2}, that
\be\label{e:peanut2}
\frac{Ct^{2-s}}{ |x|^{2-s}}\left( 
\int_{|y| \ge \Delta}
\rho_E \left( - y - \frac{t^2}{x} \right) \tilde u (y) \, dy
\right)
 \ge \frac{t^{2-s}}{ |x|^{2-s}}\left( 
\int_{|y| \le \Delta}
\rho_E \left( - y - \frac{t^2}{x} \right) v_E(y) \, dy
\right).
\ee
Combining \eqref{e:peanut1} and \eqref{e:peanut2}, we obtain that for all $|x| \le \Delta$,
\be\label{e:peanutsequel}
C (F \tilde u)(x) \ge (Fv_E)(x).
\ee 
Together with \eqref{e:peanutprelude}, we see that the same bound holds for all $x\in \R$. This completes the proof.
\end{proof}
  
\section{Test Function Estimates}\label{s:testestimates}

This section is devoted to the proof of \Cref{l:pizza}, which controls the difference $F_{E_1}  -F_{E_2}$ (where $F_E$ was defined in \eqref{e:Fdef}). These bounds will be instrumental in \Cref{s:conclusion}.
We begin in \Cref{s:preliminarytest} by considering the action of $F$ on $\tilde u_1$, $u_2$, and $u_3$ with the choice $t= g (K \log K)^{-1}$. We then present the proof of \Cref{l:pizza} in \Cref{s:differenceF}. 

\subsection{Preliminary Test Vector Bounds} \label{s:preliminarytest}
For entirety of this section, we  fix  $\L>1$, $g>0$, $\rho$, and $E_0 \in \R$ satisfying the conditions of \Cref{t:mainlambda}. 
We recall that the vector $\tilde u$ was defined in \eqref{e:tildeu}. 
We also recall the definitions of $u_2$ and $u_3$ from \eqref{e:theus}, and set 
\be\label{e:tildeu1}
\tilde u_1(x) = \frac{1}{t^2 \ln K}  \one\{ |x| \le \Delta\},
\ee
noting that 
\begin{align*}
\tilde u(x) = \tilde u_1(x) + u_2(x) + u_3(x).
\end{align*}
Further, recall that $\s$ was defined in \Cref{s:parameterchoices}.

The following lemma provides initial bounds on the action of $F$ on $\tilde u_1$, $u_2$, and $u_3$. 
It is a slight refinement of the bounds in \Cref{s:testvectoranalysis}, where we studied $u_1$ instead of $\tilde u_1$. As before, the dominant growth behavior comes from the action of $F$ on $u_2$. The lemma shows that $F\tilde u_1$ and $F u_3$ are asymptotically smaller (in $K$) than $K^{-1} \tilde u$, while we know that the spectral radius of $F$ is order $K^{-1}$ (from \Cref{c:KR} and \Cref{l:rougheigenvector}).
\begin{lemma}\label{l:tildebounds}
Fix $\eps ,G>0$. 
There exists a constant $\alpha_0(\eps) >0$ such that the following holds for all $\alpha \in (0, \alpha_0]$.
There are constants $C(G, \alpha)$ and $K_0(G, \alpha) >0$ such that 
the following holds for all  
 $s\in (\s,1)$, $g\in [0, G]$, $K>1 $, and $E \in [-G,G]$. 
For all $x \in \R$, 
\begin{align}\label{e:tildeu1bound}
(F \tilde u_1 )(x) \le&   C( t \ln K)^{-1}  \one\{ |x| \le \Delta \}   +   \left( \frac{1}{\ln K}\right) \frac{C t }{|x|^{2-s}} \one\big\{|x| \ge \Delta\big\},
\end{align}
\begin{align}\label{e:tildeu2bound}
\begin{split}
(F u_2 )(x) \le&C  t^{-1}
\log  \left( \frac{1}{\Delta} \right)  \one\{ |x| \le \Delta \}  \\
& + \frac{C t}{|x|^{2-s}}
\left(
2
\left( \rho_E(0)  + \eps \right)\log  \left( \frac{1}{\Delta} \right) 
 +  C 
\right)  \one\big\{|x| \ge \Delta\big\},
\end{split}
\end{align}
and
\begin{align}\label{e:tildeu3bound}
(F u_3 )(x) \le&C  t^{-1} \one\{ |x| \le \Delta \}   +\frac{C t}{|x|^{2-s}} \one\big\{|x| \ge \Delta\big\}.
\end{align}
\end{lemma}
\begin{proof}
The bound \eqref{e:tildeu1bound} follows from \Cref{l:u1upper}, after recalling the definition of $u_1$ in \eqref{e:theus} and dividing both sides by $\log K$; we also recall that $\s$ was chosen so that $s \in (\s,1)$ implies \eqref{e:weirdlog} and \eqref{e:tbounds}.
Similarly, the bound \eqref{e:tildeu2bound} follows from \Cref{l:u2upper}, 
and bound \eqref{e:tildeu3bound} follows from \Cref{l:u3upper}. 
\end{proof}

\begin{corollary}\label{c:Ftildeu}
Fix  $G>0$. 
There exists a constant $\alpha_0 >0$ such that the following holds for all $\alpha \in (0, \alpha_0]$.
There are constants $C(G, \alpha)$ and $K_0(G, \alpha) >0$ such that for all $K\ge K_0$, $s \in (\s,1)$, $g \in [0,G]$, and $E \in[-G, G]$, we have for all $x \in \R$ that
\begin{align}
(F \tilde u)(x) \le& C  t^{-1}
\log  \left( \frac{1}{\Delta} \right)  \one\{ |x| \le \Delta \}  \\
& + \frac{C t }{|x|^{2-s}}
  \log  \left( \frac{1}{\Delta} \right)  \one\big\{|x| \ge \Delta\big\}.
\end{align}
\end{corollary}
\begin{proof}
This follows by combining the first three estimates in \Cref{l:tildebounds}, after setting $\epsilon = 1$ and using $\| \rho_E \|_\infty \le \| \rho \|_\infty$ from  \eqref{e:rhoEupper} and the assumption that $\rho$ is $\L$-regular.
\end{proof}

\subsection{Estimate on the Difference Operator}\label{s:differenceF}
The next lemma is our main result on the action of $F^{\circ}$. We recall that $\varpi$ and $I_\star$ were defined in \eqref{e:varpi} and \eqref{e:istar}. Define $F^\circ = F^\circ_{K,t,s,E_1, E_2}$ by
\be\label{e:Fcirc}
(F^\circ u)(x) = 
\frac{t^{2-s}}{|x|^{2-s}}
\int_{-\infty}^\infty
\left( \rho_{E_2} \left( - y - \frac{t^2}{x} \right)  - \rho_{E_1} \left( - y - \frac{t^2}{x} \right) \right)u(y) \, dy.
\ee
We will use the shorthand
\be
\rho^\circ(x) = \rho_{E_2} (x)  - \rho_{E_1} (x).
\ee 

We now discuss the meaning of the following lemma in the case that $\rho$ is decreasing near a point $E_0$ such that $\rho(E_0) = (4g)^{-1}$; the case when $\rho$ is increasing is analogous. 
The bound \eqref{e:upperkey1} shows that $F^\circ$ yields a function that is negative for $|x| \ge \Delta$  when applied to some iterate $F^j \tilde u$. The confirms the heuristic that $F^\circ$ should act like a strictly negative operator when \eqref{e:le111} holds (from the end of \Cref{LambdaMonotone}). 
However,  $F^\circ F^j \tilde u$ possesses a small positive part for $|x| \le \Delta$; we estimate this exceptional positive contribution in \eqref{e:smalldeltabound}. Together, these bounds will allow us to analyze the terms $F_{E_2}^{n-j -1} (F^\circ) F_{E_1}^{j}$ arising in our introductory proof sketch (from \eqref{e:Fdiffsumintro}) in the proof of \Cref{l:penultimate} below. 
\begin{lemma}\label{l:pizza}
Set $I_\L  =  [E_0 - \L^{-1}, E_0 + \L^{-1}]$. 
Suppose that either
\be\label{e:ge111}
\rho'(E)  \ge \frac{1}{\L}
\text{ for all $E \in I_\L$ }
\ee 
or
\be\label{e:le111}
\rho'(E)  \le  - \frac{1}{\L} \text{ for all $E \in I_\L$ }
\ee 
There exists a constant $\alpha_0 >0$ such that the following is true for all $\alpha \in (0, \alpha_0]$ and integers $j \ge 0$.
There exist constants 
$K_0( \alpha),C(\alpha) , c(\alpha)>0$ such that the following bounds hold for all  $s\in (\s ,1)$,  $K\ge K_0$, and $E_1, E_2 \in I_\star$ with $E_1 < E_2$. 
\begin{enumerate}
\item
For $|x| \ge \Delta$,  if \eqref{e:ge111} holds, then
\begin{align} \label{e:lowerkey1}
(F^\circ F^j_{E_1} \tilde u)(x) \ge 
c \lambda_{E_1}^{j}  \log \left(\frac{1}{\Delta} \right)  \frac{t }{|x|^{2-s}}| E_1 - E_2|   ,
\end{align}
and if \eqref{e:le111} holds, then
\begin{align}\label{e:upperkey1}
(F^\circ F^j_{E_1} \tilde u)(x) \le - 
c \lambda_{E_1}^{j}  \log \left(\frac{1}{\Delta} \right)  \frac{t }{|x|^{2-s}}
 | E_1 - E_2|.
\end{align}
\item
For $|x| \le \Delta$, 
\be\label{e:smalldeltabound}
\big| (F^\circ F^j_{E_1} \tilde u)(x)  \big| \le C \lambda_{E_1}^{j-1} 
 \ln \left( \frac{1}{\Delta} \right)^2 |E_1 - E_2|.
\ee 

\end{enumerate}
\end{lemma}
\begin{proof}
We begin with the proof of \eqref{e:upperkey1}, starting with the case $j \ge 1$. The proof of \eqref{e:lowerkey1} is analogous and hence omitted. 
For notational convenience, we define
\begin{align*}
w_1(x) &= \big[(F^j_{E_1} \tilde u)(x) \big] \one\{ |x| \le \Delta\},\\
w_2(x)& = \big[(F^j_{E_1} \tilde u)(x) \big] \one\{\Delta\le  |x| \le 1  \},\\
w_3(x) &= \big[(F^j_{E_1} \tilde u)(x) \big] \one\{1 \le  |x|   \}. 
\end{align*}
By \Cref{l:eigenvectorsandwich} (whose use is permitted by the assumption that $E_1 \in I_\star$), along with the positivity of $F$, we have that for all  $\alpha \in (0, \alpha_0]$, 
 there exists constants $C(\alpha),c(\alpha)> 0$ such that 
 \be\label{e:sandwich20}
c (F_{E_1}^{j} v_{E_1}) (x)  \le (F_{E_1}^{j} \tilde u) (x) \le C    (F^j_{E_1}v_{E_1} )  (x).
\ee 
In the first inequality, we applied $F$ a total of  $j-1$ times to each side of \eqref{e:sandwichupperbound}, 
and in the second inequality, we applied $F$ a total of $j$ times to each side of \eqref{e:sandwichlowerbound}. Using that $v_{E_1}$ is an eigenvector of $F_{E_1}$, we obtain 
\be\label{e:sandwich21}
c \lambda_{E_1}^j  v_{E_1} (x)  \le (F_{E_1}^{j} \tilde u) (x) \le C\lambda_{E_1}^{j-1}    (F_{E_1}v_{E_1} )  (x).
\ee 
Finally, using \eqref{e:sandwichlowerbound}  to lower bound the leftmost quantity in \eqref{e:sandwich21}, and \eqref{e:sandwichupperbound} to upper bound the rightmost quantity in \eqref{e:sandwich21}, we obtain
\be\label{e:sandwich2}
c \lambda^j_{E_1} \tilde u (x)  \le (F_{E_1}^{j} \tilde u) (x) \le C \lambda_{E_1} ^{j-1}   (F_{E_1}\tilde u)  (x).
\ee 

Next, we suppose that $| x| \ge \Delta$ and consider the action of $F^\circ$ on $w_1$, $w_2$, and $w_3$ separately. For the action on $w_1$, we have
\begin{align}
\begin{split}
\left|  \frac{t^{2-s} }{|x|^{2-s}}
\int_{-\infty}^\infty
\rho^\circ  \left(  - y - \frac{t^2}{x} \right) w_1(y)  \, dy \right| 
&\le    
C\lambda_{E_1}^{j-1}
\frac{t^{2-s} }{|x|^{2-s}}
\int_{|y| \le \Delta}
\left|\rho^\circ  \left(  - y - \frac{t^2}{x} \right)  \right|(F_{E_1} \tilde u)(y)  \, dy
\\
&\le 
C| E_1 - E_2|   \lambda_{E_1}^{j-1} \log\left( \frac{1}{\Delta} \right)
\frac{1}{|x|^{2-s}}
\int_{|y| \le \Delta}
1\, dy \\
&\le 
C | E_1 - E_2|   \lambda_{E_1}^{j-1} \log\left( \frac{1}{\Delta} \right)
\frac{t^{2}}{|x|^{2-s}}. \label{e:ingredient1}
\end{split}
\end{align}
The first inequality follows from the definition of $w_1$, \eqref{e:sandwich2}, and the positivity of $F_{E_1}$. The second inequality follows from \eqref{e:Elipschitz}, \eqref{e:tbounds}, and  \Cref{c:Ftildeu}. The third inequality follows from the definition of $\Delta$ and direct integration. 

Set $\kappa = (4\L)^{-1}$. For the action of $F^\circ$ on $w_2$, we consider the regimes $|y| \in [\kappa, 1]$ and $|y| \in [\Delta, \kappa)$ separately.
On the first, we have 
\begin{align}
\begin{split}
\label{e:ingredient3a}
& \left| \frac{t^{2-s}}{|x|^{2-s}}
\left(
\int_{\kappa \le |y| \le 1 }
\rho^\circ\left( - y   - \frac{t^2}{x}  \right) w_2(y) \, dy 
\right) \right| \\
&\le 
C |E_1 - E_2| \lambda_{E_1}^{j-1}  \frac{ t^{2-s}}{|x|^{2-s}}
\left(
\int_{\kappa \le |y| \le 1 }
(F_{E_1} \tilde u)(y) \, dy 
\right)  \\
&\le 
C |E_1 - E_2| \lambda_{E_1}^{j-1}   \log  \left( \frac{1}{\Delta} \right)  \frac{ct^{3-s}}{|x|^{2-s}}
\left(
\int_{\kappa \le |y| \le 1 }
|y|^{-2+s}  \, dy 
\right)  \\
&\le
C |E_1 - E_2| \lambda_{E_1}^{j-1}  \log  \left( \frac{1}{\Delta} \right)   \frac{ t^{2}}{|x|^{2-s}}
\end{split}
\end{align}
The first inequality follows from \eqref{e:Elipschitz}, \eqref{e:sandwich2}, and the positivity of $F_{E_1}$. The second inequality follows from bounding $F_{E_1} \tilde u$ using \Cref{c:Ftildeu}. 
The obtain the last inequality, we used \eqref{e:tbounds} and integrated directly.

Then considering the regime $|y| \in [\Delta, \kappa)$, we find  from \eqref{e:sandwich2} and the definition of $\tilde u$ in \eqref{e:tildeu} that 
\be\label{e:w2lower}
w_2(x) \ge c \lambda^j_{E_1} |x|^{-2 +s} \one\{\Delta\le  |x| \le 1  \}.
\ee 
Then we have 
\begin{align}\label{e:ingredient3b}
\begin{split}
&\frac{t^{2-s}}{|x|^{2-s}}
\left(
\int_{\Delta \le |y| \le \kappa }
\rho^\circ\left( - y   - \frac{t^2}{x}  \right) w_2(y) \, dy 
\right) \\
&\le 
-c|E_1 - E_2|\lambda_{E_1}^{j}  \frac{ t^{2-s}}{|x|^{2-s}}
\left(
\int_{\Delta \le |y| \le \kappa }
| y |^{-2 +s} \, dy 
\right)  \\
&\le 
- c|E_1 - E_2|\lambda_{E_1}^j  \frac{ t^{2-s}}{|x|^{2-s}}
\left(
2\left(\frac{\kappa^{1-s} - \Delta^{1-s}}{s -1 } \right)
\right)  \\
&=
- c|E_1 - E_2|\lambda_{E_1}^j  \frac{ t^{2-s}}{|x|^{2-s}}
\left(
2\left(\frac{1 - \Delta^{1-s}}{s -1 } \right)
-
2\left(\frac{1 - \kappa^{1-s}}{s -1 } \right)
\right)  \\
&\le 
-   c|E_1 - E_2| \lambda_{E_1}^j  \log \left(\frac{1}{\Delta} \right)  \frac{t }{|x|^{2-s}}. 
\end{split}
\end{align}
In the first inequality, we used \eqref{e:w2lower} to bound $w_2$ below. We also used assumption $\eqref{e:le111}$, the restriction $|y| \le \kappa$, the assumption that $|x| \ge \Delta$, and \eqref{e:le11finite2}  to bound $\rho^\circ$ above (by a negative quantity). 
In the second inequality, we integrated directly, and the third line is an algebraic identity.
In the third inequality, we used \eqref{e:weirdlog} and the elementary estimate 
\be
\sup_{s \in \s} \left|\frac{1 - \kappa^{1-s}}{s -1 } \right| \le C.
\ee

Further, we have 
\begin{align}
\left| \frac{t^{2-s} }{|x|^{2-s}}
\int_{-\infty}^\infty
\rho^\circ \left(  - y - \frac{t^2}{x} \right) w_3(y)  \, dy \right|
&\le    
C \lambda_{E_1}^{j-1}
\frac{t^{2-s} }{|x|^{2-s}}
\int_{|y| \ge 1  }
\left| \rho^\circ \left(  - y - \frac{t^2}{x} \right) (F_{E_1} \tilde u)(y) \right|  \, dy \notag
\\
&\le 
C \lambda_{E_1}^{j-1} \log\left( \frac{1}{\Delta} \right)
\frac{t^{2 }}{|x|^{2-s}}
\int_{|y| \ge 1 }
\left| \rho^\circ  \left(  - y - \frac{t^2}{x} \right) \right| \frac{  dy }{|y|^{2 -s } } \notag  \\
&\le 
C|E_1 - E_2|  \lambda_{E_1}^{j-1} \log\left( \frac{1}{\Delta} \right)
\frac{t^{2}}{|x|^{2-s}}.\label{e:ingredient4}
\end{align}
The first inequality uses the definition of $w_3$ and \eqref{e:sandwich2}.
The second inequality uses \eqref{e:weirdlog} and bounds  $F_{E_1} \tilde u$ using \Cref{c:Ftildeu}. 
In the third inequality, we used 
 \begin{align*}
\int_{| y| \ge 1 }
\left|
\rho^\circ \left(  - y - \frac{t^2}{x} \right) \right| \frac{dy}{ |y|^{2-s}}
&\le 
C |E_1 - E_2|
\int_{ |y| \ge 1 }
 \frac{dy}{(1 + |y|)^{1/\L} |y|^{2-s}}
\ \le C|E_1 - E_2| ,
\end{align*}
where the first inequality follows from applying  \eqref{e:Elipschitz} and using 
$|x| \ge \Delta$, $|y| \ge 1$ to show that 
\be
 \left| - y - \frac{t^2}{x} \right| \ge \frac{|y|}{2},
\ee 
and the second inequality is a direct integral estimate.
Combining \eqref{e:ingredient1}, \eqref{e:ingredient3a}, \eqref{e:ingredient3b}, and \eqref{e:ingredient4} gives \eqref{e:upperkey1} for $j \ge 1$, after noting that $\lambda_{E_1}^j t \ge C \lambda_{E_1}^{j-1} t^2 (\log K)$ by \Cref{c:roughlambda} and the definition of $t$ in \eqref{e:tKdef}. The use of \Cref{c:roughlambda} is permissible for all $s \ge \s$ by the choice of $\s$ in \Cref{s:parameterchoices}.

The $j=0$ case of \eqref{e:upperkey1} follows by a similar argument. Namely, in \eqref{e:ingredient1}, \eqref{e:ingredient3a}, \eqref{e:ingredient3b}, and \eqref{e:ingredient4}, $w_1+w_2+w_3$ is replaced by $\tilde{u}$. So, one can replace all uses of the bound \eqref{e:sandwich2} used above  with the explicit definition of $\tilde{u}$ from \eqref{e:tildeu}; we omit further details.

Next, we turn to the proof of \eqref{e:smalldeltabound}, beginning with the case $j \ge 1$. 
Suppose that  $| x| \le  \Delta$. We have
\begin{align}\label{e:small1}
\begin{split}
&\frac{t^{2-s} }{|x|^{2-s}}
\left| \int_{-\infty}^\infty
\rho^{\circ} \left(  - y - \frac{t^2}{x} \right) w_1(y)  \, dy\right|  \\
&\le 
C \lambda_{E_1}^{j-1} 
\frac{t^{2-s}}{|x|^{2-s}}
\int_{|y| \le \Delta}
\left| \rho^{\circ} \left(  - y - \frac{t^2}{x} \right)  \right| (F_{E_1} \tilde u)(y)  \, dy \\
&\le 
C \lambda_{E_1}^{j-1} t^{-1} \log\left( \frac{1}{\Delta} \right)
\frac{t^{2 - s}}{|x|^{2-s}} 
\int_{|y| \le \Delta}
\left| \rho^{\circ} \left(  - y - \frac{t^2}{x} \right) \right|  \, dy \\
&= C \lambda_{E_1}^{j-1} t^{-1}  \log\left( \frac{1}{\Delta} \right)
t^{-(2-s)}
\int_{|y| \le \Delta}
\left| \rho^{\circ} \left(  - y - \frac{t^2}{x} \right) \right|  \left( \frac{t^2}{|x|}\right)^{2-s} \, dy \\
&\le C|E_1 - E_2|   \lambda_{E_1}^{j-1}  \log\left( \frac{1}{\Delta} \right).
\end{split}
\end{align}
The first inequality follows from \eqref{e:sandwich2} and the positivity of $F_{E_1}$. 
The second inequality follows from \Cref{c:Ftildeu}, and the final inequality comes from \eqref{e:tbounds}, \Cref{l:supboundscirc}, and the definition of $\Delta$. 

Next, defining $z =  - t^2/ x$, we have 
\begin{align}
\begin{split}
&\frac{t^{2-s}}{|x|^{2-s}}
\left| 
\int_{-\infty}^\infty
\rho^\circ \left(  - y - \frac{t^2}{x} \right) w_2(y)\, dy  \right| \\
&\le 
C \lambda_{E_1}^{j-1} 
\frac{t^{2-s}}{|x|^{2-s}}
\int_{\Delta \le |y| \le 1 }
\left|
\rho^\circ \left(  - y - \frac{t^2}{x} \right) \right| (F_{E_1} \tilde u)(y) \, dy  \\
&\le 
C \lambda_{E_1}^{j-1}  \ln \left( \frac{1}{\Delta} \right) \frac{t^{2}}{|x|^{2-s}}
\left(
\int_{\Delta \le |y| \le 1 } \frac{dy}{|y|^{2-s}}
\right) \left(\sup_{|y| \le 1 } \left|   \rho^\circ \left( - y  - \frac{t^2}{x}  \right) \right| \right) \\
&\le 
C \lambda_{E_1}^{j-1} 
\frac{t^2 }{|x|^{2-s}}  \ln \left( \frac{1}{\Delta} \right)^2 \cdot \sup_{|y| \le 1} \left|  \rho^\circ \left( - y - \frac{t^2}{x} \right) \right| \\
&=  C \lambda_{E_1}^{j-1}   \ln \left( \frac{1}{\Delta} \right)^2   \cdot \sup_{|y| \le 1 }  \big| \rho^\circ (y+z) \big| | z|^{2-s } \\
 &\le 
C|E_1 - E_2| \lambda_{E_1}^{j-1}  \ln \left( \frac{1}{\Delta} \right)^2 .  \label{e:small2}
\end{split}
\end{align}
The first inequality uses  \eqref{e:sandwich2} and the positivity of $F_{E_1}$. 
The second inequality uses   \Cref{c:Ftildeu}, \eqref{e:tbounds}, and bounds $| \rho^\circ|$ by its supremum on the domain of integration.
The third inequality follows from a direct integration, and  last line follows from  \Cref{l:supboundscirc}.

Finally, again setting  $z =  - t^2/ x$, we have 
\begin{align}
\begin{split}
&\frac{t^{2-s}}{|x|^{2-s}}\left|  
\int_{-\infty}^\infty
\rho^\circ  \left(  - y - \frac{t^2}{x} \right)   w_3 (y)\, dy \right| \\ 
&\le 
C \lambda_{E_1}^{j-1} 
\frac{t^{2-s}}{|x|^{2-s}}
\int_{|y| \ge 1 }
\left| \rho^\circ  \left(  - y - \frac{t^2}{x} \right) \right| (F_{E_1} \tilde u)(y)   \, dy \\
&\le
C \lambda_{E_1}^{j-1} 
\frac{t^{2}}{|x|^{2-s}}
 \ln \left( \frac{1}{\Delta} \right) \int_{|y| \ge 1 }
\left| \rho^\circ \left(  - y - \frac{t^2}{x} \right) \right|  |y|^{-2+s} \, dy \\
&=
C \lambda_{E_1}^{j-1} 
 \ln \left( \frac{1}{\Delta} \right)
\int_{|y| \ge 1 }
\big|
\rho^\circ (  - y  + z ) \big|  \frac{|z|^{2-s} }{|y|^{2-s}}  \, dy \\
&\le
C |E_1 - E_2| \lambda_{E_1}^{j-1}  \ln \left( \frac{1}{\Delta} \right) .\label{e:small3}
\end{split}
\end{align}
The first inequality uses  \eqref{e:sandwich2} and the positivity of $F_{E_1}$. 
The second inequality uses  \Cref{c:Ftildeu} and \eqref{e:tbounds}. 
The last inequality uses \Cref{l:supboundscirc}.
Then \eqref{e:smalldeltabound} follows from combining \eqref{e:small1}, \eqref{e:small2}, and \eqref{e:small3}. 

Finally, we note that the proof of the $j=0$ case of \eqref{e:smalldeltabound} is nearly identical to the $j \ge 1$ case, except we replace every appearance of $w_1+w_2+w_3$ by $\tilde{u}$. So, one can again replace all uses of the bound \eqref{e:sandwich2} used above  with the explicit definition of $\tilde{u}$ from \eqref{e:tildeu}. We omit the details.
\end{proof}

We also prove an analogue of the previous lemma for $v_E$, which will be useful in the next section. 
\begin{lemma}
Set $I_\L  =  [E_0 - \L^{-1}, E_0 + \L^{-1}]$. 
Suppose that either
\be\label{e:ge1111c}
\rho'(E)  \ge \frac{1}{\L}
\text{ for all $E \in I_\L$ }
\ee 
or
\be\label{e:le1111c}
\rho'(E)  \le  - \frac{1}{\L} \text{ for all $E \in I_\L$ }
\ee 
There exists a constant $\alpha_0 >0$ such that the following is true for all $\alpha \in (0, \alpha_0]$. 
There exist constants 
$K_0( \alpha),C(\alpha), c(\alpha) >0$ such that the following bounds hold for all  $s\in (\s ,1)$,  $K\ge K_0$, and $E_1, E_2 \in I_\star$ with $E_1 < E_2$. 
 \begin{enumerate}
\item
For $|x| \le \Delta$, 
\be\label{e:coffee10}
\big|  (F^\circ  v_{E_1} ) (x) \big|
\le 
C t^{-1}  \ln \left( \frac{1}{\Delta} \right) |E_1 - E_2|
\ee 
\item
For $|x| \ge \Delta$,  if \eqref{e:ge1111c} holds, then
\begin{align} 
(F^\circ v_{E_1} ) (x)  \ge  c   |E_1 - E_2|  \ln\left( \frac{1}{\Delta} \right)  \frac{t}{|x|^{2-s}},
\end{align}
and if \eqref{e:le1111c} holds, then
\begin{align}\label{e:coffee2}
(F^\circ v_{E_1} ) (x)  \le - c   |E_1 - E_2|  \ln\left( \frac{1}{\Delta} \right)  \frac{t}{|x|^{2-s}}.
\end{align}
\end{enumerate}
\end{lemma}
\begin{proof}
We suppose that \eqref{e:le1111c} holds; the proof for the case where \eqref{e:ge1111c} holds is similar and hence omitted. 
Note that because $\lambda_{E_1}$ is an eigenvalue of $F_{E_1}$, 
\begin{align}\label{e:prologue}
F^\circ v_{E_1} = \lambda_{E_1}^{-1} F^\circ ( F_{E_1} v_{E_1} )
\end{align}
When $|x| \le \Delta$, we obtain from the definition of $F^\circ$ that 
\begin{align}\label{e:coffee1}
\begin{split}
\big|  F^\circ ( F_{E_1} v_{E_1} ) (x) \big| &\le    \frac{t^{2-s} }{|x|^{2-s}}
\int_{-\infty}^\infty 
\left| \rho^\circ  \left(  - y - \frac{t^2}{x} \right)  \right|( F_{E_1} v_{E_1}) (y) \, dy \\
&\le 
  \frac{C t^{2-s} }{|x|^{2-s}}
\int_{-\infty}^\infty 
\left| \rho^\circ  \left(  - y - \frac{t^2}{x} \right)  \right|( F_{E_1} \tilde u ) (y) \, dy \\
&\le C \ln\left( \frac{1}{\Delta} \right)^2 |E_1 - E_2|.
\end{split}
\end{align}
The first inequality follows from the positivity of $F_{E_1}$ and $v_{E_1}$, and the definition of $F^\circ$. The second inequality follows from \eqref{e:sandwichupperbound}. 
The proof of the last inequality is identical to the proof of \eqref{e:smalldeltabound}, so we omit it. Combining this inequality with \eqref{e:prologue} gives that, for $|x| \le \Delta$,
\begin{align}
\big| (F^\circ v_{E_1} ) (x) \big| & \le C \lambda_{E_1}^{-1}  \ln\left( \frac{1}{\Delta} \right)^2 |E_1 - E_2|
\le C t^{-1} \ln\left( \frac{1}{\Delta} \right) |E_1 - E_2|,
\end{align}
where we used \Cref{c:roughlambda} and the definition of $t$ and $\Delta$ in \eqref{e:tKdef} in the last inequality. This shows \eqref{e:coffee10}.

We next consider the case that $|x| \ge \Delta$. By \Cref{l:eigenvectorsandwich},
\be\label{e:sandwich1}
c (F v_{E_1})(x)  \le (F \tilde u) (x) \le C (F v_{E_1})(x)
\ee 
for all $x \in \R$.  
Set $\kappa = (5 \L)^{-1}$ and $J = \{ y \in \R : y \in [\Delta, \kappa] \}$. 
Then 
\begin{align}\label{e:ingredient3b0}
\begin{split}
&\frac{t^{2-s}}{|x|^{2-s}}
\left(
\int_{ J  }
\rho^\circ\left( - y   - \frac{t^2}{x}  \right) (F_{E_1} v_{E_1})(x) \, dy 
\right) \\
&\le 
\frac{c t^{2-s}}{|x|^{2-s}}
\left(
\int_{ J  }
\rho^\circ\left( - y   - \frac{t^2}{x}  \right) (F_{E_1} \tilde u )(x) \, dy 
\right)  \\
&\le 
-   c|E_1 - E_2| \lambda_{E_1}  \log \left(\frac{1}{\Delta} \right)  \frac{t }{|x|^{2-s}}. 
\end{split}
\end{align}
The first inequality follows from bounding $F_{E_1} v_{E_1}$ below using \eqref{e:sandwich1}; we also used assumption \eqref{e:le1111c}, the restriction $|y| \le \kappa$, the assumption that $|x| \ge \Delta$, and \eqref{e:le11finite2}  to bound $\rho^\circ$ above (by a negative quantity).
The second inequality is \eqref{e:ingredient3b}. 

By a similar argument, using the upper bound for $F_{E_1} v_{E_1}$ in \eqref{e:sandwich1}, \eqref{e:ingredient1}, \eqref{e:ingredient3a}, and \eqref{e:ingredient4}, we find that 
\begin{align}\label{e:newingredient}
\frac{t^{2-s}}{|x|^{2-s}}
\left|
\int_{ J  }
\rho^\circ\left( - y   - \frac{t^2}{x}  \right) (F_{E_1} v_{E_1})(x) \, dy 
\right| \le 
C |E_1 - E_2|  \log  \left( \frac{1}{\Delta} \right)   \frac{ t^{2}}{|x|^{2-s}}.
\end{align}
Using the definition of $F_{E_1}$, $(F_{E_1} v_{E_1}) = \lambda_{E_1} v_{E_1}$,  \Cref{c:roughlambda}, \eqref{e:ingredient3b0}, \eqref{e:newingredient}, and the definitions of $t$ and $\Delta$ in \eqref{e:tKdef}, we find that that for $|x| \ge \Delta$, \eqref{e:coffee2} holds. 
\end{proof}

\section{Monotonicity of the Leading Eigenvalue}\label{s:conclusion}
This section establishes \Cref{t:mainlambda}, which states, roughly speaking, that the monotonicity of $\rho$ near an energy $E_0$ implies the monotonicity of $\lambda_{s,E}$ for $E$ near $E_0$. 
 We begin in \Cref{s:provebootstrap} by proving \Cref{l:bootstrap}, which connects the free energy $\phi$ to the eigenvalue $\lambda_{s,E}$ at points $E \in \R$ where $\Im  R_{00}$ has a real boundary value. Then, in \Cref{s:monotonicitysub}, we prove \Cref{t:mainlambda}. 
 Throughout this section, we fix  $\L>1$, $g>0$, $\rho$, and $E_0 \in \R$ satisfying the conditions of \Cref{t:mainlambda}. We also recall the definition of $I_\star$ and choice of $\varpi$ from \eqref{e:istar} and \eqref{e:varpi}.

\subsection{Proof of \Cref{l:bootstrap}}\label{s:provebootstrap}

Set
\be
 \tilde w(x) = p_E \left( - \frac{x}{t^2} \right). \ee 
 The follow lemma allows us to use $\tilde w$ as a test vector for the purpose of computing $\lambda_{s,E}$ through iterated powers of $F_E$. We recall that $\tilde u$ was defined in \eqref{e:tildeu}.

\begin{lemma}\label{l:initialdatabound}
There exists a constant $\alpha_0 >0$ such that the following holds for all $\alpha \in (0, \alpha_0]$. There exists $K_0(\alpha) > 1 $ such that for all $s \in (5/6, 1)$, $K\ge K_0$, and $t,E \in \R$, there exist constants $C(K, E ,t ,s,\alpha), c(K,E,t,s,\alpha) > 0$ such that for all $x \in \R$,
\be\label{e:initialdatabound}
c  \tilde u  (x) \le (F^2 \tilde w)(x)\le C (F^2 \tilde u)(x).
\ee 
\end{lemma}
\begin{proof}
Using \eqref{e:pEquadratic}, there exists a constant $C(E,K) >0$ such that 
\be
p_E(x) \le C \tilde u(x),
\ee 
and the positivity of $F$ yields the upper bound in \eqref{e:initialdatabound}.
For the lower bound, we first suppose that $|x| \ge \Delta$. Again by \eqref{e:pEquadratic}, we have 
\be\label{e:pElowerproof}
p_E(x) \ge c u_2(x).
\ee 
Then
\be\label{e:Ftildew}
(F \tilde w)(x)  \ge \frac{c}{|x|^{2-s}} \ge u_2(x).
\ee 
The first inequality follows from \Cref{l:lowerbound} and \eqref{e:pElowerproof}, and the second is from the definition of $u_2$ in \eqref{e:theus}.
Then 
\be\label{e:Ftildew2}
(F^2 \tilde w)(x)   \ge c(Fu_2)(x) \ge c u_2(x).
\ee 
where the first inequality uses \eqref{e:Ftildew} and the positivity of $F$, and the second follows from  \Cref{l:lowerbound}. 

Next, we assume that $|x| \le \Delta$ and set $\delta = \alpha/2$. We have 
\begin{align*}
(F^2 \tilde w)(x)  &\ge  \frac{c t^{2-s}}{|x|^{2-s}}
\left(
\int_{ |y| \ge \Delta  }
\rho_E \left( - y  - \frac{t^2}{x}  \right) \frac{dy}{|y|^{2-s}} \right) \\
& \ge 
 \frac{c t^{2-s}}{|x|^{2-s}}
\left(
\int_{- t^2/x - \delta  }^{- t^2/x + \delta }
\rho_E \left( - y  - \frac{t^2}{x}  \right) \frac{dy}{|y|^{2-s}} \right) \\
 & \ge 
 \frac{c t^{2-s}}{|x|^{2-s}}
\left(
\int_{- t^2/x - \delta  }^{- t^2/x + \delta }
 \frac{dy}{|y|^{2-s}} \right) \\
 &\ge  \frac{c t^{2-s}}{|x|^{2-s}} \left| \frac{x}{t^2} \right|^{2-s} \ge t^{-2+s}.
\end{align*}
The first inequality above follows from the first inequality in \eqref{e:Ftildew} and the definition of $F$. 
The second inequality follows from restricting the region of integration; we use that $|x| \le \Delta$ implies $| t^2/x|  \ge \alpha$, and take $K$ large enough so that $\alpha \ge  2 \Delta$.
The third inequality uses that $\varpi$ was chosen so that \eqref{e:peppercorn} holds for all $x \in [ - 2 \varpi, 2\varpi]$. 
In the fourth inequality, we directly evaluated the integral. 
This completes the proof of the lower bound in \eqref{e:initialdatabound}. 
\end{proof}

\begin{proof}[Proof of \Cref{l:bootstrap}]
Throughout this proof, we suppose that $E$ is such that the limit $\varphi (s; E) = \lim_{\eta \rightarrow 0} \varphi (s; E + \mathrm{i} \eta)$ exists. This is true for almost all $E\in \R$, by the sixth part of \Cref{l:aizenman}. 

Recall that $\Upsilon_L$ was defined in \Cref{amoment2}. 
\Cref{l:PhiEA} implies that 
\begin{equation}\label{e:feblim}
\lim_{j \rightarrow \infty} L^{-1} \log \Phi_L( s ; E + \iu \eta_j) = \log K  + L^{-1} \log \Upsilon_L(s ; E ),
\end{equation}
where we used that the $R_{0v}$ summands in $\Phi_L$ are identically distributed. From \Cref{c:transfer},  
\begin{equation}\label{e:feblim0}
\Upsilon_{L}(s ; E ) = S\big( F^{L+1}(\tilde w)\big).
\end{equation}
From \eqref{e:initialdatabound}, \Cref{l:eigenvectorsandwich}, and the positivity of $F$, we get 
\begin{equation}
c S\big( F^{L-1}(v_E )\big)   \le c S\big( F^{L-1}(\tilde u )\big)  \le S\big( F^{L+1}(\tilde w)\big) \le  C S\big( F^{L+1}(\tilde u)\big) \le 
C S\big( F^{L+1}(v_E)\big).
\end{equation}
Since $v_E$ is a positive eigenvalue of $F$, we obtain 
\begin{equation}
c  \lambda_{E}^{L-1} S(v_E)
 \le 
S\big( F^{L+1}(\tilde w)\big) \le  C \lambda_{E}^{L+1} S(v_E).
\end{equation}
Inserting this in \eqref{e:feblim0} and putting the result in \eqref{e:feblim} gives 
\begin{equation}
\lim_{L \rightarrow \infty}
\lim_{j \rightarrow \infty} L^{-1} \log \Phi_L( s ; E + \iu \eta_j)  =\log K +  \log \lambda_{E},
\end{equation}
and hence 
\be
\lim_{j \rightarrow \infty} \phi( s ; E + \iu \eta_j)  =\log K +  \log \lambda_{E},
\ee
where the interchange of limits is justified by the fifth part of \Cref{l:aizenman}. 
The conclusion then follows from \eqref{e:limitr0j2}. 
\end{proof}

\subsection{Monotonicity} \label{s:monotonicitysub}
The following lemma establishes that $ \ln \lambda_{s,E}$ is monotonic in $E$ on intervals where $\rho'$ is bounded away from zero. 
\begin{lemma}\label{l:penultimate}
Set $I_\L  =  [E_0 - \L^{-1}, E_0 + \L^{-1}]$. 
Suppose that either
\be\label{e:ge111b}
\rho'(E)  \ge \frac{1}{\L}
\text{ for all $E \in I_\L$ }
\ee 
or
\be\label{e:le111b}
\rho'(E)  \le  - \frac{1}{\L} \text{ for all $E \in I_\L$ }.
\ee 
 There exists a constant $ K_0(\L)>0$ such that the following holds. For all $K \ge K_0$,  there exists $s_0(K, \L)\in(0,1)$ such that for all $s \in (s_0,1)$, the function $f(e)= \ln \lambda_{K,t,s,e}$ satisfies the following property.  If \eqref{e:ge111b} holds, then
\be\label{1015}
f(e_1)  \le f(e_2)
\ee
for all $e_1, e_2 \in I_\star$ such that $e_1 < e_2$, 
and if \eqref{e:le111b} holds,
\be\label{1016}
f(e_1)  \ge f(e_2)
\ee 
for all 
$e_1, e_2 \in  I_\star $ such that $e_1 < e_2$.
\end{lemma}
\begin{proof}
We prove only \eqref{1016}, since the proof of \eqref{1015} is nearly identical. 
We set $\alpha = \alpha_0$ in the definition of $\Delta$ (see \eqref{e:tKdef}), where $\alpha_0$ is given by the statement of \Cref{l:pizza}.
By the positivity of $F$ and $\tilde u$, 
\be
\| F^n_{E_1} \tilde u \|_1 = \int_{-\infty}^\infty (F^n_{E_1} \tilde u)(x) \, dx,
\ee 
and similarly with $E_1$ replaced by $E_2$. Then
\be\label{e:Fdiffsum0}
\| F^n_{E_2} \tilde u \|_1 - \| F^n_{E_1} \tilde u \|_1 = 
\int_{-\infty}^\infty   \big( (F^n_{E_2}   - F^n_{E_1}) \tilde u\big)(x) \, dx .
\ee 
Further, recalling that $F^\circ$ was defined in \eqref{e:Fcirc}, we have 
\be\label{e:Fdiffsum}
F^n_{E_2} -  F^n_{E_1}    = \sum_{j=0}^{n-1} F_{E_2}^{n-j -1} (F_{E_2} - F_{E_1}) F_{E_1}^{j} = \sum_{j=0}^{n-1} F_{E_2}^{n-j -1} (F^\circ) F_{E_1}^{j}.
\ee 
From \Cref{l:pizza}, we obtain for all integers $j$ such that $j \ge 0$ that 
\begin{align}
\begin{split}
\label{e:circintermediate}
F^\circ (F_{E_1}^{j} \tilde u) (x) \le&
 C| E_1 - E_2|   \lambda_{E_1}^{j-1}  \ln \left( \frac{1}{\Delta}\right)^2 \one \{ |x| \le \Delta \}
 \\
&  - c| E_1 - E_2|  \lambda_{E_1}^{j}\log \left(\frac{1}{\Delta} \right)  \frac{t }{|x|^{2-s}} 
  \one\{ |x| \ge \Delta \} \\ 
=&
 C| E_1 - E_2|   \lambda_{E_1}^{j-1} \Delta \ln \left( \frac{1}{\Delta}\right)^2 u_1(x)
 \\
&  - c| E_1 - E_2|  \lambda_{E_1}^{j}\log \left(\frac{1}{\Delta} \right) t \big( u_2(x) + u_3(x) \big),
\end{split}
\end{align}
where the equality follows from the definitions of $u_1$, $u_2$, and $u_3$ in \eqref{e:theus}.  Then, since $F_{E_2}$ is an integral operator with a positive kernel, 
\begin{align}
\begin{split}
\label{e:circintermediate2}
 &(F_{E_2}^{n-j -1} F^\circ F_{E_1}^{j} \tilde u) (x) \\
 &\le
 C| E_1 - E_2|   \lambda_{E_1}^{j-1} \Delta \ln \left( \frac{1}{\Delta}\right)^2 ( F_{E_2}^{n-j -1}u_1)(x)
 \\
&\quad   - c| E_1 - E_2|  \lambda_{E_1}^{j}\log \left(\frac{1}{\Delta} \right) t \big(  (F_{E_2}^{n-j -1}u_2)(x) + ( F_{E_2}^{n-j -1}u_3)(x) \big).
\end{split}
\end{align}

We begin by analyzing the $u_1$ term in \eqref{e:circintermediate2}. We suppose that $j \neq n-1$; the case $j = n-1$ will be treated later. 
We have
\begin{align*}
(F_{E_2} u_1)(x) &\le C \Delta t^{-1} u_1(x) + Ct \big( u_2(x) + u_3(x) \big)\\
&\le C t u_1(x) + Ct v_{E_2}(x),
\end{align*}
where we used \Cref{l:u1upper} and \eqref{e:tbounds} in the first inequality, then \Cref{l:sumac} and the definition of $\Delta$ in \eqref{e:tKdef} in the second inequality. 
Then, because $v_{E_2}$ is an eigenvector of $F_{E_2}$, 
\begin{align*}
(F^2_{E_2} u_1)(x) &\le C t  (F_{E_2} u_1)(x) + C  t ( F_{E_2} v_{E_2} ) (x)  \\
&\le C ^2  t^2  u_1(x)  + C^2t^{2}  v_{E_2} (x)  + C t  \lambda_{E_2} v_{E_2}(x).
\end{align*}
For all $k \ge 2$, we obtain analogously that 
\be\label{e:artichoke}
(F^k_{E_2} u_1)(x) 
\le C^k t^k  u_1(x) +  \left( \sum_{m=0}^{k-1} C^{ k  -m} t^{ k -m } \lambda_{E_2}^m \right) v_{E_2}(x).
\ee 
Since \eqref{e:tbounds} and \Cref{c:roughlambda} imply that $t \le C \lambda_{E_2} (\log K)^{-1}$ for $s \ge \s$, there exists $K_0 \in \N$ and $s_0(K)$ such that for $K \ge K_0$ and $s \in (s_0,1)$, we have
\be
\sum_{m=0}^{k-1} C^{ k  -m} t^{ k -m } \lambda_{E_2}^m
\le \lambda_{E_2}^{k}
\sum_{m=0}^{k-1} C^{ k  -m} (\log K)^{-(k - m) }  \le (\ln K)^{-1} \lambda_{E_2}^{k}.
\ee 
We conclude from \eqref{e:artichoke} that for any $k \ge 1$, 
\be\label{e:baruiterationbound}
(F^k_{E_2} u_1)(x) \le C^k t^k  u_1(x)  + C  (\ln K)^{-1} \lambda_{E_2}^{k} v_{E_2}(x).
\ee 
This implies, using the positivity of $F_{E_2}$, $u_1$, and $v_{E_2}$, that 
\be\label{e:Fu1}
\| F^k_{E_2} u_1 \|_1 \le 2 C^k  t^{k } + C (\ln K)^{-1}   \lambda_{E_2}^{k}  \| v_{E_2} \|_1.
\ee 
We then find 
\begin{align}\label{e:badpartnearzero}
\begin{split}
&| E_1 - E_2|   \lambda_{E_1}^{j-1} \Delta \ln \left( \frac{1}{\Delta}\right)^2  \| F_{E_2}^k u_1 \|_1 \\
&\le| E_1 - E_2|   \ln \left( \frac{1}{\Delta}\right)^2  \left(    C^{k+1}  t^{k + 2 } \lambda_{E_1}^{j-1}  + C t^{2}    \lambda_{E_1}^{j-1} \lambda_{E_2}^{k}  (\ln K)^{-1} \|v_{E_2} \|_1\right),  
\end{split}
\end{align}
which bounds the integral over $\R$ of  first term in \eqref{e:circintermediate2} for all $j < n-1$.

We now consider the second term in \eqref{e:circintermediate2}. Note that by the definition of $\tilde u$ in \eqref{e:tildeu}, and the definition of $\tilde u_1$ in \eqref{e:tildeu1}, we have 
\begin{align}\label{1028}
- (F^k_{E_2} u_2)(x) - (F_{E_2}^k u_3)(x) = - (F^k_{E_2} \tilde u)(x) +(F^k_{E_2} \tilde u_1)(x) .
\end{align}
Using \eqref{e:Fu1},  the inequality $\tilde u_1 \le C (\ln K)^{-1} u_1$ (which is immediate from the definitions of $\tilde u_1$ and $u_1$ in \eqref{e:tildeu} and \eqref{e:theus}, respectively), and the definition of $\Delta$ in \eqref{e:tKdef}, we have 
\begin{align}
\begin{split}
\label{e:badpartnearzero2}
&| E_1 - E_2|  \lambda_{E_1}^{j}\log \left(\frac{1}{\Delta} \right) t   \|  F^k_{E_2} \tilde u_1 \|_1 \\
 &\le | E_1 - E_2|   \  \left(2 C^k  t^{k +1 }  \lambda_{E_1}^{j} + C   t \lambda_{E_1}^{j}  \lambda_{E_2}^{k}  (\ln K)^{-1} \| v_{E_2} \|_1\right).
 \end{split}
\end{align}

Next, applying $F^{k-1}$ to both sides of \eqref{e:sandwichupperbound}  gives
\be\label{e:Fiteratelower}
(F^k_{E_2} \tilde u)(x) \ge  c (F^k_{E_2} v_{E_2} )(x).
\ee 
Using that $v_{E_2}$ is an eigenvector, this implies 
\be\label{e:goodpart}
- (F^k_{E_2} \tilde u)(x)  \le - c \lambda_{E_2}^k v_{E_2}(x).
\ee
We conclude that for all integers $j$ such that $0 \le  j < n-1$,  
\begin{align}
\begin{split}\label{e:jbulk}
&\int_{-\infty}^\infty \big(F_{E_2}^{n-j -1}  F^\circ  F_{E_1}^{j}\tilde u  \big)(x)\, dx \\
&\le 
 - c  \|v_{E_2} \|_1  | E_1 - E_2  | t  \lambda_{E_1}^{j} \lambda_{E_2}^{n-j -1}  \log \left(\frac{1}{\Delta} \right)  \\
 & \quad +  C | E_1 - E_2|     \left(2 C^{n-j-1}  t^{n-j }  \lambda_{E_1}^{j} + C   t \lambda_{E_1}^{j}  \lambda_{E_2}^{n-j-1} (\ln K)^{-1}  \| v_{E_2} \|_1\right) \\
 &\quad  +C  | E_1 - E_2|   \ln \left( \frac{1}{\Delta}\right)^2  \left(    C^{n-j}  t^{n-j+1 } \lambda_{E_1}^{j-1}  + C t^{2}    \lambda_{E_1}^{j-1} \lambda_{E_2}^{n-j-1}(\ln K)^{-1}   \|v_{E_2} \|_1\right)\\
& \le - c \| v_{E_2} \|_1 |E_1 - E_2| \lambda_{E_1}^j \lambda_{E_2}^{n-j} 
+C (\ln K)^{-1} |E_1 - E_2| \lambda_{E_1}^j \lambda_{E_2}^{n-j}  \\
&\le  - c  \| v_{E_2} \|_1 |E_1 - E_2| \lambda_{E_1}^j \lambda_{E_2}^{n-j}.
\end{split}
\end{align}
The first inequality follows from combining \eqref{e:circintermediate2}, \eqref{e:badpartnearzero}, \eqref{e:badpartnearzero2}, \eqref{1028}, and \eqref{e:goodpart}. 
In the second inequality, we used \Cref{c:roughlambda} to control the factors of $\lambda_{E_1}$ and $\lambda_{E_2}$. We also assumed that $K \ge K_0$ for some $K_0$ depending on the size of the constants $C$ and $c$, and recalled the definitions of $t$ and $\Delta$ in \eqref{e:tKdef}.  In the last line, we chose $s$ large enough so that 
\be
c \| v_{E_2} \|_1 = \frac{c}{1-s} \ge 2 C ,
\ee
and adjusted the value of $c$ in the following line (supposing without loss of generality that $c<1$).

Next, we consider the case $j=n-1$ in \eqref{e:Fdiffsum}.
We have
\begin{align}
\begin{split}
\int_{-\infty}^\infty F^\circ (F_{E_1}^{n-1} \tilde u)(x)  \, dx 
& \le C |E_1 - E_2|   \lambda_{E_1}^n  - c |E_1 - E_2| \lambda_{E_1}^n (1 -s )^{-1}  \\ 
& \le  C |E_1 - E_2|   \lambda_{E_1}^n  - c |E_1 - E_2| \lambda_{E_1}^n   \| v_{E_2} \|_1    \\ 
& \le  - c |E_1 - E_2| \lambda_{E_1}^n   \| v_{E_2} \|_1  
\end{split}  . \label{e:jequalsnminusone}
\end{align}
The first inequality is from the definitions of $t$ and $\Delta$ in \eqref{e:tKdef}, \eqref{e:tbounds},  \Cref{c:roughlambda} and \eqref{e:circintermediate}.  
In the second inequality, we used the normalization for $v_{E_2}$ defined in \eqref{e:vnormalization}. In the third inequality, we  increased the value of $s$ if necessary.
Combining \eqref{e:Fdiffsum0}, \eqref{e:Fdiffsum}, \eqref{e:jbulk} and \eqref{e:jequalsnminusone}, 
and using \Cref{c:roughlambda}, we obtain
\begin{align}\label{e:diff-firstline}
\| F^n_{E_2} \tilde u \|_1 - \| F^n_{E_1} \tilde u \|_1 
& \le  - c  \| v_{E_2} \|_1  \sum_{j=0}^n \lambda_{E_1}^j \lambda_{E_2}^{n-j},
\end{align}
where we took $n$ sufficiently large (in a way that depends only on the constants in \eqref{e:diff-firstline}). Hence,
\begin{align}
\begin{split}
\label{e:logtaylor}
\ln \lambda_{s, E_2}  - \ln \lambda_{s, E_1} &= 
\lim_{n\rightarrow \infty} \frac{ \ln \| F^{n}_{E_2} \tilde u \|_1 - \ln \| F_{E_1}^{n} \tilde u \|_1}{n}  \le 0,
\end{split}
\end{align}
where we used \eqref{e:diff-firstline} in the last inequality. 
We conclude that $\lambda_{s,E_1} \ge \lambda_{s,E_2}$. 
\end{proof}

\begin{lemma}
	
	\label{lambdae120} 
	
Set $I_\L  =  [E_0 - \L^{-1}, E_0 + \L^{-1}]$. 
Suppose that either
\be\label{e:ge1111b}
\rho'(E)  \ge \frac{1}{\L}
\text{ for all $E \in I_\L$ }
\ee 
or
\be\label{e:le1111b}
\rho'(E)  \le  - \frac{1}{\L} \text{ for all $E \in I_\L$ }
\ee 
 There exist constants $\delta, K_0>0$ such that the following holds for all $K \ge K_0$ and $E_1, E_2 \in I_\star$ such that $E_1 < E_2$. 
 If \eqref{e:ge1111b} holds, then 
\be
\liminf_{s\rightarrow 1 } \big( \lambda_{s,E_2 } -  \lambda_{s,E_1} \big) > 0.
\ee
If \eqref{e:le1111b} holds, then 
\be
\liminf_{s\rightarrow1}   \big( \lambda_{s,E_1 } -  \lambda_{s,E_2} \big) > 0.
\ee
\end{lemma}
\begin{proof}
We suppose that \eqref{e:le1111b} holds throughout this proof; the argument for \eqref{e:ge1111b} is analogous. 
We begin by noting that for all integers $m\ge 0$,
\begin{align}
\begin{split}\label{e:icecream}
 F_{E_2}^{m} (F^\circ v_{E_1}) (x) 
& \le  C|E_1 - E_2|t   \ln \left( \frac{1}{\Delta} \right) (F_{E_2}^{m}  u_1)(x)\\
&\quad 
- c |E_1 - E_2| t 
\ln \left( \frac{1}{\Delta} \right)
\big( (F_{E_2}^{m}  u_2)(x) + (F_{E_2}^{m}  u_3)(x) \big)\\
& \le  C|E_1 - E_2|t   \ln \left( \frac{1}{\Delta} \right) (F_{E_2}^{m}  u_1)(x)\\
&\quad 
- c |E_1 - E_2| t 
\ln \left( \frac{1}{\Delta} \right)
\big( (F_{E_2}^{m}  \tilde u) (x) \big)\\
&\le   C|E_1 - E_2|t   \ln \left( \frac{1}{\Delta} \right) \left(C^{m} t^{m}  u_1(x)  + C (\ln K)^{-1} \lambda_{E_2}^{m} v_{E_2}(x) \right)\\
&\quad 
 -c  |E_1 - E_2| t \lambda_{E_2}^{m}
\ln \left( \frac{1}{\Delta} \right) v_{E_2}(x).
\end{split}
\end{align}
The first inequality comes \eqref{e:coffee10} and \eqref{e:coffee2}, the positivity of $F$, and the definitions of $u_1$, $u_2$, and $u_3$ in \eqref{e:theus}. The second inequality follows from the definition of  $\tilde u$ in \eqref{e:tildeu}. The third inequality uses \eqref{e:baruiterationbound} and the lower bound in \eqref{e:sandwich21}. 
Then  for all $m \ge 0$, 
\begin{align}
\begin{split}\label{e:march48}
F_{E_2}^m(F^\circ v_{E_1} ) (x) 
&\le C|E_1 - E_2|t   \ln \left( \frac{1}{\Delta} \right) \left(C^{m} t^{m}  u_1(x)  -  c\lambda_{E_2}^{m} v_{E_2}(x) \right)\\
&\le C|E_1 - E_2|t  \lambda_{E_2}^{m}  \ln \left( \frac{1}{\Delta} \right) \left(C^{m} (\ln K)^{-m}  u_1(x)  - c v_{E_2}(x) \right),
\end{split}
\end{align}
where the first inequality follows from \eqref{e:march50}, and the second inequality follows from the  definition of $t$ and \Cref{c:roughlambda}.

Set $m_0 = 2$.
By \Cref{l:sumac2},  there exists $K_0$ depending only on the constants in the previous inequality such that, for all $K \ge K_0$ and $m \ge m_0$, 
\begin{align}\label{e:instep}
F_{E_2}^m(F^\circ v_{E_1} ) (x) \le - c \lambda_{E_2}^{m+1}  |E_1 - E_2| v_{E_2}(x).
\end{align}

Then we have for all $ n \ge m_0$ that 
\begin{align}
\begin{split}
\label{e:march50}
(F^n_{E_2}v_{E_1}) (x)   &= (F^n_{E_1}v_{E_1})(x) +   \sum_{j=0}^{n-1} \big(F_{E_2}^{n-j -1} F^\circ F_{E_1}^{j} v_{E_1} \big)(x) \\
&=  \lambda_{E_1}^nv_{E_1}(x) +   \sum_{j=0}^{n-1}  \lambda_{E_1}^j \big(F_{E_2}^{n-j -1} F^\circ v_{E_1}\big)(x) \\ 
&\le  \lambda_{E_1}^nv_{E_1}(x) +  C |E_1 - E_2|  \sum_{j=0}^{m_0} C^j  \lambda_{E_1}^j  \lambda_{E_2}^{n-j}   u_1(x)   \\
& \quad -  c |E_1 - E_2| \sum_{ j = m_0+1}^n  \lambda_{E_1}^j \lambda_{E_2}^{n-j} v_{E_2}(x)\\
&= \lambda_{E_1}^nv_{E_1}(x) +  C |E_1 - E_2|  \lambda_{E_1}^n \sum_{j=0}^{m_0} C^j   ( \lambda_{E_2}/ \lambda_{E_1})^{n-j}   u_1(x)   \\
& \quad -  c |E_1 - E_2|   \lambda_{E_1}^n  \sum_{ j = m_0+1}^n ( \lambda_{E_2}/ \lambda_{E_1})^{n-j} v_{E_2}(x).
\end{split}
\end{align}
The first equality is  \eqref{e:Fdiffsum}, and the second equality follows since $v_{E_1}$ is an eigenvector.
The inequality follows from using \eqref{e:march48} in the terms with $j \le m_0$ (along with the definitions of $t$ and $\Delta$ in \eqref{e:tKdef}) and \eqref{e:instep} for the terms with $j > m_0$, along with \Cref{c:roughlambda}. 
The last equality rewrites the line before it. 

Recall that $\lambda_{s,E_1} \ge \lambda_{s,E_2}$. Arguing by contradiction, we fix $E_1$ and $E_2$, and we suppose that 
\be\label{e:usefullimit} 
\liminf_{s\rightarrow 1} \big( \lambda_{s,E_1}  - \lambda_{s,E_2} \big) = 0
\ee 
 In this case, by taking $s$ sufficiently close to $1$ and $n$ sufficiently large, each in a way that depends only on $K$, we find using \eqref{e:sandwichlowerbound} and \eqref{e:march50} (to get $u_1 (x) \le  C (\ln K) v_{E_2}(x)$) that
\begin{align}
(F^n_{E_2}v_{E_1}) (x)  \le  \lambda_{E_1}^nv_{E_1}(x)  -  c   \lambda_{E_1}^n |E_1 - E_2|  v_{E_2}(x).
\end{align}
Using the previous inequality twice gives
\begin{align}
(F^{2n}_{E_2}v_{E_1}) (x)  \le  \lambda_{E_1}^{2n}v_{E_1}(x)  - c|E_1 - E_2|  \lambda_{E_1}^{2n}  v_{E_2}(x)  -  c |E_1 - E_2|   \lambda_{E_1}^n  \lambda_{E_2}^n v_{E_2}(x).
\end{align}
In general, multiple iterations yield that for all $k$, we have
\begin{align}
\begin{split}
(F^{k n}_{E_2}v_{E_1}) (x) 
&\le \lambda_{E_1}^{kn} v_{E_1}(x) - c |E_1 - E_2|  \sum_{j=0}^{k-1} \lambda_{E_2}^{jn} \lambda_{E_1}^{(k-j) n}v_{E_2}(x)\\
& =\lambda_{E_1}^{kn} \left( v_{E_1}(x)  - c |E_1 - E_2|  \sum_{j=0}^{k-1} (\lambda_{E_2}/\lambda_{E_1})^{jn} v_{E_2}(x) \right).
\end{split}
\end{align}
Using \eqref{e:salt} and \eqref{e:sandwichlowerbound}, and arguing as after \eqref{e:usefullimit}, we have for $k$ sufficiently large and $s$ sufficiently close to $1$, each in a way that depends only on $|E_1 - E_2|$ and $K$, that 
\begin{align}
 (F^{k n}_{E_2}v_{E_1}) (x)  \le - c \lambda_{E_1}^{kn} v_{E_2}(x).
\end{align}
This is a contradiction, since $F_{E_2}$ sends non-negative functions to non-negative functions, and $v_{E_1}(x) \ge 0$ for all $x \in \R$.
\end{proof}
\begin{proof}[Proof of \Cref{t:mainlambda}]
This follows  immediately 
from \Cref{c:KR} and \Cref{l:penultimate}.
\end{proof}

\appendix 

\section{Continuity of the Free Energy}\label{s:lyapunovcontinuity}

This section is devoted to the proof of \Cref{l:phicontinuity}. 
The arguments in this section follow closely those  in Section 12 of \cite{aggarwal2022mobility}, and we give the details here for completeness. Also, the content at the beginning of this appendix bears some resemblance to the earlier work \cite{germinet2004characterization}. 
We begin in \Cref{s:continuitypreliminary} with a preliminary bound on fractional moments of $\Im R_{00}$, \Cref{l:qsexpectation}, and use it to deduce \Cref{p:imvanish}.
In \Cref{s:continuitylemmaA2}, we prove a lemma on the continuity of the fractional moments of $R_{0v}$, \Cref{l:cutoffclose}, which enables all further continuity arguments in this section. \Cref{s:continuitylemmaA2} also contains a sequence of other preliminary lemmas, addressing the continuity $\phi(s;z)$ in $z$, followed by
the proof of \Cref{l:phicontinuity}.

All constants in this section may depend on the choice of $K$ and $g$ in \Cref{t:main}, but we omit this dependence in the notation. 
\subsection{Preliminary Estimates}
\label{s:continuitypreliminary}

We begin with a fractional moment bound for the resolvent. For later use in a more general context, we state it for the resolvent of $\hamiltonian$ for a general disorder parameter $t \in (0,1)$. 
\begin{lemma}\label{l:qsexpectation}
Fix  $s, \delta ,\eps \in (0,1)$. There exists a constant $C=C(\delta,\eps, s) > 1$  such that the following holds for all $t \in (\eps ,1)$ and any $z = E + \iu \eta \in \bbH$ such that  $\eta \in (0,1)$. If $\varphi (s; z) < -\delta$, then
\bex
\E \left[ \big( \Im R_{00}( z ) \big)^{s/2} \right] \le C \eta^{s/2}.
\eex
\end{lemma}
\begin{proof}
		
For every vertex $v \in \mathbb{V}$, set $R_{0v} = R_{0v} (z)$. By the Ward identity \eqref{sumrvweta}, 
		
		\be\label{e:appfirstcentered}
		\Im R_{00}( z ) = \eta \sum_{v \in \mathbb{V}} |R_{0v}|^2.
		\ee
		Using  $s/2 < 1$, we find
		\begin{flalign*}
	\Bigg( \displaystyle\sum_{v \in \mathbb{V}} |R_{0v}|^2 \Bigg)^{s/2} = \Bigg( \displaystyle\sum_{L = 0}^{\infty} \displaystyle\sum_{v \in \mathbb{V} (L)} |R_{0v}|^2 \Bigg)^{s/2} \le \displaystyle\sum_{L = 0}^{\infty} \displaystyle\sum_{v \in \mathbb{V} (L)} |R_{0v}|^{s}.
		\end{flalign*}		
	Taking expectations, recalling \Cref{moment1} and \eqref{e:limitr0j2} and using the assumption that $\varphi (s; z) < -\delta$, we deduce the existence of a constant $C > 1$ such that
		\begin{flalign*}
	 \mathbb{E}  \Bigg[ \displaystyle\sum_{L = 0}^{\infty} \displaystyle\sum_{v \in \mathbb{V} (L)} |R_{0j}|^{s} \Bigg] =   \displaystyle\sum_{L = 0}^{\infty} \Phi_L (s; z) \le C  \displaystyle\sum_{L = 0}^{\infty} e^{-\delta L} \le  \delta^{-1} C.
		\end{flalign*}
This completes the proof after recalling \eqref{e:appfirstcentered}. 
\end{proof}
\begin{proof}[Proof of \Cref{p:imvanish}]
	By \Cref{l:qsexpectation}, we have $\mathbb{E} \big[ (\Imaginary R_{00} (E + \iu\eta_j))^s \big] \le C \eta_j^s$ for all $j\in\Zplus$. Then $(\Imaginary R_{00} (E + \iu \eta_j))^s$ converges to $0$ in expectation as $j$ tends to infinity, which implies it converges to zero in probability.
	\end{proof}
	
\subsection{Continuity Estimates}\label{s:continuitylemmaA2}
We begin with a key technical lemma that underlies the proofs of our continuity estimates. 
\begin{lemma}\label{l:cutoffclose}
Fix $s, \eps \in (0,1)$ and $\delta >0$. There exist constants $C(\delta, \eps, s) > 1$ and $c(\delta, \eps, s) > 0$ 
such that the following holds for all $t\in (\eps,1)$. For  $i=1,2$, fix $\eta_i \in (0,1)$ and $z_i = E_i + \iu \eta_i \in \bbH$. Suppose that $\phi(s;z_2) < - \delta$. Then
\bex
\E
\left[
\sum_{v \in \bbV(L)} \left|   
 R_{0v}(z_1) -  R_{0v}(z_2) 
\right|^s 
\right] 
\le  C^{L+1}|z_1 - z_2|^s (\eta_1 \eta_2)^{-s/2}
\cdot \eta_2^{c } .
\eex
\end{lemma}
\begin{proof}
For $k \in \unn{0}{L-2}$, we define
\begin{align}\label{PsiDef}
\begin{split}
&\Psi_{L, k} (s;  z_1, z_2) \\
&= 
\E\left[ 
\sum_{w \in \bbV(k+1)}
\sum_{v \in \bbD_{L-k-1}(w)}
\left|
 R_{0 w_-}(z_2)  t  R^{(w_-)}_{w v }(z_1)
-  R_{0 w}(z_2)  t R^{(w)}_{w_+ v } (z_1)
\right|^s
\right].
\end{split}
\end{align}
We also set 
\bex
\Psi_{L, -1} (s;  z_1, z_2)=
\E\left[ 
\sum_{w \in \bbV(1)}
\sum_{v \in \bbD_{L-1} (w)}
\left|
 R_{0 v}(z_1)
- R_{0  0 }(z_2)  t  R^{(0)}_{w v}(z_1 )
\right|^s
\right]
\eex
and
\bex
\Psi_{L, L-1} (s;  z_1, z_2)=
\E\left[ 
\sum_{w \in \bbV(L-1)}
\sum_{v \in \bbD(w)}
\left|
 R_{0  w }(z_2)  t   R^{(w)}_{vv}(z_1)
- 
 R_{0 v}(z_2)
\right|^s
\right].
\eex
By 
the definitions of the $\Psi_{L, j}$, and the elementary inequality that $(a+b)^s \le a^s + b^s$ for all $a,b \ge 0$, 
\be\label{thephisum}
\E
\left[
\sum_{v \in \bbV(L)} \left|   
 R_{0v}(z_1) -  R_{0v}(z_2) 
\right|^s 
\right] \le 
\sum_{j = -1}^{L-1} \Psi_{L, j} (s; z_1, z_2).
\ee
To complete the proof, it suffices to bound the sum on the right side of \eqref{thephisum}.

Considering the terms in the sum defining \eqref{PsiDef},  \Cref{rproduct} gives
\begin{align}\label{e:febsplit}
\begin{split}
&  R_{0 w_-}(z_2) t R^{(w_-)}_{w v }(z_1)
-  R_{0 w}(z_2)  t R^{(w)}_{w_+ v } (z_1)\\
& =  R_{0 w_-}(z_2)  t  R^{(w_-)}_{w w}(z_1) 
 t
 R^{(w)}_{w_+ v} (z_1)\\
&\quad -  R_{0 w_-}(z_2)  t   R^{(w_-)}_{w w}(z_2) 
t 
 R^{(w)}_{w_+ v} (z_1)\\
&=  R_{0 w_-}(z_2)  t  \big(R^{(w_-)}_{w w}(z_1) -   R^{(w_-)}_{w w}(z_2) \big)
t
 R^{(w)}_{w_+ v} (z_1).
 \end{split}
\end{align}
Then using \eqref{e:febsplit} in the definition of $\Psi_{L, k}$ in \eqref{PsiDef} gives
\begin{align*}
&\Psi_{L, k} (s;  z_1, z_2)\notag\\
&= 
\E\left[
\sum_{w \in \bbV(k+1)}
\sum_{v \in \bbD_{L-k-1}(w)}
|t|^{2s}
\left|
R_{0 w_-}(z_2)  \big(R^{(w_-)}_{w w}(z_1) -   R^{(w_-)}_{w w}(z_2) \big)
 R^{(w)}_{w_+ v} (z_1)
\right|^s 
\right].
\end{align*}
We now use \Cref{l:uniformlyintegrable} to bound the expectation of each term in the sum. We apply this lemma after using the law of total expectation and conditioning on all disorder variables $V_{x}$ corresponding to vertices not in the path from $0$ to $w_-$. This gives 
\begin{align}
\Psi_{L, k}& (s;  z_1, z_2)\notag \\
&\le 
 C^{k+1} \E\left[
\sum_{w \in \bbV(k+1)}
\sum_{v \in \bbD_{L-k-1}(w)}
|t|^{2s}  \left|\big(R^{(w_-)}_{ww}(z_1) -   R^{(w_-)}_{ww}(z_2) \big)
 R^{(w)}_{w_+ v} (z_1)
\right|^s 
\right]\notag\\
&\le 
C^{k+1} \E\left[
\sum_{w \in \bbV(1)}
\sum_{v \in \bbD_{L-k-2}(w)}
|t|^{2s}   \left|\big(R_{00}(z_1) -   R_{00}(z_2) \big)
 R^{(0)}_{w v} (z_1)
\right|^s 
\right]\label{phiprev}
\end{align}
for some $C > 1$, where we used that the joint law of $R^{(w_-)}_{ww}(z_1)$, $R^{(w_-)}_{ww}(z_2)$, and $ R^{(w)}_{w_+ v} (z_1)$ in the second line is equivalent to the joint law of $R_{00}(z_1)$, $ R_{00}(z_2)$, and $ R^{(0)}_{w v} (z_1)$ in the third line.
Using the resolvent identity $\bm{A}^{-1} - \bm{B}^{-1} = \bm{A}^{-1} (\bm{B} - \bm{A}) \bm{B}^{-1}$, we write 
\bex
R_{00}(z_1) -  R_{00}(z_2)
= (z_2 - z_1) \sum_{v \in \bbV} R_{0v}(z_1) R_{v0}(z_2).
\eex
 H\"older's inequality and the Ward identity \eqref{sumrvweta} together give the bound
\begin{align}
\left| R_{00}(z_1) -  R_{00}(z_2) \right|
&\le |z_1 - z_2| 
\left(\sum_{v \in \bbV} |R_{0v}(z_1)|^2 \right)^{1/2}
\left(\sum_{v \in \bbV} |R_{0v}(z_2)|^2 \right)^{1/2}\notag
\\
& = |z_1 - z_2| (\eta_1\eta_2) ^{-1/2} 
\big( \Im R_{00}(z_1) \big)^{1/2}
\big( \Im R_{00}(z_2) \big)^{1/2}.\label{holder2}
\end{align}
Inserting \eqref{holder2} into \eqref{phiprev} gives
\begin{align}\notag
&\Psi_{L, k} (s; z_1, z_2)\le  C^{k+1}|z_1 - z_2|^s (\eta_1 \eta_2)^{-s/2}\\
&\times  \E\left[
\sum_{w \in \bbV(1)}
\sum_{v \in \bbD_{L-k-2}(w)}
\big( \Im R_{00}(z_1) \big)^{s/2}
\big( \Im R_{00}(z_2) \big)^{s/2}
\left|
R^{(0)}_{w v}(z_1)
\right|^s  \label{intPhi}
\right].
\end{align}

Let $\kappa = s/4$. By Markov's inequality, the assumption $\phi(s;z_2) < - \delta$, and \Cref{l:qsexpectation}, 
\be\label{simplemarkov}
\P\left( \big( \Im R_{00}(z_2) \big)^{s/2} > \eta_2^\kappa  \right) \le C \eta_2^{s/2 - \kappa}.
\ee
Define the event 
\bex
\mathscr A = \left\{ \big( \Im R_{00}(z_2) \big)^{s/2} < \eta_2^\kappa \right\}.
\eex
Then 
\begin{align}\label{ongoodevent}
\begin{split}
&\E\left[\one_{\mathscr A}
\sum_{w \in \bbV(1)}
\sum_{v \in \bbD_{L-k-2}(w)}
\big( \Im R_{00}(z_1) \big)^{s/2}
\big( \Im R_{00}(z_2) \big)^{s/2}
\left|
R^{(0)}_{w v}(z_1)
\right|^s
\right]
\\
&\le \eta^{\kappa}_2\cdot \E\left[
\sum_{w \in \bbV(1)}
\sum_{v \in \bbD_{L-k-2}(w)}
\big| R_{00}(z_1) \big|^{s/2}
\left|
R^{(0)}_{w v}(z_1)
\right|^s
\right]
\le 
\eta^{\kappa}_2
C^{L-k},\end{split}
\end{align}
where the last inequality follows from two applications of \Cref{l:uniformlyintegrable}.
 We next note that the elementary inequality $ab \le a^2 + b^2$ gives 
\bex
\big( \Im R_{00}(z_1) \big)^{s/2} \big( \Im R_{00}(z_2) \big)^{s/2}  \le\big( \Im R_{00}(z_1) \big)^{s}
 +
\big( \Im R_{00}(z_2) \big)^{s} .
\eex
The previous line implies
\begin{align}
\begin{split}
& \E\left[
\one_{\mathscr A^c}
\sum_{w \in \bbV(1)}
\sum_{v \in \bbD_{L-k-2}(w)}
\big( \Im R_{00}(z_1) \big)^{s/2}
\big( \Im R_{00}(z_2) \big)^{s/2}
\left|
t
R^{(0)}_{w v}(z_1)
\right|^s 
\right]\\
&\le  \E\left[
\one_{\mathscr A^c} \sum_{w \in \bbV(1)}
\sum_{v \in \bbD_{L-k-2}(w)}
\big( \Im R_{00}(z_1) \big)^{s}
\left|
t
R^{(0)}_{w v}(z_1)
\right|^s 
\right] \\
&+
 \E\left[
\one_{\mathscr A^c}\sum_{w \in \bbV(1)}
\sum_{v \in \bbD_{L-k-2}(w)}
\big( \Im R_{00}(z_2) \big)^{s}
\left|
t
R^{(0)}_{w v}(z_1)
\right|^s 
\right]\\
&\le \P(\mathscr A^c)^{c_1}  \cdot  C^{L-k}
= \eta_2^{c_1s } \cdot  C^{L-k}
\end{split}
\label{onbadevent}
\end{align}
for some $c_1(s) > 0$, where we used \eqref{simplemarkov} and 
the following bound in the last line. By H\"older's inequality and  \Cref{l:uniformlyintegrable}, we have, after fixing $\kappa >0$ such $(1+\kappa) s < 1$, that for $i=1,2$,
\begin{align*}
&\E\left[
\one_{\mathscr A^c}\sum_{w\in \mathbb{V}(1)} \sum_{v \in \bbD_{L-1}(w) }
\big( \Im R_{00}(z_i) \big)^{s}
\left|
t
R^{(0)}_{w v}(z_1)
\right|^s
\right]\\
&\le 
\P(\mathscr A^c)^{\kappa /( 1 + \kappa)} 
\E\left[\sum_{w\in \mathbb{V}(1)} \sum_{v \in \bbD_{L-1}(w) }
\left(\big( \Im R_{00}(z_i) \big)^{s} 
\left|
t
R^{(0)}_{w v}(z_1)
\right|^s \right)^{1+\kappa} 
\right]^{1/(1+\kappa)}\\
&\le 
\P(\mathscr A^c)^{\kappa / ( 1 + \kappa)} \cdot C^{1+L}.
\end{align*}
The claimed bound follows after 
inserting \eqref{ongoodevent} and \eqref{onbadevent} into \eqref{intPhi}, and using \eqref{thephisum}.
\end{proof}

The next lemma states a continuity estimate for $\phi(s; E + \iu \eta)$ as $\eta$ is a fixed and $E$ varies.  
\begin{lemma}\label{l:goingsideways}
Fix $s , \eps \in (0,1)$ and $ \delta, \eta \in (0,1)$. Then for every $\mathfrak b >0$, there exists $\mathfrak a (\delta, \eps, \eta,  s, \mathfrak b) > 0$ such that following holds for all $t\in (\eps,1)$. For every $E_1, E_2 \in \bbR$ such that  $|E_1 - E_2| < \mathfrak a$ and  ${\phi(s;E_2 + \iu \eta) < - \delta}$,
we have
\bex
\big|
\phi(s;E_1 + \iu \eta ) - \phi(s;E_2 + \iu \eta)
\big| < \mathfrak b.
\eex
\end{lemma}
\begin{proof}
Set $z_1 = E_1 + \iu \eta$ and $z_2 = E_2 + \iu \eta$. We may assume that $|E_1 - E_2| < 1/10$. By \Cref{l:aizenman}, there exists $C_1 > 1$ such that for any $L \in \Zplus$,
\be\label{combineme2}
\left|
\big( 
\phi_L(s;z_1) - \phi_L(s;z_2) 
 \big)
 -
 \big( 
\phi(s;z_1) - \phi(s;z_2) 
 \big)
\right|
\le \frac{C_1}{L}.
\ee
We set $L = \lceil 2 C_1 \mathfrak b^{-1} \rceil$.\footnote{From now on, we drop the ceiling functions when doing computations with $L$, for brevity.} Then it remains to bound the difference $\big| 
\phi_L(s;z_1) - \phi_L(s;z_2) 
 \big|$.

Recalling the definition \eqref{sumsz1}, we have
\begin{align}\label{mango1}
\phi_L(s;z_1)
-
\phi_L(s;z_2) 
=
\frac{1}{L}
\log \left(
\frac{\Phi_L(s; z_1)}{\Phi_L(s; z_2)}
\right).
\end{align}
We write
\begin{align}\notag
\left|\log \left(
\frac{\Phi_L(s; z_1)}{\Phi_L(s; z_2)}
\right)\right|
&=
\left|\log \left(
\frac{ \Phi_L(s; z_2) + \big(\Phi_L(s; z_1) - \Phi_L(s; z_2)  \big)}{\Phi_L(s; z_2)}
\right)\right|\\
&\le 
 \frac{ 2 \left|\E
\left[
\sum_{v \in \bbV(L)} \left|   
 R_{0v}(z_1) 
\right|^s
\right]
-
\E
\left[
\sum_{v \in \bbV(L)} \left|   
 R_{0v}(z_2) 
\right|^s
\right]\right|}{\Phi_L(s; z_2)},\label{mango2}
\end{align}
where the inequality is valid under the assumption that 
\be
\left|\E
\left[
\sum_{v \in \bbV(L)} \left|   
 R_{0v}(z_1) 
\right|^s
\right]
-
\E
\left[
\sum_{v \in \bbV(L)} \left|   
 R_{0v}(z_2) 
\right|^s
\right]\right|\le \frac{1}{2}\Phi_L(s; z_2).\label{mango3}
\ee
From the lower bound in \eqref{e:crudephi} and \eqref{e:limitr0j2}, and the fact that $\E [ \log | R_{00}(z) |^s]$ is finite for all $z \in \bbH$ by \cite[Lemma B.2]{aizenman2006stability},  there exists $c_1 \in (0,1)$ such that
\be\notag
\Phi_L(s; z_2) \ge c_1^{L+1}.
\ee

Using \Cref{l:cutoffclose}, and the  elementary inequality $\big| |x|^s  - |y|^s \big| \le | x -y|^s$, there exist constants $C_2>1$ and $c_2>0$ such that
\begin{align}
\begin{split}
&\left|\E
\left[
\sum_{v \in \bbV(L)} \left|   
 R_{0v}(z_1) 
\right|^s
\right]
-
\E
\left[
\sum_{v \in \bbV(L)} \left|   
 R_{0v}(z_2) 
\right|^s
\right]\right|
\\
&\qquad \le    C_2^{L+1}|z_1 - z_2|^s (\eta  \eta)^{-s/2}
\cdot \eta_2^{c_2 }.
\label{crudebound}
\end{split}
\end{align}
We fix $\mathfrak a( \delta, \eta, s, \mathfrak b) > 0$ such that
\be\label{blueberry3}
 C_2^{L+1}  |E_1 - E_2 |^s \eta^{c_2-s}  
\le \frac{1}{4} \cdot c_1^{L+1} \mathfrak b
\ee
when $|E_1 - E_2 | < \mathfrak{a}$, recalling $L$ was fixed below \eqref{combineme2}.
With this choice, \eqref{crudebound} gives 
\begin{align}\notag
\left|
\E
\left[
\sum_{v \in \bbV(L)} \left|   
 R_{0v}(z_1) 
\right|^s
\right]
-
\E
\left[
\sum_{v \in \bbV(L)} \left|   
 R_{0v}(z_2) 
\right|^s
\right]
\right|
\le \frac{1}{4} \cdot c_1^{L+1} \mathfrak{b}.
\end{align}
Combining the previous line with \eqref{mango1}, \eqref{mango2}, and \eqref{mango3},
we get
\be\notag
\big| \phi_L(s;z_1)
-
\phi_L(s;z_2) \big| < \frac{\mathfrak{b}}{2}.
\ee
Combining the previous line with \eqref{combineme2} and the choice of $L$ below \eqref{combineme2} completes the proof.
\end{proof}
The following lemma allows us to extend the negativity of $\phi(s; E + \iu \eta_0)$ toward the real axis, if $\eta_0$ is sufficiently small.
\begin{lemma}\label{l:goingdown}
Fix $s, \eps  \in (0,1)$ and $ \delta \in (0,1)$. Then for every $\mathfrak b \in (0, \delta)$, there exists $\mathfrak a (s, \delta, \eps,  \mathfrak{b}) \in (0,1)$ such that following holds for all $t \in (\eps,1)$. For every $E\in \bbR$ and  $\eta_0 \in (0, \mathfrak a)$ such that ${\phi(s;E + \iu \eta_0)} < - \delta$,
\bex
\limsup_{\eta \rightarrow 0}
\phi(s; E + \iu \eta) < - \mathfrak b.
\eex
\end{lemma}
\begin{proof}
From the lower bound in \eqref{e:crudephi} and \eqref{e:limitr0j2},  there exists $c_1 \in (0,1)$ such that
\be\label{upcondition}
\Phi_L(s; E + \iu \eta) \ge c_1^{L+1}.
\ee
for all $E\in \bbR$. Let $C_2>1$ and $c_2 > 0$ be the two constants given by \Cref{l:cutoffclose}. 
We begin by considering two points $z_1 = E + \iu \eta_1$ and $z_2 = E + \iu \eta_2$ with $\eta_1, \eta_2 \in (0,1)$ such that $\eta_1 \in [\eta_2^{1+c_2/2}, \eta_2]$, so that $\eta_1 \le \eta_2$. We also suppose that $\phi(s;E + \iu \eta_2) < - \delta/2$. 
We set
\be\label{choices}
L = ( - \log \eta_1)^{1/2}.
\ee

By \Cref{l:cutoffclose} and the  elementary inequality $\big| |x|^s  - |y|^s \big| \le | x -y|^s$, 
\begin{align}
\begin{split}
&\left|\E
\left[
\sum_{v \in \bbV(L)} \left|   
 R_{0v}(z_1) 
\right|^s
\right]
-
\E
\left[
\sum_{v \in \bbV(L)} \left|   
 R_{0v}(z_2) 
\right|^s
\right]\right|
\le   C_2^{L+1}  |\eta_1 - \eta_2 |^s (\eta_1 \eta_2)^{-s/2}
\cdot \eta_2^{c_2 }.
\label{crudebound12}
\end{split}
\end{align}
Then by \eqref{choices} and \eqref{crudebound2}, there exists $\eta_0(s, c_1, c_2, C_2)>0$ such that, if $\eta_2 \in (0 , \eta_0)$, then
\be\label{crudebound2}
\left|\E
\left[
\sum_{v \in \bbV(L)} \left|   
 R_{0v}(z_1) 
\right|^s
\right]
-
\E
\left[
\sum_{v \in \bbV(L)} \left|   
 R_{0v}(z_2) 
\right|^s
\right]\right|
\le \frac{1}{4} \cdot c_1^{L+1}. 
\ee
Using \eqref{upcondition},  \eqref{combineme2}, \eqref{mango1}, and \eqref{mango2}, we find that there exists $C_3 > 1$ such that 
\be \label{phigoingdown}
\big| \phi(s;z_1) - \phi(s;z_2) \big|
\le \frac{C_3}{L} + \frac{1}{4} \cdot c_1^{L+1}
< C_3 (  - \log \eta_2)^{-1/2},
\ee
where we increased the value of $C_3$ in the second inequality.

Now consider a decreasing sequence $\{\nu_k\}_{k=1}^\infty$ of reals such that $\nu_1 < \eta_0$, $\phi(s; E + \iu \nu_1) < -\delta$, and $\nu_{k+1} = \nu_k^{1 + c_2/2}$ for all $k \in \Zplus$. Set $w_k = E + \iu \nu_k$. With $\mathfrak{b}$ chosen as in the statement of this lemma, we will show by induction that there exists $\mathfrak a(\mathfrak b)>0$ such that $\phi(s; w_k) < - \mathfrak b$ for all $k\in \Zplus$ if $\nu_1 < \mathfrak a$.  We may suppose that $\mathfrak b > \delta /2$. 
For brevity, we set $c_4 = 1 + c_2/2$.

We now claim that $\phi(s; w_j) < - \mathfrak b$ for all $j \in \Zplus$; we will prove this claim by induction. 
For the induction hypothesis at step $n \in \Zplus$, suppose that $\phi(s; w_j) < - \mathfrak b$ holds for all $j \le n$.
We will prove the same estimate holds for $j = n+1$. 

The bound \eqref{phigoingdown} gives, for all $k\le n$, that
\begin{align*}
\left|
\phi(s;w_k) - \phi(s;w_{k+1}) 
\right|  &\le C_3 ( - \log \nu_k)^{-1/2}\\
 &= C_3 \left( - \log\left( \nu_1^{c_4^{k}} \right)\right)^{-1/2}
= C_3 c_4^{-k/2} ( - \log \nu_1)^{-1/2}.
\end{align*}
This previous line implies that
\begin{align}\notag
\left|
\phi(s;w_1) - \phi(s;w_{n+1})
\right|
&\le 
\sum_{k=1}^{n} 
\left|
\phi(s;w_k) - \phi(s;w_{k+1})
\right| \\
&\le C_3 ( - \log \nu_1)^{-1/2} \sum_{k=1}^\infty c_4^k\le C_5 ( - \log \nu_1)^{-1/2}\label{returnto}
\end{align}
for some $C_5(C_3, c_4) >1$. We choose $\mathfrak a$ so that 
\bex
C_5 ( - \log \mathfrak a)^{-1/2} < \delta - \mathfrak b.
\eex
Then \eqref{returnto} and the assumption that $\nu_1 < \mathfrak{a}$ implies that 
\be\label{returnto2}
\phi(s;w_{n+1})
\le \phi(s;w_1) + C_5 ( - \log \nu_1)^{-1/2} <  - \delta + (\delta - \mathfrak b) < -\mathfrak b.
\ee
This completes the induction step and shows that $\phi(s; w_j) < - \mathfrak b$ for all $j \in \Zplus$.

We now claim that for any that any $\tilde w = E + \iu \tilde \nu$ with $\tilde \nu \in (0, \mathfrak a)$, we have $
\phi(s; \tilde w) < - \mathfrak b$. To see this, observe that there is a unique index $k \in \Zplus$ such that $\nu_k > \tilde \nu \ge \nu_{k+1}$. Consider the sequence $\{\tilde w_j\}_{j=1}^{k+1}$ defined by $\tilde w_j = w_j$ for $j \neq k+1$, and $\tilde w_{k+1} = E+ \iu \tilde \nu$. Then the same induction argument that gave \eqref{returnto2} also gives 
 \be\label{returnto3}
\phi(s; \tilde w_{k+1})
 < -\mathfrak b.
\ee
This completes the proof.
\end{proof}
The next lemma complements the previous one. It shows that if $\phi(s;E + \iu \eta)$ becomes negative as $z$ approaches the boundary (as $\eta \rightarrow 0$), then this quantity is negative for $\eta$ an entire interval $(0, \mathfrak a)$. We omit the proof, since it is essentially the same as the argument for \Cref{l:goingdown} (done in reverse, choosing $\eta \in [\eta_2, \eta_2^{1-c_2/2}]$ instead of $\eta_1 \in [\eta_2^{1+c_2/2}, \eta_2]$ at each step). Full details of an extremely similar proof are available in \cite[Lemma 12.10]{aggarwal2022mobility}.
\begin{lemma}\label{l:goingup}
Fix $s, \eps \in (0,1)$ and $ \delta \in (0,1)$. Then for every $\mathfrak b \in (0, \delta)$, there exists $\mathfrak a (\delta, \eps, s, \mathfrak b) \in (0,1)$ such that following holds for every $t\in (\eps, 1)$. For every $E\in \bbR$ such that 
\be\label{limsuphypo}
\limsup_{\eta \rightarrow 0}
\phi(s; E + \iu \eta) < - \delta,
\ee
we have
\bex
\sup_{\eta \in (0, \mathfrak{a}]} \varphi (s; E + \mathrm{i} \eta) < -\mathfrak{b}.
\eex
\end{lemma}
Equipped with the previous lemmas, we are now able to prove \Cref{l:phicontinuity}.
\begin{proof}[Proof of \Cref{l:phicontinuity}]
By \Cref{l:goingdown} and \Cref{l:goingup}, there exists $\delta_1(\kappa, \omega) \in (0, 1) $ such that the following two claims hold for all $E\in I$. 
First, if
\bex
\limsup_{\eta\rightarrow 0}
\phi(s; E + \iu \eta) < -\kappa.
\eex
then for all $\eta \in (0, \delta_1]$,
\be \label{down}
\phi(s; E + \iu \eta) < -\kappa + \frac{\omega}{3}.
\ee
Second, if $\eta \in (0, \delta_1]$ and
\bex
\phi(s; E + \iu \eta) < -\kappa + \frac{2 \omega}{3},
\eex
then
\be\label{up}
\limsup_{\eta\rightarrow 0}
\phi(s; E + \iu \eta) < -\kappa + \omega.
\ee
Next, by \Cref{l:goingsideways}, there exists $\delta_2(\delta_1, \kappa, \omega)>0$ such that, if 
\bex
\phi(s; E_1 + \iu \delta_1 ) < -\kappa + \frac{\omega}{3}
\eex
then
\be\label{sideways}
\big|
\phi(s; E_2 + \iu \delta_1 ) - \phi(s; E_1 + \iu \delta_1 )
\big|
< \omega/3
\ee
for all $E_1, E_2 \in I$ such that $|E_1 - E_2| < \delta_2$. 

Under the assumption that $|E_1 - E_2| < \delta_2$, using \eqref{down}, \eqref{up}, and \eqref{sideways} with the choice $\eta = \delta_1/2$ gives 
\bex
\limsup_{\eta\rightarrow 0}
\phi(s; E_2 + \iu \eta) < -\kappa + \omega,
\eex
as desired.
\end{proof}

\section{Continuity of the Free Energy in Disorder Strength}\label{s:lyapunovcontinuity2}

The goal of this section is to prove \Cref{l:bulkphi}. In \Cref{s:contprelim2}, we provide preliminary continuity estimates for the resolvent entries $R_{vw}$ as the disorder parameter $t$ varies. \Cref{s:disorderbootstrap} contains the proof of \Cref{l:bulkphi}. 

In this section, we modify our notation and write $\bfR(t;z) = \bfR(z)$ for the resolvent of $\hamiltonian$ to emphasize the dependence of $\bfR$ on the parameter $t$ in the definition of $\hamiltonian$. Similarly, we write $\Phi_L(s;t;z) = \Phi_L(s;z)$ and $\phi(s;t;z) = \phi(s;z)$. Further, for this section, we drop our convention that $t = g ( K\ln K)^{-1}$ and instead think of $t\in (0.1)$ as a free parameter. All constants in this section may depend on the choice of $K$ in $\hamiltonian$, and we always omit this dependence in the notation. 

Some aspects of this appendix are analogous to \Cref{s:lyapunovcontinuity}, and consequently, some proofs will be omitted. 

\subsection{Continuity Estimates for the Disorder Parameter}\label{s:contprelim2}
We begin with a lemma that parallels \Cref{l:cutoffclose}, except that is concerns continuity in $t$ instead of $z$. It will similarly form the basis of our bootstrapping argument for the parameter $t$. 
\begin{lemma}\label{l:cutoffcloseT}
Fix $s \in (0,1)$ and $\delta >0$. There exist constants $C(\delta, s) > 1$ and $c(s) > 0$ 
such that the following holds. Fix $ z = E + \iu \eta \in \bbH$ with $\eta \in (0,1)$, and for  $i=1,2$, fix $t_i \in (0,1)$.
 Suppose that $\phi(s; t;_2;z) < - \delta$. Then
\bex
\E
\left[
\sum_{v \in \bbV(L)} \left|   
 R_{0v}(t_1;z) -  R_{0v}(t_2;z) 
\right|^s 
\right] 
\le  C^{L+1}|t_1 - t_2|^s  \eta^{ -s + c}.
\eex
\end{lemma}
\begin{proof}
For $k \in \unn{0}{L-2}$, we define
\begin{align}\label{PsiDefT}
&\Psi_{L, k} (s;  t_1, t_2 ; z)\notag \\
&= 
\E\left[ 
\sum_{w \in \bbV(k+1)}
\sum_{v \in \bbD_{L-k-1}(w)}
\left|
 R_{0 w_-}(t_2; z )  t_1  R^{(w_-)}_{w v }(t_1 ; z)
-  R_{0 w}(t_2;z)  t_1 R^{(w)}_{w_+ v } (t_1 ; z_1)
\right|^s
\right].
\end{align}
We also set 
\bex
\Psi_{L, -1} (s; t_1, t_2 ; z)=
\E\left[ 
\sum_{w \in \bbV(1)}
\sum_{w \in \bbD_{L-1} (w)}
\left|
 R_{0 v}(t_1 ; z )
- R_{0  0 }(t_2; z )  t_1  R^{(0)}_{w v}(t_1 ; z  )
\right|^s
\right]
\eex
and
\bex
\Psi_{L, L-1} (s;  t_1, t_2; z )=
\E\left[ 
\sum_{w \in \bbV(L-1)}
\sum_{v \in \bbD(w)}
\left|
 R_{0  w }(t_2; z )  t_1   R^{(w)}_{vv}(t_1; z )
- 
 R_{0 v}(t_2 ;z )
\right|^s
\right].
\eex
By 
the definitions of the $\Psi_{L, j}$, and the elementary inequality that $(a+b)^s \le a^s + b^s$ for all $a,b \ge 0$, 
\be\label{thephisumT}
\E
\left[
\sum_{v \in \bbV(L)} \left|   
 R_{0v}(t_1;z ) -  R_{0v}(t_2;z ) 
\right|^s 
\right] \le 
\sum_{j = -1}^{L-1} \Psi_{L, j} (s; z_1, z_2).
\ee
To complete the proof, it suffices to bound the sum on the right side of \eqref{thephisumT}.

Considering the terms in the sum defining \eqref{PsiDefT},  \Cref{rproduct} gives
\begin{align}\label{e:febsplitT}
\begin{split}
&   R_{0 w_-}(t_2; z )  t_1  R^{(w_-)}_{w v }(t_1 ; z)
-  R_{0 w}(t_2;z)  t_1 R^{(w)}_{w_+ v } (t_1 ; z_1) \\
& =  R_{0 w_-}(t_2;z)  t_1  R^{(w_-)}_{w w}(t_1; z) 
 t_1
 R^{(w)}_{w_+ v} (t_1; z)
 -  
 R_{0 w_-}(t_2;z)  t_2  R^{(w_-)}_{w w}(t_2; z) 
 t_1
 R^{(w)}_{w_+ v} (t_1; z)\\
&=  R_{0 w_-}(t_2;z) 
 \big( t_1  R^{(w_-)}_{w w}(t_1; z)  -    t_2 R^{(w_-)}_{w w}(t_2; z)  \big)
 t_1
 R^{(w)}_{w_+ v} (t_1; z).
 \end{split}
\end{align}
Then using \eqref{e:febsplit} in the definition of $\Psi_{L, k}$ in \eqref{PsiDef} gives
\begin{align*}
&\Psi_{L, k} (s;  t_1, t_2; z )\notag\\
&\le 
\E\left[
\sum_{w \in \bbV(k+1)}
\sum_{v \in \bbD_{L-k-1}(w)}
\left|
R_{0 w_-}(t_2;z) 
 \big( t_1  R^{(w_-)}_{w w}(t_1; z)  -    t_2 R^{(w_-)}_{w w}(t_2; z)  \big)
 R^{(w)}_{w_+ v} (t_1; z)
\right|^s 
\right].
\end{align*}
To get the inequality, we used the assumption that $t_1 \in (0,1)$. 

We now use \Cref{l:uniformlyintegrable} to bound the expectation of each term in the sum. We apply this lemma after using the law of total expectation and conditioning on all disorder variables $V_{x}$ corresponding to vertices not in the path from $0$ to $w_-$. This gives 
\begin{align}
\Psi_{L, k} &(s;  t_1, t_2; z )\notag \\
&\le 
 C^{k+1} \E\left[
\sum_{w \in \bbV(k+1)}
\sum_{v \in \bbD_{L-k-1}(w)}
 \left| \big( t_1  R^{(w_-)}_{w w}(t_1; z)  -    t_2 R^{(w_-)}_{w w}(t_2; z)  \big)
 R^{(w)}_{w_+ v} (t_1; z) 
\right|^s 
\right]\notag\\
&\le 
C^{k+1} \E\left[
\sum_{w \in \bbV(1)}
\sum_{v \in \bbD_{L-k-2}(w)}
  \left|\big( t_1 R_{00}(t_1;z) -   t_2 R_{00}(t_2;z) \big)
 R^{(0)}_{w v} (t_1;z)
\right|^s 
\right]\label{phiprev2T}
\end{align}
for some $C > 1$, where the last inequality follows from the equivalence of the joint law of $R^{(w_-)}_{w w}(t_1; z)$, $R^{(w_-)}_{w w}(t_2; z) $, and $ R^{(w)}_{w_+ v} (t_1; z) $ in the second line with that of $R_{00}(t_1;z)$,  $R_{00}(t_2;z)$, and $ R^{(0)}_{w v} (t_1;z)$ in the third line. 
Note that 
\begin{equation}\label{e:febphiT}
 t_1 R_{00}(t_1;z) -   t_2 R_{00}(t_2;z)
 =  (t_1 - t_2) R_{00}(t_1;z)  + t_2\big( R_{00}(t_1;z) -  R_{00}(t_2;z)\big).
\end{equation}
We consider the two terms arising from \eqref{phiprev2T} by substituting in the expression \eqref{e:febphiT}. For the first term, we have 
\begin{align}\label{e:newtermT}
\begin{split}
&  \E\left[
\sum_{w \in \bbV(1)}
\sum_{v \in \bbD_{L-k-2}(w)}
  \left|(t_1 - t_2) R_{00}(t_1;z)  
 R^{(0)}_{w v} (t_1;z)
\right|^s 
\right]\le C^{L -k + 1} |t_1 - t_2|^s ,
\end{split}
\end{align}
where we again used  \Cref{l:uniformlyintegrable} to bound the expectation of each term in the sum, and increased the value of $C$ if necessary. 
For the second term, we note that the resolvent identity $\bm{A}^{-1} - \bm{B}^{-1} = \bm{A}^{-1} (\bm{B} - \bm{A}) \bm{B}^{-1}$ implies 
\begin{align}
\begin{split}
&t_2\big( R_{00}(t_1;z) -  R_{00}(t_2;z)\big) \\
&= t_2 \sum_{v,w \in \bbV} R_{0v}(t_1;z) \big( (t_1 - t_2) A_{v w}  \big) R_{w0}(t_2 ; z) \\ 
&= t_2 (t_1 - t_2) \left( \sum_{v \in \bbV} R_{0v}(t_1;z) R_{v_+ 0}(t_2;z)  + \sum_{v \in \bbV} R_{0v_+}(t_1;z) R_{v0}(t_2;z) \right).
\end{split}
\end{align}
By using H\"older's inequality and the Ward identity \eqref{sumrvweta} in the previous line, and using $t_2 \in (0,1)$, we find
\begin{align}
&\left|t_2\big( R_{00}(t_1;z) -  R_{00}(t_2;z)\big) \right| \\
&\le |t_1 - t_2| 
\left(\sum_{v \in \bbV} \big |R_{0v}(t_1;z )\big|^2 \right)^{1/2}
\left(\sum_{v \in \bbV} \big |R_{0v}(t_2;z ) \big|^2 \right)^{1/2}\notag
\\
& = |t_1 - t_2| \eta^{-1} 
\big( \Im R_{00}(t_1; z ) \big)^{1/2}
\big( \Im R_{00}(t_2; z ) \big)^{1/2}.\label{holder2T}
\end{align}
Inserting \eqref{e:newtermT} and  \eqref{holder2T} into \eqref{phiprev2T} gives
\begin{align}
\begin{split}
&\Psi_{L, k} (s; z_1, z_2)\\ &\le  C^{L  + 1} |t_1 - t_2|^s\\
&\quad +  C^{k+1}|t_1 - t_2|^s \eta^{-s} \\
&\quad\quad  \times
  \E\left[
\sum_{w \in \bbV(1)}
\sum_{v \in \bbD_{L-k-2}(w)}
\big( \Im R_{00}(t_1; z ) \big)^{s/2}
\big( \Im R_{00}(t_2; z ) \big)^{s/2}
\left|
R^{(0)}_{w v}(t_1;z)
\right|^s  \label{intPhiT}
\right].
\end{split}
\end{align}

Let $\kappa = s/4$. By Markov's inequality, the assumption $\phi(s;t_2,z) < - \delta$, and \Cref{l:qsexpectation}, 
\be\label{simplemarkovT}
\P\left( \big( \Im R_{00}(t_2;z ) \big)^{s/2} > \eta^\kappa  \right) \le C \eta^{s/2 - \kappa}.
\ee
Define the event 
\bex
\mathscr A = \left\{ \big( \Im R_{00}(t_2;z ) \big)^{s/2} < \eta^\kappa \right\}.
\eex
Then 
\begin{align}\label{ongoodeventT}
\begin{split}
&\E\left[\one_{\mathscr A}
\sum_{w \in \bbV(1)}
\sum_{v \in \bbD_{L-k-2}(w)}
\big( \Im R_{00}(t_1; z ) \big)^{s/2}
\big( \Im R_{00}(t_2; z ) \big)^{s/2}
\left|
R^{(0)}_{w v}(t_1;z)
\right|^s
\right]
\\
&\le \eta^{\kappa} \cdot \E\left[
\sum_{w \in \bbV(1)}
\sum_{v \in \bbD_{L-k-2}(w)}
\big| R_{00}(t_1; z ) \big|^{s/2}
\left|
R^{(0)}_{w v}(t_1;z)
\right|^s
\right]
\le 
\eta^{\kappa}
C^{L-k},\end{split}
\end{align}
where the last inequality follows from two applications of \Cref{l:uniformlyintegrable}.

 We next note that the elementary inequality $ab \le a^2 + b^2$ gives 
\bex
\big( \Im R_{00}(t_1;z) \big)^{s/2} \big( \Im R_{00}(t_2;z) \big)^{s/2}  \le\big( \Im R_{00}(t_1;z) \big)^{s}
 +
\big( \Im R_{00}(t_2;z) \big)^{s} .
\eex
The previous line implies
\begin{align}
\begin{split}
& \E\left[
\one_{\mathscr A^c}
\sum_{w \in \bbV(1)}
\sum_{v \in \bbD_{L-k-2}(w)}
\big( \Im R_{00}(t_1;z) \big)^{s/2}
\big( \Im R_{00}(t_2;z) \big)^{s/2}
\left|
t
R^{(0)}_{w v}(t_1;z)
\right|^s 
\right]\\
&\le  \E\left[
\one_{\mathscr A^c} \sum_{w \in \bbV(1)}
\sum_{v \in \bbD_{L-k-2}(w)}
\big( \Im R_{00}(t_1;z) \big)^{s}
\left|
t
R^{(0)}_{w v}(t_1;z)
\right|^s 
\right] \\
&+
 \E\left[
\one_{\mathscr A^c}\sum_{w \in \bbV(1)}
\sum_{v \in \bbD_{L-k-2}(w)}
\big( \Im R_{00}(t_2;z) \big)^{s}
\left|
t
R^{(0)}_{w v}(t_1;z)
\right|^s 
\right]\\
&\le \P(\mathscr A^c)^{c_1}  \cdot  C^{L-k}
= \eta^{c_1s } \cdot  C^{L-k}
\end{split}
\label{onbadeventT}
\end{align}
for some $c_1(s) > 0$, where we used \eqref{simplemarkovT} and 
the following bound in the last line. By H\"older's inequality and  \Cref{l:uniformlyintegrable}, we have, after fixing $\kappa >0$ such $(1+\kappa) s < 1$, that for $i=1,2$,
\begin{align*}
&\E\left[
\one_{\mathscr A^c}\sum_{w\in \mathbb{V}(1)} \sum_{v \in \bbD_{L-1}(w) }
\big( \Im R_{00}(t_i;z) \big)^{s}
\left|
t
R^{(0)}_{w v}(t_1;z)
\right|^s
\right]\\
&\le 
\P(\mathscr A^c)^{\kappa /( 1 + \kappa)} 
\E\left[\sum_{w\in \mathbb{V}(1)} \sum_{v \in \bbD_{L-1}(w) }
\left(\big( \Im R_{00}(t_i;z) \big)^{s} 
\left|
t
R^{(0)}_{w v}(t_1;z)
\right|^s \right)^{1+\kappa} 
\right]^{1/(1+\kappa)}\\
&\le 
\P(\mathscr A^c)^{\kappa / ( 1 + \kappa)} \cdot C^{1+L},
\end{align*}
The conclusion follows after inserting \eqref{ongoodeventT} and \eqref{onbadeventT} into \eqref{intPhiT}, and using \eqref{thephisumT}.
\end{proof}

The next lemma states a continuity estimate for $\phi(s; t; z)$ as $z$ is fixed and $t$ varies. 
We omit the proof since, given the previous lemma, it is essentially identical to that of \Cref{l:goingsideways}. 
\begin{lemma}\label{l:goingsidewaysT}
Fix $s, \delta, \eps, \eta \in (0,1)$ and $ E \in \R$. Then for every $\mathfrak b >0$, there exists $\mathfrak a (s, \delta, \eta, E, \mathfrak b) > 0$ such that following holds. For every $t_1, t_2\in (\eps, 1)$ such that  $|t_1 - t_2| < \mathfrak a$ and  ${\phi(s;t_2; E + \iu \eta) < - \delta}$,
we have
\bex
\big|
\phi(s;t_1; E + \iu \eta ) - \phi(s;t_2 ; E+ \iu \eta)
\big| < \mathfrak b.
\eex
\end{lemma}
Combining the previous lemmas in this section, along with \Cref{l:goingdown} and \Cref{l:goingup}, we find the following lemma. Again, as its proof is very similar to the proof of \Cref{l:phicontinuity}, we omit it.
\begin{lemma}\label{l:phicontinuityT}
	Fix $s, \eps \in (0,1)$, 
	$\kappa >0$, and a compact interval $I \subset \bbR$. There exists a constant $\delta(s,\eps, \kappa, I) >0$ such that the following holds. For every $t_1, t_2 \in (\eps, 1)$ such that $|t_1 - t_2 | \le \delta$ and 
\bex
\limsup_{\eta \rightarrow 0} \varphi (s; t_1; E + \mathrm{i} \eta) < - \kappa,
\eex
we have	
\bex
\limsup_{\eta \rightarrow 0} \varphi (s; t_2; E + \mathrm{i} \eta) 
< 
-\kappa + \omega.
\eex
\end{lemma}
\subsection{Disorder Bootstrap}\label{s:disorderbootstrap}
\begin{proof}[Proof of \Cref{l:bulkphi}]
Fix $E \in D_-$. With the value of $g$ given by the lemma statement, we define $t' = g ( K \ln K)^{-1}$ and treat $t \in (0,t')$ as a free parameter in this proof. By \Cref{l:rougheigenvector} and \eqref{e:rhoEupper}, there exists $K_0(E) >0$ such that for all $K \ge K_0$, there exists $s_0(K) \in (0,1)$ such that  we have 
\be\label{e:the200}
\lambda_{K,t,s,E} \le K^{-1}( 1 - \L^{-200})
\ee
 for all  $t \in [0, t']$ and $s \in [s_0(K) , 1 )$. 

We now fix an infinite path $\mathfrak{p} = (v_0, v_1, \ldots )$ of vertices in $\mathbb{V}$, with $v_0 = 0$. By \Cref{l:uniformlyintegrable}, for every $K >1$, there exists $t_0(K) > 0$ and a constant $C>0$ such that that for all $t \in (0, t_0]$ and $z \in \bbH$, we have 
\begin{equation}
\E\Big[ \big|R_{0v_L }(t; z) \big|^{1/2}  \Big]  \le C (Ke)^{-L}.
\end{equation}
From the definition of $\phi(s;z)$ in \eqref{szl}, we obtain that 
\be
\phi(1/2; t ; z ) \le  - 1 
\ee
for all $t \in (0, t_0]$ and $z \in \bbH$. By the first part of \Cref{l:aizenman}, we have 
\be\label{e:phialls}
\phi(s; t ; z ) \le  - 1 
\ee
for all $s \in (1/2, 1)$. 

For all $s \in (0,1)$, set 
\bex
\kappa' = \kappa'(E,g,s) = \inf_{t \in [t_0, t'] } \big( 1 - K  \lambda_{K,t, s,E} \big),
\qquad \kappa = \kappa(E,g,s) = \frac{\min ( \kappa', 1  )}{2}.
\eex
Recalling \eqref{e:the200}, we have $\kappa' >0$ for $s=s_0$. 

Using \Cref{l:phicontinuityT} on the interval $[t_0, t']$ and with $\omega$ equal to $\kappa$, let $\delta$ be the constant given by \Cref{l:phicontinuityT}.  
Let $\{t_i\}_{i=1}^M$ be a set of real numbers such that $t_M = t'$ and 
\bex
t_{i+1} < t_i, \qquad | t_{i+1} - t_i | < \delta
\eex
for all $i \in \unn{0}{M-1}$. 
We claim that for every $i\in \unn{0}{M}$, we have 
\be\label{bobT}
\limsup_{\eta \rightarrow 0} \exp\big(\phi(s_0; t_i ; E+ \iu \eta)\big) \le  K \lambda_{K, t_i, s_0,E} .
\ee
We will show the claim by induction on $i$. The base case $i=0$  follows combining \eqref{e:phialls}, \Cref{p:imvanish}, and \Cref{l:bootstrap}. 
Next, for the induction step, we assume that the induction hypothesis holds for some $i\in \unn{0}{M-1}$, and we will show it holds for $i+1$. Using the induction hypothesis at $i$ and the definition of $\kappa$, we have 
\bex
\limsup_{\eta \rightarrow 0} \exp\big(\phi(s;t_i;  E + \iu \eta)\big)  \le K \lambda_{s,E_i} \le 1 - 2 \kappa.
\eex
By \Cref{l:phicontinuityT} and the definition of $\delta$, 
\begin{align}\label{p19T}
\limsup_{\eta \rightarrow 0} \varphi (s; t_{i+1} ; E + \mathrm{i} \eta)
&< \limsup_{\eta \rightarrow 0} \varphi (s; t_i ; E + \mathrm{i} \eta) + \kappa\le  \log (1 - 2 \kappa)  +  \kappa \le - \kappa,
\end{align}
where the last inequality follows from the elementary bound $\ln(1-x) < -x$. 
Using \Cref{p:imvanish}, we see that \eqref{p19} implies
\be\label{jcvgT}
\lim_{\eta \rightarrow 0} \Im R_{00}(t_{i+1}; E + \iu \eta) = 0.
\ee
Let $\{\eta_j\}_{j=1}^\infty$ be a sequence such that
\be\label{otterT}
\lim_{j \rightarrow \infty}
\exp \left( \phi(s; t_{i+1}; E  + \iu \eta_j) \right)= \limsup_{\eta\rightarrow 0} \exp \left( \phi(s; t_{i+1} ; E+ \iu \eta) \right).
\ee
Then using \eqref{jcvgT} and \Cref{l:bootstrap}, equation \eqref{otterT} implies that 
\bex
\limsup_{\eta \rightarrow 0} \exp \big( \phi(s; E_{i+1} + \iu \eta) \big) = \lim_{j \rightarrow \infty}
\exp \big( \phi(s; E_{i+1} + \iu \eta_j) \big) = K \lambda_{E_{i+1},s}.   
\eex 
This completes the induction step and shows that \eqref{bobT} holds for all $i\in \unn{0}{M}$. In particular, taking $i=M$, we have 
\be
\limsup_{\eta \rightarrow 0} \exp\big(\phi(s; t'; E + \iu \eta)\big) \le  \lambda_{K, t', E,s}.
\ee
By \Cref{p:imvanish}, this implies $ \lim_{\eta \rightarrow 0} \Im R_{00} (t'; E + \iu \eta) = 0$.

Recall that the choice of $E \in D_-$ was arbitrary.
We conclude that for every $E \in D_-$, we have 
\begin{equation}
\lim_{\eta \rightarrow 0} \Im R_{00} (t'; E + \iu \eta) =0
\end{equation}
in probability.  Then by \Cref{l:bootstrap} and \eqref{e:the200}, we have $\phi(1; t' ;  E) < 0$ for every $E$ where the limit (as $\eta \rightarrow 0$) defining $\phi(1;t',E)$ exists; this limit exists for all $E \in D_-$ except for a set of measure zero, by the sixth part of \Cref{l:aizenman}. This concludes the proof.
\end{proof}

\section{Proofs of Auxiliary Results from \Cref{s:proof}} 
This appendix contains the proofs of  \Cref{l:pEexists} and \Cref{l:bigE}.

\subsection{Proof of  \Cref{l:pEexists}}\label{s:pEexists}
We now show \Cref{l:pEexists}, which states that \eqref{e:rde} admits at least one solution. Our proof uses the well-known fixed point theorem of Tychonoff. The formulation here is taken from \cite[II.7]{granas2003fixed}.
\begin{lemma}[{\cite[II.7, Theorem (1.13)]{granas2003fixed}}]\label{l:KRgranas}
Let $V$ be a locally convex topological space, and let $W$ be a compact, convex subset of $V$. Then every continuous function $f : W \rightarrow W$ has a fixed point.
\end{lemma}
Given this lemma, we proceed to the proof of \Cref{l:pEexists}. 
\begin{proof}[Proof of \Cref{l:pEexists}]
Let $\mathcal W$ denote the space of probability measures $\mu$ on $\R$ such that for every $r \ge 1$, we have $\mu([-r,r]^c) \le 2  \L r^{-1}$. 

Next, let $S$ be the map that takes a probability measure $\mu\in \mathcal W$ to the probability measure corresponding to the distribution of 
\begin{equation}\label{e:rdemap}
\frac{1}{ V_0  - E  - t^2 \sum_{i=1}^K X_i},
\end{equation}
where $X_1,\dots, X_K$ are independent, identically distributed random variables with common distribution $\mu$. Since $V_0$ is continuous with density $\rho$, and $\| \rho \|_\infty \le \L$, the denominator of \eqref{e:rdemap} is also continuous with a density bounded by $\L$. If $\nu$ is the probability measure of the corresponding distribution, this implies that 
\begin{equation}\label{e:nueps}
\nu( [-\eps, \eps ]) \le 2  \L \eps,
\end{equation} which implies that $(S\mu)( [-r,r]) \le 2 \L r^{-1}$, so $S\mu \in \mathcal W$. 

The bound \eqref{e:nueps} also implies that the map $S$ is continuous, since if $(\mu_n)_{n=1}^\infty$ converges to $\mu$ in $\mathcal W$, and $f \in C_b(\R)$, we can write
\[
\E_{S\mu_n}[f] = \E_{S\mu_n}[f \one_{|x| \le \eps^{-1}}]  + \E_{S\mu_n}[f\one_{|x| > \eps^{-1}}].
\]
Then first term converges to $\E_{S\mu}[f \one_{|x| \le \eps^{-1}}]$, by the continuous mapping theorem, and the second can be made arbitrarily small using \eqref{e:nueps} and taking $\eps$ sufficiently small. Hence $\lim_{n\rightarrow \infty} \E_{S\mu_n}[f]  = \E_{S\mu}[f] $ for every $f \in C_b(\R)$, showing $S$ is continuous on $\mathcal W$.

Finally, note that $\mathcal W$ is a subset of $\mathcal M$, the space of finite, signed Borel measures on $\R$. We note that $\mathcal M$ is the dual of $C_b(\R)$, the space of bounded, continuous functions on $\R$ endowed with the supremum norm, and equip it with the weak-* topology. Since the weak-* topology is generated by seminorms, $\mathcal M$ is locally convex. 
By Prokhorov's theorem, the set $\mathcal W$ is sequentially compact in this topology. The induced topology on $\mathcal W$ is metrizable (by the Prokhorov metric), so its sequential compactness implies that it is also compact (since these concepts are equivalent for metric spaces). We have shown that $S$ is a continuous self-map of a compact set of locally convex topological vector space, \Cref{l:KRgranas} implies it must have a fixed point, which is the desired solution to \eqref{e:rde}.
\end{proof}

\subsection{Proof of \Cref{l:bigE}} \label{s:bigEproof}

We fix an infinite path $\mathfrak{p} = (v_0, v_1, \ldots ) \in \mathbb{V}$, with $v_0 = 0$.
For every non-negative integer $n$, define
\be\label{e:gammadef}
\gamma_n(s;z) = \E \left[ \big|R^{(v_{n+1})}_{0 v_n}(z) \big|^s \right] .\ee 

The following two bounds for $\gamma_n$ are taken from \cite{aizenman2013resonant}. The first is \cite[Lemma 3.3]{aizenman2013resonant}, and the second is \cite[Lemma 3.4]{aizenman2013resonant}.  We recall that $K$ denotes the branching number of the tree $\bbT$ (defined in \eqref{s:tightbinding}). 
	 \begin{lemma}[{\cite[Lemma 3.3 and Lemma 4.4]{aizenman2013resonant}}] \label{l:aizenman2} 	
	 	Fix $s, \eps \in (0,1)$. \begin{enumerate}
\item 
There exists a constant $C(s,\eps)>0$ such that for all $t \in (\eps, 1)$, $z \in \bbH$, $K >1$, and  non-negative integers $m$ and $n$,
\begin{equation}\label{e:submultiplicative}
\gamma_{m+n}(s;z) \le C \gamma_m(s;z)  \gamma_n(s;z) .
\end{equation}
\item 
Fix $A, K>0$. 
There exists a constant $C_0(s,\eps,A, K)>0$ such that for all $t \in (\eps, 1)$, $z \in \bbH$ with $|z| \le A$, and  non-negative integers $n$,
\begin{equation}\label{e:offdiagonalbound}
 \E \left[ \big|R_{0 v_n}(z) \big|^s \right]  \le C_0 \gamma_n(s;z).
\end{equation}
\end{enumerate}
\end{lemma}
\begin{remark}
In the stating this lemma, we have taken into account how our definition of $\hamiltonian$ (in \eqref{e:hamiltonian}) differs (by a scaling) from the operator considered in   \cite{aizenman2013resonant}. We also used the fact, noted at \cite[(3.20)]{aizenman2013resonant}, that the constant $C$ in \eqref{e:submultiplicative} can be taken to be independent of the spectral parameter $z$. 
\end{remark}

\begin{proof}[Proof of \Cref{l:bigE}]
Recall the definition of the path $\mathfrak{p}$ from above \eqref{e:gammadef}.
Using \eqref{e:offdiagonalbound}, and iterating \eqref{e:submultiplicative} $L-1$ times (setting $m=1$ in \eqref{e:submultiplicative} each time), we find there exist constants $C(s, t), C_0(s, t, E)>0$ such that
\begin{equation}\label{e:gamma00}
 \E \left[ \big|R_{0 v_L}(z) \big|^s \right]  \le C_0 \gamma_L(s;z)
 \le C_0 C^L \gamma_1^L(s;z). 
\end{equation}
We now aim to bound $\gamma_1(s;z)$. 
From \eqref{e:submultiplicative}, there exists $C(s, t)>0$ such that
\begin{equation}\label{e:gamma0}
\gamma_1(s;z) \le C \gamma_0(s;z) = C \cdot  \E \left[ \big | R_{0 0}^{(v_1)}  \big|^s \right].
\end{equation}
 For any vertices $u,w \in \bbV$, we abbreviate $R_{uw} = R_{uw}(z)$. 
From \Cref{rproduct} and \eqref{qvv}, we have 
\begin{align}\label{e:largeEstep0}
\begin{split}
| R_{0 0}^{(v_1)} |  
&= 
\frac{ t }{\left|z - V_0 - t^2 \sum_{j=1}^{K-1} \Gamma_j \right|} ,
\end{split}
\end{align}
where the $\Gamma_j$ are i.i.d.\ with the same distribution as $R_{00}(z)$. 
Define the event 
\begin{align}
\mathcal F = \left\{  \left| V_0 -  t^2 \sum_{j=1}^{K-1} \Re \Gamma_j \right| < \frac{E}{2} \right\},
\end{align}
and observe from \eqref{e:largeEstep0} that 
\begin{align}\label{e:largeEstep1}
\E\big[\one_{\mathcal F} | R_{0 0}^{(v_1)} |^s \big ] =
 \E\left[ \one_{\mathcal F} \cdot \frac{ t^s  }{\left|z - V_0 - t^2 \sum_{j=1}^{K-1} \Gamma_j \right|^s}  \right] 
 & \le\left( \frac{2t }{|E|} \right)^s. 
\end{align}
Next, using \Cref{l:uniformlyintegrable} and Markov's inequality, we see that there exists $C(s,t) > 0$ such that 
\be
\E \big[ |\Gamma_1|^{1/2} \big] < C, \qquad \P \big( |\Re \Gamma_1| > x\big) \le C |x|^{-1/2}.
\ee 
From a union bound, we obtain that there exists $C(s,t,K)>0$ such that
\begin{equation}\label{e:gammaunion}
\P
\left(
\left|
 t^2  \sum_{j=1}^{K-1} \Re \Gamma_j 
 \right | \ge \frac{|E|}{4} 
\right) \le (K-1) \cdot \P \left(  |\Re \Gamma_1| > \frac{ |E|}{ 4 t^2} \right) \le C|E|^{-1/2}.
\end{equation}
Further, by the assumption that $\rho$ is $\L$-regular (the first part of \Cref{d:Lregular}), we have 
\be\label{e:Vdecay}
\P \left( | V_0|   > \frac{|E|}{4}  \right) \le \frac{C}{|E|^2}.
\ee 
Combining \eqref{e:gammaunion} and \eqref{e:Vdecay},  we find that there exists a constant $C(s,t,K)>0$ such that 
\begin{equation}
\P( \mathcal F^c) \le C |E|^{-1/2}. 
\end{equation}
From H\"older's inequality, we find using the $s=1/2$ case of the previous line that that 
\begin{equation}\label{e:appholder}
\E\big[\one_{\mathcal F^c} | R_{0 0}^{(v_1)} |^{1/2} \big ]  \le \P(\mathcal F^c)^{1/3} \cdot  \E\big[ | R_{0 0}^{(v_1)} |^{3/4} \big ]^{2/3} \le C |E|^{-1/6}.
\end{equation}
In the last inequality, we used \eqref{e:largeEstep0} to show that that there exists $C(t)>0$ such that
\begin{align}\label{e:largeEstep2}
\E\big[ | R_{0 0}^{(v_1)} |^{3/4} \big ] =
 \E\left[ \frac{ t^s  }{\left|z - V_0 - t^2 \sum_{j=1}^{K-1} \Gamma_j \right|^{3/4}}  \right] 
 & \le C,
\end{align}
where in the second inequality we conditioned on all variables but $V_0$, took expectation in $V_0$, and used the bound 
\begin{equation}
\sup_{a\in \R} \E \left[  \frac{1}{|V_0- a |^s} \right] \le \frac{ (2 \| \rho \|_\infty)^s }{1-s}
\end{equation}
from \cite[(A5)]{aizenman2013resonant}, which holds for all $s\in(0,1)$. 
Combining \eqref{e:appholder} with the $s=1/2$ case of \eqref{e:largeEstep1}, we find 
\begin{equation}
\E\big[| R_{0 0}^{(v_1)} |^{1/2} \big ]  \le C |E|^{-1/6}.
\end{equation}
Inserting the previous line into \eqref{e:gamma0} and using \eqref{e:gamma00}, we find
\be
 \E \left[ \big|R_{0 v_L}(z) \big|^{1/2} \right]  \le C_0 C^L |E|^{-L/6}. 
\ee 
Since this bound is independent of $\eta$, we conclude using the definition of $\phi(1/2; E+\iu \eta)$ in \eqref{szl} (and \eqref{e:phiLKLsketch}) 
that for all $\eta \in (0,1)$,
\begin{equation}
\phi(1/2; E+\iu \eta) \le C + \log K - \frac{ \log |E|}{6},
\end{equation}
where $C(s,t) >0$ does not depend on $E$. 
By the first part of \Cref{l:aizenman} and the previous line, we find that 
\begin{equation}\label{e:takeetatozero}
\phi(1; E+\iu \eta)  \le \phi(1/2; E+\iu \eta) \le C +  \log K - \frac{ \log |E|}{6}
\end{equation}
for all $\eta \in (0,1)$. 
We choose $\mathfrak B>0$ so that $|E| \ge \mathfrak B$ implies that  $(\log |E|) /6  \ge  \log K + C + 2$. 
This completes the proof after taking $\eta$ to zero in \eqref{e:takeetatozero}.
\end{proof}

\section{Proofs of Results from \Cref{s:misc}}\label{s:miscproofs}
This appendix provides the proofs of certain assertions from \Cref{s:miscproofs}. 
\Cref{l:32} and  \Cref{l:integrallowerbound} are proved in \Cref{s:integralboundproofs}, while \Cref{s:densityproof} contains the proof of  \Cref{l:rhoebound}.
\subsection{Proofs of Integral Bounds}\label{s:integralboundproofs}
\begin{proof}[Proof of \Cref{l:32}]
Using the bound $1 + |x|^{2-s} \ge 1$ in the definition of $I(A)$, we find that 
\begin{equation}\label{e:firsthorn}
I(A) \le \int_{-\infty}^\infty \frac{dx }{(1 + (x-A)^2)}\le  \int_{-\infty}^\infty \frac{dx }{1 + x^2} \le C. 
\end{equation}
We may now suppose that $|A| \ge 5$, since the case $|A| \le 5$ is covered by \eqref{e:firsthorn}.
We further suppose that $A>0$; the other case is similar.
Set
\begin{align}
\begin{split}
I_1(A) &= \int_{-\infty}^{-A/2} \frac{dx }{(1 + (x-A)^2)(1+|x|^{2-s} )},\\
I_2(A) &= \int_{-A/2 }^{A/2} \frac{dx }{(1 + (x-A)^2)(1+|x|^{2-s} )},\\
I_3(A) &= \int_{A/2 }^\infty \frac{dx }{(1 + (x-A)^2)(1+|x|^{2-s} )}.
\end{split}
\end{align}
Using that $|x | \ge A/2$ implies $|x|^{2-s} + 1 \ge |A/2|^{2-s}$ in the integrand defining $I_1$, we have 
\be\label{e:I1}
I_1(A) \le \int_{-\infty}^{-A/2} \frac{dx }{(1 + (x-A)^2)|x|^{2-s}} 
\le 
\left(\frac{2}{A}\right)^{2-s} \int_{-\infty}^{\infty } \frac{dx }{1 + (x-A)^2}
\le C A^{-(2-s)}.
\ee 
Nearly identical reasoning gives 
\be
I_3 (A) \le\int_{A/2 }^\infty \frac{dx }{(1 + (x-A)^2) |x|^{2-s} } \le 
\left(\frac{2}{A}\right)^{2-s} \int_{-\infty}^{\infty } \frac{dx }{1 + (x-A)^2}
\le C A^{-(2-s)}.
\ee 
For $I_2$, recall that we suppose $|A| \ge 5$. 
Set $J_1 = [-1,1]$ and $J_2 = [-A/2, A/2] \setminus [-1,1]$. 
From the definition of $I_2$, we obtain
\begin{align}
\begin{split}
I_2(A) &= \int_{J_1} \frac{dx }{(1 + (x-A)^2)(1+|x|^{2-s} )}+  \int_{J_2}  \frac{dx }{(1 + (x-A)^2)(1+|x|^{2-s} )}\\
&\le \left( \frac{2}{A} \right)^{2}  \int_{J_1} \frac{dx }{(1+|x|^{2-s} )}  +  \int_{J_2}  \frac{dx }{(x-A)^2 |x|^{2-s} }\\
&\le 2\left( \frac{2}{A} \right)^{2-s}  + \int_{J_2} \frac{dx }{|x-A|^{2-s} |x|^{2} }\\ 
&\le 2\left( \frac{2}{A} \right)^{2-s}  +  \left( \frac{2}{A} \right)^{2-s} \int_{J_2} \frac{dx }{ |x|^{2} }\le C\left( \frac{2}{A} \right)^{2-s}.
\end{split}
\end{align}
In the second line, we used $1 + (x-A)^2 \ge (A/2)$ to bound the denominator in the first integral. 
In the third line, we used that $|x| \le |x-A|$ for $x\in I_2$. The fourth line follows from using $|x-A| \ge A/2$ for $x \in I_2$. 

Using our bounds for $I_1$, $I_2$, and $I_3$ yields 
\be\label{e:secondhorn}
I(A) = I_1(A) + I_2(A) + I_3(A)  \le   C A^{-(2-s)}.
\ee 
Then combining \eqref{e:secondhorn} and \eqref{e:firsthorn} completes the proof. 
\end{proof}

\begin{proof}[Proof of \Cref{l:integrallowerbound}]
We first consider that case where $|A| \le B + 2$. Elementary reasoning shows that there exists a constant $C(B)>0$ such that for every $x \in \R$, 
\begin{equation}
\sup_{|A| \le B+2} \big(1 + (x-A)^2\big)  \le  C ( 1 + x^2). 
\end{equation}
Then for $|A| \le B + 2$, we have by inserting this bound into the integrand in the definition of $I^\circ$ that 
\begin{equation}\label{e:febintlower1}
I^\circ(A,B) \ge c \int_J \frac{dx }{(1 + x^2) |x|^{2-s} } \ge c,
\end{equation}
for some constant $c(B)>0$, where the $B$-dependence enters through the last inequality. 

 We now consider the case $|A| > B + 2$, and suppose without loss of generality that $A>0$. 
  Then
\begin{align}
\int_J \frac{dx }{(1 + (x-A)^2) |x|^{2-s} }
&\ge 
\int_{A-1}^{A} \frac{dx }{(1 + (x-A)^2) |x|^{2-s} }\\
&\ge |A|^{-2 +s } \int_{A-1}^{A} \frac{dx }{(1 + (x-A)^2) }\\
&= |A|^{-2 +s } \int_{-1}^0 \frac{dx }{(1 + x^2) }\\
&\ge c |A|^{-2+s}.\label{e:lowerb1}
\end{align}
The first inequality follows from restricting the region of integration (since the integrand is positive), the second inequality follows from $|x| \ge A$, the third line follows from a change of variables, and the last line follows since the integrand in the third line is bounded below by $1/2$. 
\end{proof}

\subsection{Proof of Density Estimates} \label{s:densityproof}
\begin{proof}[Proof of \Cref{l:rhoebound}]
We begin with the upper bound in \eqref{e:pEquadratic}. 
From \eqref{e:bapsta2} and \eqref{e:bapsta7}, there exists a constant $C(E,K)>0$ such that 
\be\label{e:bapsta440}
p_{E}(x) \le \frac{C}{1 + |x|^2}. 
\ee 
Next, we consider the lower bound. From \eqref{e:pelow}, there exists a constant $c>0$ such that 
\begin{equation}\label{e:crudepelowerproof}
\inf_{|y| \le 2 \L} p_E(y) > c.
\end{equation}
Further, recall from \eqref{e:loweraux1} that 
\be\label{e:loweraux10}
p_E(x) = \frac{1}{x^2} \int_\R
p^{(K)}_E (y) \rho\left( \frac{1}{x} + E + y\right) \,dy .
\ee 
Since $\rho$ is $\L$-regular, we see by \Cref{d:Lregular} that
\be\label{e:rholowerproof2}
\inf_{|x| \ge 2 \L } \inf_{|y +E | \le (2 \L)^{-1}}  \rho\left( \frac{1}{x} + y  + E \right) \ge \L^{-1}. 
\ee 
Additionally, \eqref{e:pelow} implies that there exists $c(E,K)>0$ such that 
\be\label{e:plowerproof2}
\inf_{y \in [-E - (2\L)^{-1} , - E + (2\L)^{-1} ]} p^{(K)}_E (y) > c.
\ee 
Inserting these expressions into the integral representation in \eqref{e:loweraux10} and bounding this integral below by considering just $y\in [- E - (2\L)^{-1} , -E + (2\L)^{-1} ]$, we find for $|x| \ge 2 \L$ that 
\begin{equation}\label{e:asbeforeapp}
p_E(x) \ge  \frac{1}{x^2} \int_{-E - (2\L)^{-1}}^{-E + (2\L)^{-1}} c \, dx \ge \frac{c}{x^2},
\end{equation}
where we decrease the value of $c$ in the final inequality. Combining this estimate with \eqref{e:crudepelowerproof}, we obtain
\be
p_E (x) \ge \frac{c}{1+x^2},
\ee
which completes the proof of the lower bound in \eqref{e:pEquadratic}. Further, noting that this lower bound is a (scale of) standard Cauchy density, and convolutions of Cauchy densities are again Cauchy (and using the definition of $p^{(M)}_E$), we obtain that for every $M \in \Zplus$, there exists $c(M)>0$ such that 
\be\label{e:pMlower}
p^{(M)}_E (x) \ge \frac{c}{1+x^2}.
\ee

 We now consider the upper bound in  \eqref{e:rhoequadratic}. Using \eqref{e:bananav3} (to represent $\rho_E$ as an integral), \eqref{e:bapsta440}, and the assumption that $\rho$ is $\L$-regular, we have 
\begin{align}
\begin{split}
\rho_E(x) &= \int_{-\infty}^\infty \label{e:banana}
p^{(K-1)}_{E}(y) \rho(x + y +E) \, dy \\
&\le C  \int_{-\infty}^\infty 
 \frac{1 }{1 +  |y|^2} \cdot \frac{1}{1 + | x + y  +E |^2 } \,dy \\
&\le \frac{C}{1 +  | x +E |^{2}}.
\end{split}
\end{align}
In the last line, we used \Cref{l:32}. This completes the proof of the upper bound for $\rho_E$.

We now turn to the lower bound in \eqref{e:rhoequadratic}. 
As in \eqref{e:rholowerproof2}, we find
\be\label{e:rholowerproof}
 \inf_{|y +x +  E | \le (2 \L)^{-1}}  \rho\left( x +  y  + E\right) \ge \L^{-1}. 
\ee 
Inserting this bound and \eqref{e:pMlower} into the integral representation in \eqref{e:banana} and bounding the integral below by considering just $y\in [- E - x - (2\L)^{-1} , -E - x + (2\L)^{-1} ]$, we find that
\begin{equation}\label{e:asbeforeapp2}
\rho_E(x) \ge c  \int_{-E -x  - (2\L)^{-1}}^{-E - x + (2\L)^{-1}} \frac{ dy}{1 + y^2}  \ge \frac{c}{1+|x+E|^2},
\end{equation}
where we decrease the value of $c$ in the final inequality.  This completes the proof of the lower bound in \eqref{e:rhoequadratic}.

Finally, the bound $\| \rho_E \|_\infty \le \| \rho \|_\infty$ in \eqref{e:rhoEupper} is an immediate consequence of \eqref{e:banana} and the fact that $p^{(K-1)}_{E}$  integrates to one.
\end{proof}

\printbibliography
\end{document}